\documentclass[11pt]{amsart}

\usepackage{amsthm, amsfonts, amssymb,graphicx, mathrsfs}
\usepackage[dvipsnames]{xcolor}
\usepackage{tikz-cd}
\usepackage{mathtools}
\usepackage{pinlabel}
\usepackage{bbm}
\usepackage[pdftex,%
  final,%
  colorlinks=true,%
  linkcolor=NavyBlue,%
  citecolor=NavyBlue,%
  filecolor=NavyBlue,%
  menucolor=NavyBlue,%
  urlcolor=NavyBlue,%
  bookmarks=true,%
  bookmarksdepth=3,%
  bookmarksnumbered=true,%
  bookmarksopen=true,%
  bookmarksopenlevel=2,%
]{hyperref}

\graphicspath{{Figures/}}

\newcommand{\dfn}[1]{{\textbf{#1}}}

\numberwithin{equation}{section}

\newcommand{\leg}{\ensuremath{\Lambda}}
\newcommand{\nleg}[1]{\ensuremath{\leg^{#1}}}
\newcommand{\stcleg}{\ensuremath{\widehat{\leg}}}
\newcommand{\tcleg}{\ensuremath{\overline{\leg}}}
\newcommand{\otcleg}{\ensuremath{\underline{\leg}}}
\newcommand{\df}{\ensuremath{\partial}}
\newcommand{\alg}{\ensuremath{\mathcal{A}}}

\newcommand{\aug}{\ensuremath{\varepsilon}}
\newcommand{\Aug}{\ensuremath{\mathcal{A}\mathsf{ug}}}

\newcommand{\obj}{\ensuremath{\mathrm{ob}}}
\newcommand{\base}{\ensuremath{\mathcal{T}}}
\newcommand{\reeb}{\ensuremath{\mathcal{R}}}
\newcommand{\gens}
{\ensuremath{\mathcal{S}}}
\DeclareMathOperator{\cone}{Cone}
\DeclareMathOperator{\cocone}{Cocone}
\DeclareMathOperator{\cyl}{Cyl}
\DeclareMathOperator{\cocyl}{Cocyl}
\DeclareMathOperator{\Hom}{Hom}
\DeclareMathOperator{\CC}{\mathrm{CC}}
\DeclareMathOperator{\CH}{\mathrm{CH}}
\DeclareMathOperator{\HH}{\mathrm{HH}}

\newcommand{\rr}{\ensuremath{\mathbb{R}}}
\newcommand{\zz}{\ensuremath{\mathbb{Z}}}
\newcommand{\qq}{\ensuremath{\mathbb{Q}}}

\newcommand{\ff}{\ensuremath{\mathbb{F}}}

\newcommand{\bbid}{\ensuremath{\mathbbm{1}}}

\newcommand{\mca}{\ensuremath{\mathcal{A}}}
\newcommand{\mcb}{\ensuremath{\mathcal{B}}}
\newcommand{\mcc}{\ensuremath{\mathcal{C}}}
\newcommand{\mcd}{\ensuremath{\mathcal{D}}}
\newcommand{\mcm}{\ensuremath{\mathcal{M}}}
\newcommand{\mcn}{\ensuremath{\mathcal{N}}}
\newcommand{\mcq}{\ensuremath{\mathcal{Q}}}
\newcommand{\mcs}{\ensuremath{\mathcal{S}}}
\newcommand{\mcl}{\ensuremath{\mathcal{L}}}
\newcommand{\mcp}{\ensuremath{\mathcal{P}}}
\newcommand{\mct}{\ensuremath{\mathcal{T}}}

\newcommand{\mbw}{\ensuremath{\mathbf{w}}}
\newcommand{\mbx}{\ensuremath{\mathbf{x}}}
\newcommand{\mby}{\ensuremath{\mathbf{y}}}
\newcommand{\mbz}{\ensuremath{\mathbf{z}}}
\newcommand{\mba}{\ensuremath{\mathbf{a}}}
\newcommand{\mbb}{\ensuremath{\mathbf{b}}}

\newcommand{\ul}[1]{\underline{#1}}
\newcommand{\cev}[1]{\reflectbox{\ensuremath{\vec{\reflectbox{\ensuremath{#1}}}}}}

\theoremstyle{plain}
\newtheorem{thm}{Theorem}[section]

\newtheorem{lem}[thm]{Lemma}

\newtheorem{prop}[thm]{Proposition}

\newtheorem{claim}[thm]{Claim}

\theoremstyle{definition}
\newtheorem{defn}[thm]{Definition}

\theoremstyle{remark}
\newtheorem{rem}[thm]{Remark}
\newtheorem{ex}[thm]{Example}

\begin{document}

\title[Weak Relative CY Structures for LCH]{Weak Relative Calabi-Yau Structures for Legendrian Contact Homology}

\author[J. Ma]{Jiajie Ma} \address{Duke University,
Durham, NC 27708} \email{jason.ma@duke.edu}

\author[J. Sabloff]{Joshua M. Sabloff} \address{Haverford College,
Haverford, PA 19041} \email{jsabloff@haverford.edu} 

\keywords{Legendrian knot, Legendrian Contact Homology, Augmention category, Calabi-Yau structure}
\subjclass[2020]{53D42; 53D37, 57K10, 57K33}

\begin{abstract}
    Legendrian Contact Homology (LCH) and its augmentations are important invariants of Legendrian submanifolds, and for Legendrian knots in the standard contact 3-space in particular.  We increase understanding of the algebraic structure of LCH by generalizing the duality isomorphism and long exact sequence for linearized LCH for Legendrian knots to a weak relative Calabi-Yau structure for $A_\infty$ bimodules over the positive augmentation category.
\end{abstract}

\date{\today}

\maketitle

\setcounter{tocdepth}{1}
\tableofcontents

\section{Introduction}
\label{sec:intro}

\subsection{Context and Goals}
\label{ssec:intro-context}

Invariants built from $J$-holomorphic curves play a central role in both symplectic and smooth topology.  In the study of Legendrian knots, the fundamental $J$-holomorphic invariant is the Legendrian Contact Homology (LCH), which is derived from the Chekanov-Eliashberg differential graded algebra (DGA).  First defined almost thirty years ago \cite{chv, yasha:icm, egh}, LCH has yielded essential insights into the topology of the space of Legendrian submanifolds and Lagrangian cobordisms between them. The goal of this paper is to better understand the structure of LCH by extending a type of duality on linearized LCH \cite{high-d-duality, duality} to the non-linear setting of  augmentation categories of a Legendrian knot.

Structural results are particularly important for understanding the nature and applications of LCH.  Among the first of these  results was the existence of a Poincar\'e-Lefschetz-like duality and a fundamental class for the LCH, linearized with respect to an augmentation \cite{high-d-duality,duality}.  This duality eased computations (see, for example, \cite{casey-henry, melvin-shrestha}), connected the linearized LCH to the Poincar\'e-Lefschetz duality of Lagrangian fillings \cite{c-dr-g-g-cobordism, rizell:lifting, ekholm:rsft, ehk:leg-knot-lagr-cob}, provided a framework for capacities that measure the length \cite{josh-lisa:cob-length} and relative Gromov width \cite{josh-lisa:rel-gr-width} of Lagrangian cobordisms, and even inspired duality results for related invariants such as those based on generating families \cite{josh-lisa:obstr} and microlocal sheaves \cite{li:reeb-sheaf}. 
    
The linearized LCH of a Legendrian knot $\leg$ ignores a wealth of non-linear information contained in the differential of the full DGA.  In order to reintroduce that non-linear information, it has proven useful to use $A_\infty$ structures.  The starting point was to define  $A_\infty$ algebras for a fixed augmentation \cite{products}; this idea was then generalized to the $A_\infty$ categories $\Aug_-(\leg)$ \cite{bc:bilinear} and $\Aug_+(\leg)$ \cite{nrssz:aug-sheaf}, whose objects are augmentations and whose structure maps capture the (dual of) the LCH differential.  The augmentation categories are designed to connect the LCH framework to wrapped Fukaya categories and to sheaf invariants \cite{nrssz:aug-sheaf}. As frameworks for understanding LCH have grown more sophisticated, linearized versions of duality have followed.  In \cite{products}, the original linearized duality was interpreted in terms of cup and cap products. In \cite{bc:bilinear}, the duality long exact sequence of \cite{high-d-duality} was extended and reinterpreted in terms of ``bi-linearized'' LCH.  Finally, in \cite{nrssz:aug-sheaf}, the duality long exact sequence was recast in terms of a relationship between the (linear) morphism spaces of $\Aug_-(\leg)$ and of $\Aug_+(\leg)$.

The next step in this development, embodied in this paper, is to incorporate non-linear information into the duality structure, i.e.\ to upgrade the duality relationship between the morphism spaces of $\Aug_\pm(\leg)$ to a relationship on the full $A_\infty$ structure.  The main difficulty is in figuring out the proper algebraic setting for this upgrade.  We frame our results using Brav and Dyckerhoff's relative Calabi-Yau structures \cite{bd:rel-cy}.  Calabi-Yau (CY) structures on $A_\infty$ algebras first appeared in Kontsevich and Soibelman's work \cite{kontsevich-soibelman}. Since then, they have arisen naturally in the geometry of CY manifolds and in various formulations of mirror symmetry; see, for example, \cite{ganatra:thesis, ganatra:cyclic, sheridan:fano} and, in particular, Seidel's notion of a boundary $A_\infty$ algebra \cite{seidel:lefschetz1} that motivated earlier work on this paper.  See \cite{kl:sheaf-CY} for a more comprehensive survey of the recent wave of discoveries of CY and relative CY structures.

\subsection{Results}
\label{ssec:results}

The main theorem of this paper is a generalization of the linearized duality to a weak right relative CY structure. The precise definition of a relative Calabi-Yau structure appears in Section~\ref{sec:wrcy}.

To state the theorem, we need an additional $A_\infty$ category beyond the augmentation categories $\Aug_\pm(\leg)$.  The circle category $\mcc(\leg)$ is based on the Morse theory of the Legendrian $\leg$ and comes equipped with a projection functor $\pi_\mcc$  from $\Aug_+(\leg)$. From these categories, we will derive three $\Aug_+(\leg)$-bimodules:  $\mcm_+$ is the diagonal bimodule of $\Aug_+(\leg)$, $\mcm_-$ is a submodule of $\mcm_+$ related to the diagonal bimodule of $\Aug_-(\leg)$, and $\mcn$ is the pullback of the diagonal bimodule of $\mcc$ over the projection $\pi_\mcc$. Precise definitions of all of these categories and bimodules will be presented in Section~\ref{sec:aug-cat}. For simplicity, we work with coefficients in $\ff_2 = \zz/2\zz$, though there should be no mathematical obstruction to lifting the coefficients to $\qq$ or even to any commutative ring. 

\begin{thm}
\label{thm:main}
    Let $\leg$ be a Legendrian knot in the standard contact $\rr^3$. For the augmentation category $\Aug_+(\leg)$ over $\ff_2$, the projection functor $\pi_\mcc: \Aug_+(\leg) \to \mcc(\leg)$ admits a weak right relative Calabi-Yau structure of dimension $2$.
\end{thm}

Again, the full definition of a weak right relative CY structure will appear below, but briefly, the theorem tells us that there is an isomorphism of exact triangles of $\Aug_+(\leg)$-bimodules of the following form in the derived category:

\begin{equation} \label{eq:wrcy}
\begin{tikzcd}
	\cdots \arrow[r] & \mcm_+[1] \arrow[r,"\pi_{\mcn}"] \arrow[d] & \mcn[1] \arrow[r, "I_\mcn"] \arrow[d] & \mcm_-[2] \arrow[d] \arrow[r] & \cdots \\
	\cdots \arrow[r] &\mcm_-^\vee[-1] \arrow[r, "I^\vee_\mcn"] & \mcn^\vee \arrow[r,"\pi^\vee_{\mcn}"] & \mcm_+^\vee \arrow[r] & \cdots
\end{tikzcd}
\end{equation}

\begin{rem} \label{rem:reduce1}
By reducing Diagram \eqref{eq:wrcy} to the linear level (i.e.\ to the level of morphisms), using the first vertical quasi-isomorphism to replace $\Hom_+(\aug, \aug)$ with $\Hom_-(\aug,\aug)^\vee[-2]$, and translating between the language of \cite[\S5]{nrssz:aug-sheaf} and \cite{high-d-duality}, we recover the original duality long exact sequence in \cite{high-d-duality} from the top line of the diagram.  See Section~\ref{ssec:wrcy-linear} for more details, along with a discussion of how to generalize the linear fundamental class.
\end{rem}

We make several comments on the relationship between our main theorem and other past and recent work.  First and foremost, we note that Chen \cite{chen:lsft-dual} has proven a closely related result.  Chen constructs $\mcm_-$, proves that the top line in \eqref{eq:wrcy} is an exact sequence of $\Aug_+$-bimodules, and shows that the first and third vertical maps in \eqref{eq:wrcy} are $A_\infty$ quasi-isomorphisms (cf.\ Proposition~\ref{prop:eta-quasi-iso}, below). Chen's results hold over an arbitrary commutative ring.  Chen uses a different framework for the proof, namely Ng's Legendrian Symplectic Field Theory \cite{lenny:lsft} rather than the close analysis of the DGAs of $n$-copies of the Legendrian used in this paper.  Both groups of authors view their results as independent, and virtually simultaneous, discoveries.  

Other recent related work includes L\'egout's generalization of the linear duality map to a CY structure on the full Chekanov-Eliashberg DGA for horizontally displaceable Legendrian spheres \cite{legout:cy}.  Dimitrouglou Rizell and L\'egout have a further generalization of the duality long exact sequence in \cite{high-d-duality} to a weak relative CY structure for general dg-bimodules using LCH over chains in the based loop space of $\leg$; see \cite{drl:dg-duality} for a preliminary report. Relatedly, Gorsky and Haiden \cite{gh:counting-cy} constructed a weak relative CY structure on a dg-category quasi-equivalent to $\Aug_+(\leg)$ by gluing together ``2CY structures'' defined on simple Legendrian tangles.  Asplund \cite{asplund:singular-knot}  constructs a \emph{strong} relative smooth CY structure involving the full Chekanov-Eliashberg DGA for a class of higher-dimensional singular links of Legendrian unknots by using a relation to relative Ginzburg algebras. Finally, Kuo and Li construct a \emph{strong} relative CY structure  for microlocal sheaf invariants \cite{kl:sheaf-CY}; this structure was predicted in  \cite[Remark 1.10]{li:reeb-sheaf}. 

Next, the main theorem should generalize to higher-dimensional Legendrian submanifolds in the standard contact $\rr^{2n+1}$, or any ``horizontally displaceable'' Legendrian in a $1$-jet space. As noted in \cite[Remark 1.4]{nrssz:aug-sheaf}, however, in higher dimensions there are difficulties with ensuring that the augmentation category is well-defined with respect to choices of perturbation of the $n$-copy. See \cite{liu:high-d-aug-cat} for the necessary machinery. A generalization of \cite[Theorem 3.6]{high-d-duality} to multiple mixed punctures would also be necessary.  Alternatively, the aforementioned work of Dimitroglou Rizell and L\'egout \cite{drl:dg-duality} paves the way for an approach to duality in higher dimensions using DGAs with coefficients in chains on the based loop space of $\leg$ rather than using augmentation categories and $n$-copies.

Finally, there is evidence that the weak relative CY structure on the full augmentation category in the main theorem generalizes to a strong relative CY structure.  The evidence comes from the partial result stated in the main theorem, from observations about symmetries of the inverse to the leftmost quasi-isomorphism in \eqref{eq:wrcy}, from Asplund's aforementioned work \cite{asplund:singular-knot}, from the existence of such a strong structure in the closely related microlocal sheaf invariants \cite{kl:sheaf-CY}, and from computations on examples such as those in \cite{products}.

\subsection{Plan of the Paper}
\label{ssec:plan}

We review the construction of Legendrian Contact Homology in Section~\ref{sec:lch-intro}, with an eye towards reviewing older duality results and setting notation that will be used in future sections.  In Section~\ref{sec:a-infty}, we set the language of $A_\infty$ categories and bimodules.  We complete the necessary algebraic background in Section~\ref{sec:wrcy} by defining weak relative CY structures and proving a key algebraic lemma that lets us construct such a structure from even weaker information.

We begin to construct the geometric framework for the proof of Theorem~\ref{thm:main} with a careful discussion of pushoffs of Legendrian knots, especially the Reeb chords and immersed disks in those pushoffs, in Section~\ref{sec:copies}.  In Section~\ref{sec:aug-cat}, we use this geometric framework to review the definition of the augmentation categories $\Aug_\pm(\leg)$ from \cite{nrssz:aug-sheaf}, and then proceed to define the $\Aug_+(\leg)$-bimodules that appear in the statement of the main theorem. We prove that the first and last vertical maps in Diagram~\eqref{eq:wrcy} are quasi-isomorphisms in Proposition~\ref{prop:eta-quasi-iso}.  Finally, in Section~\ref{sec:main-pf}, we prove the homotopy commutativity of Diagram~\eqref{eq:wrcy}, finish the proof of Theorem~\ref{thm:main}, and discuss the relationship between Theorem~\ref{thm:main} and the linear duality of \cite{high-d-duality,duality}.

\subsection*{Acknowledgements}

We thank Zhenyi Chen for the open dialogue and sharing of ideas as both teams worked to complete their approaches to weak relative CY structures for augmentation categories.  We also thank Johan Asplund, Georgios Dimitroglou Rizell, Sheel Ganatra, Christopher Kuo, No\'emie Legout, Wenyuan Li, Aaron Lowe, Lenny Ng, Paul Seidel, and Yiyang Xu for stimulating conversations about the material in this paper.  We are also grateful to the participants of the Philadelphia Area Contact / Topology seminar for their early feedback on this work. JM was partially supported by NSF grant DMS-2003404 and JMS was partially supported by NSF grant DMS-1406093 during the research leading to this paper.  Part of the research was done while JMS was hosted by the Institute for Advanced Study; in particular, this paper is partly based on work supported by the Institute for Advanced Study.

\subsection*{AI Declaration}

No AI was used in the production of this paper, neither in the research itself nor in the writing.

\section{DGAs and Legendrian Contact Homology}
\label{sec:lch-intro}

The goal of this section is to briefly review the foundational ideas of Legendrian Contact Homology (LCH) for Legendrian knots in $\rr^3$: the Chekanov-Eliashberg algebra in Section~\ref{ssec:ce-dga}, the idea of an augmentation in Section~\ref{ssec:aug}, duality for linearized Legendrian Contact Homology in Section~\ref{ssec:old-duality}, and Mishachev's link grading in Section~\ref{ssec:dga-link}.  These sections serve mostly to set notation and perspective. For more depth for all of this material, see the recent survey \cite{en:lch-survey} or the original papers \cite{chv, yasha:icm, ens}; we follow the notation developed in \cite{nrssz:aug-sheaf} as much as reasonable.

\subsection{The Chekanov-Eliashberg Algebra}
\label{ssec:ce-dga}

Our first task is to sketch the definition of the Chekanov-Eliashberg differential graded algebra (DGA) of an oriented Legendrian link $\leg$. We emphasize that we only define the DGA over the field $\ff_2$; see the references above for more general coefficients. First, label the crossings of the Lagrangian projection $\pi_{xy}(\leg)$ --- that is, label the Reeb chords of $\leg$ --- with natural numbers from $1$ to $n$. Further,  choose base points $\circ_1, \ldots, \circ_M$ on $\leg$ so that the complement of the base points is contractible; in particular, there must be at least one base point on every component of $\leg$. Let $\reeb = \{a_1, \ldots, a_n\}$ be labels corresponding to the Reeb chords, let $\base = \{t_1, t_1^{-1}, \ldots, t_M, t_M^{-1}\}$ be labels corresponding to the base points with $\base_\pm = \{t_1^{\pm1}, \ldots, t_M^{\pm1}\}$, and let $\mcs = \reeb \cup \base$.  

Define the algebra $\alg$ to be the unital algebra freely generated by $\mcs$ over the field $\ff_2$, quotiented by the relations $t_it_i^{-1} = 1$ and $t_i^{-1}t_i = 1$; such an algebra is called \dfn{semi-free} on $\mcs$.  The algebra $\alg$ has a grading in $\zz/2r(\leg)\zz$ derived from the Conley-Zehnder index, where $r(\leg)$ is the greatest common divisor of the rotation numbers of the components of $\leg$; see \cite{chv,ens} for details when $\leg$ is a knot and \cite{lenny:computable} in the case of a link.  Denote the grading of a generator $a_i$ by $|a_i|$; we specify that the gradings of $t_i^{\pm1}$ are all zero.

The differential $\df$ on $\alg$ is defined by counting immersed disks in the Lagrangian diagram $\pi_{xy}(\leg)$, which  stand in for pseudo-holomorphic disks in the symplectization $\rr \times \rr^3$; see \cite{ees:high-d-analysis,ens} for more on the latter perspective.  More specifically, begin by decorating the quadrants of the Lagrangian diagram near each crossing with positive or negative \dfn{Reeb signs} as in Figure~\ref{fig:disk-defn}(a).  To find the differential of a generator $a \in \reeb$, let $\mathbf{b} = b_1 \cdots b_n$ be a word in $\alg$.  Define $\Delta(a,\mathbf{b})$ to be the set of immersions $u: D^2 \to \rr^2$ (up to smooth reparametrization) that extend continuously to the closed disk, whose boundaries lie in the Lagrangian projection $\pi_{xy}(\leg)$, and whose boundaries encounter the letters of $\mathbf{b}$ in counterclockwise order.  Additionally, each immersed disk must satisfy the following:

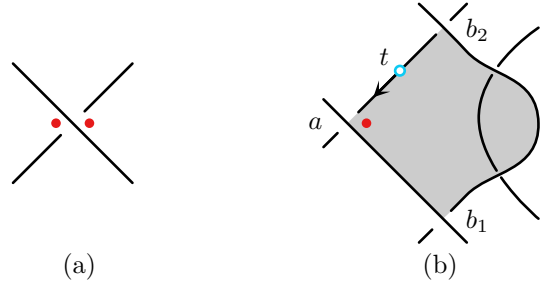
\begin{figure}
\labellist
\small\hair 2pt
 \pinlabel {(a)} [ ] at 30 0
 \pinlabel {(b)} [ ] at 166 0
 \pinlabel {$a$} [ ] at 119 54
 \pinlabel {$b_1$} [ ] at 180 17
 \pinlabel {$b_2$} [ ] at 180 89
 \pinlabel {$t$} [ ] at 145 79
\endlabellist
    \begin{center}
        \includegraphics{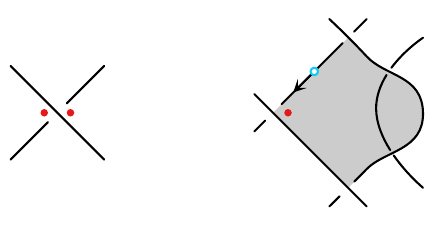}
    \end{center}
    \caption{(a) Reeb signs at a crossing of the Lagrangian diagram of a Legendrian knot $\leg$, with dots denoting positive corners and a lack of decoration denoting a negative corner. (b) A schematic picture of a disk in $\Delta_\leg (a,b_1 b_2 t)$.}
    \label{fig:disk-defn}
\end{figure}

\begin{enumerate}
\item The disk has only \emph{convex} corners at the crossings,
\item The disk must have a corner covering a positive Reeb sign at $a$, 
\item If $b_i \in \reeb$, the disk has a corner covering a negative Reeb sign at $b_i$, and
\item If $b_{i} \in \base_+$ (resp.\ $\base_-$), then the orientation of $\leg$ and the counterclockwise orientation of the boundary of $u$ agree (resp.\ disagree).
\end{enumerate}
See Figure~\ref{fig:disk-defn}(b). The differential is then defined on generators in $\reeb$ by
\begin{equation} \label{eq:df}
 \df a = \sum_{u \in \Delta(a,\mathbf{b})} \mathbf{b}.
 \end{equation}
Additionally, define the differential to vanish on the base point generators $t_i^{\pm1}$. Extend the differential over all of $\alg$ by linearity and the Leibniz rule.

\begin{figure} 
\labellist
\small\hair 2pt
 \pinlabel {$a_1$} [ ] at 65 89
 \pinlabel {$a_2$} [ ] at 135 89
 \pinlabel {$a_3$} [ ] at 96 25
 \pinlabel {$a_4$} [ ] at 146 44
 \pinlabel {$a_5$} [ ] at 188 41
 \pinlabel {$a_6$} [ ] at 202 87
 \pinlabel {$a_7$} [ ] at 202 18
 \pinlabel {$t$} [ ] at 231 105
\endlabellist
\centerline{\includegraphics{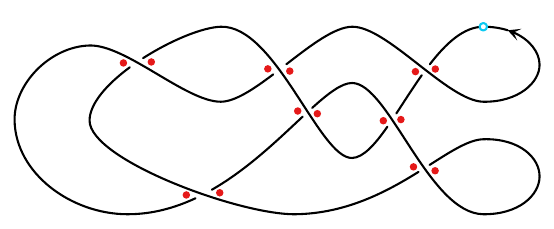}}
    \caption{The Lagrangian diagram of a Legendrian figure eight knot.}
    \label{fig:fig-eight}
\end{figure}
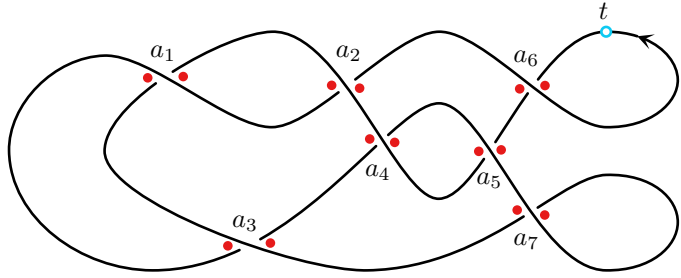

\begin{thm}[\cite{chv}; see also \cite{ens, lenny:computable}]
  The differential $\df$ is well-defined, has degree $-1$, and squares to zero.  The homology $H_*(\alg,\df)$ is invariant under Legendrian isotopy.
\end{thm}

\begin{ex} \label{ex:running-dga}
  The knot in Figure \ref{fig:fig-eight} has seven crossings and a single base point. The generators $t,a_4,a_5$ have grading $0$; $a_1,a_6,a_7$ have grading $1$; and $a_2,a_3$ have grading $-1$. The differential is given by:
  \begin{align*}
      \df a_1 &= \df a_2 = \df a_3 = \df a_5 = \df t^{\pm 1} = 0 \\
      \df a_4 &= a_2 + a_3 + a_2a_1a_3 \\
      \df a_6 &= t + a_5 + a_1a_3a_5 \\
      \df a_7 &= 1 + a_5 + a_5a_2a_1.
  \end{align*}
\end{ex}

The Chekanov-Eliashberg DGA is filtered by the length of the Reeb chords.  This follows from an application of Stokes' Theorem.  Stated precisely, we have:  

\begin{lem}[{\cite[Lemma 6.1]{chv}}] \label{lem:stokes}
Let $\ell(x)$ represent the length of the Reeb chord $x$. If the set $\Delta(a, b_1 \cdots b_n)$ is nonempty, then 
\begin{equation} \label{eq:stokes}
    \ell(a) > \sum_i \ell(b_i).
\end{equation}
In particular, if $\mcb \subset \mca$ is a subalgebra freely generated by all chords below a given height, then it is a sub-DGA.
\end{lem}

\subsection{Augmentations}
\label{ssec:aug}

The Chekanov-Eliashberg DGA is difficult to use directly for practical computations, so we use Chekanov's idea of an \dfn{augmentation} to linearize the differential.  An augmentation is a DGA map $\aug: (\alg, \df) \to (\ff_2,0)$, where $(\ff_2,0)$ is the trivial DGA in grading $0$.  Note that $\aug(t_i^{\pm 1})$ must be $1$. Such augmentations do not always exist \cite{fuchs:augmentations,fuchs-ishk,rulings}, and further, ever more sophisticated ``counts'' of augmentations are interesting Legendrian invariants \cite{hr:ruling-poly, nrss:augm-cat, ns:augm-rulings}. 

An augmentation induces further structures on a semi-free DGA.  We define the $\ff_2$-algebra $\alg^\aug = \alg / (t_i = \aug(t_i))$; since $\df t_i^{\pm 1} = 0$, the differential descends to $\alg^\aug$.  Setting $A$ to be the graded vector space generated by $\reeb$, we write $\alg^\aug$ as a unital tensor algebra:
\[\alg^\aug = TA = \bigoplus_{k \geq 0} A^{\otimes k}.\]
This, in turn, leads us to define the truncated tensor algebra, or \dfn{bar algebra}
\[\bar{T}A = \bigoplus_{k \geq 1} A^{\otimes k}.\]
By changing coordinates using the automorphim $\Phi_\aug: \alg^\aug \to \alg^\aug$ defined by $\Phi_\aug(a_i) = a_i + \aug(a_i)$, we have a new \dfn{twisted differential}
\[ \df^\aug = \Phi_\aug \df \Phi_\aug^{-1}\]
on $\alg^\aug$.  The following proposition is well-known (see \cite{chv}, for example) and straightforward to check.

\begin{prop} \label{prop:augmented-df}
    The differential $\df^\aug$ descends to a differential on the bar algebra $\bar{T}A$, i.e.\ $(\bar{T}A, \df^\aug)$ is itself a DGA.
\end{prop}

The DGA $(\bar{T}A, \df^\aug)$ is the \dfn{augmented DGA} of $(\alg, \df)$ with respect to $\aug$. Observe that, by the Leibniz rule, the differential $\df^\aug$ on $\bar{T}A$ is determined by its components $\df^\aug_k: A \to A^{\otimes k}$.  We call the homology of the chain complex $(A, \df_1^\aug)$ the \dfn{linearized Legendrian contact homology}, which we also denote by $LCH_*(\aug)$.  The linear differential $\df^\aug_1$ comes from disks with one positive corner, one negative corner, and perhaps other ``augmented'' negative corners on which $\aug$ is nonzero; see Figure~\ref{fig:aug-diff}.  

\begin{ex} \label{ex:running-aug-dga}
    Referring back to Example~\ref{ex:running-dga}, it can easily be  checked that there are exactly two augmentations $\aug_1,\aug_2$ over $\ff_2$. Both send the degree $0$ generators $t$ and $a_5$ to $1$, but differ on $a_4$, with $\aug_1(a_4) = 0$ and $\aug_2(a_4) = 1$. Since $a_4$ does not contribute to terms in the differential, the resulting twisted differentials are the same:
    \begin{align*}
        \df^{\aug_i} a_1 &= \df^{\aug_i} a_2 = \df_1^{\aug_i} a_3 = \df^{\aug_i} a_5 = 0 \\
        \df^{\aug_i} a_4 &= a_2 + a_3 + a_2a_1a_3\\
        \df^{\aug_i} a_6 &= a_5 + a_1a_3 + a_1a_3a_5 \\
        \df^{\aug_i} a_7 &= a_5+a_2a_1+a_5a_2a_1,
    \end{align*}
    for $i = 1,2$. The only nontrivial linear terms are $\df^{\aug_i}_1 a_4 = a_2+a_3, \df^{\aug_i}_1 a_6 = a_5,$ and $\df^{\aug_i} a_7 = a_5$. It follows that the linearized Legendrian contact homology induced by either augmentation is generated by $[a_1],[a_2]$, and $[a_6+a_7]$. In particular, $LCH_1(\aug_i) \simeq \ff_2^2$, $LCH_{-1}(\aug_i) \simeq \ff_2$, and the homology in all other degrees is $0$.
\end{ex}

\begin{figure}
\labellist
\small\hair 2pt
 \pinlabel {$a$} [ ] at 8 53
 \pinlabel {$a_1$} [ ] at 54 6
 \pinlabel {$a_2$} [ ] at 90 6
 \pinlabel {$a_n$} [ ] at 135 6
 \pinlabel {$b$} [ ] at 179 54
 \pinlabel {$b_1$} [ ] at 135 102
 \pinlabel {$b_2$} [ ] at 99 102
 \pinlabel {$b_m$} [ ] at 54 102
\endlabellist
\centerline{\includegraphics{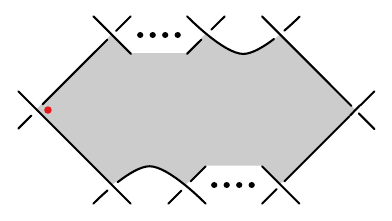}}
\caption{A term $b$ in the linearized differential $\df^\aug_1(a)$ arises from a disk in $\Delta_\leg(a, a_1, \ldots, a_n, b, b_1, \ldots b_m)$ with $\aug(a_i) = 1 = \aug(b_i)$.}
\label{fig:aug-diff}
\end{figure}

It will prove useful to consider the duals of augmented DGAs.  Denote by $A^*$ the dual vector space $\hom(A, \ff_2)$.  The dual space has a dual basis $\{a^*_i\}$ --- that is, $\langle a^*_i, a_j \rangle = \delta_{ij}$ --- whose gradings are defined by $|a^*_i| = |a_i|$.  Extend the pairing between the dual bases to a pairing between $\bar{T}A$ and $\bar{T}A^*$ so that
\[ \langle a^*_{i_k} \cdots a^*_{i_1}, a_{i_1} \cdots a_{i_k} \rangle = 1, \]
with all other pairings going to $0$.  As $\bar{T}A$ is a (graded) algebra, we see that $\bar{T}A^*$ is naturally a (graded) co-algebra called the \dfn{cobar coalgebra}.

To turn $\bar{T}A^*$ into a DG co-algebra (DGCA), we take the adjoint $\delta_\aug$ of $\df^\aug$ with respect to the pairing $\langle, \rangle$:
\[\langle \delta_\aug x^*, y \rangle = \langle x^*, \df^\aug y \rangle.\]
See \cite[\S1]{gj:a-infty} for a concise introduction to DG co-algebras (DGCAs). Using Proposition~\ref{prop:augmented-df}, we obtain the following.

\begin{prop} \label{prop:dual-dgc}
    The map $\delta_\aug: \bar{T}A^* \to \bar{T}A^*$ is a codifferential, i.e.\ it is a coderivation of degree $1$ with $\delta_\aug^2 = 0$, and hence $(\bar{T}A^*, \delta_\aug)$ is a DGCA.
\end{prop}

As with the differential $\df^\aug$, the codifferential $\delta_\aug$ is determined by its components $\delta_\aug^k: (A^*)^{\otimes k} \to A^*$.  Further, observe that $(A^*, \delta_\aug^1)$ forms the cochain complex dual to $(A, \df^\aug_1)$, whose homology we term the \dfn{linearized contact cohomology} and denote by $LCH^*(\aug)$.

\begin{ex} \label{ex:running-dual-dga}
    For either augmentation $\aug_i$ in Example \ref{ex:running-aug-dga}, the codifferential of the induced DGCA is given by:
    \begin{align*}
        \delta^1_{\aug_i} (a_2^*) &= \delta_{\aug_i}^1 (a_3^*) = a_4^* & \delta^3_{\aug_i}(a_3^*, a_1^*, a_2^*) &= a_4^*\\
        \delta^1_{\aug_i} (a_5^*) &= a_6^* + a_7^* & \delta^3_{\aug_i}(a_5^*, a_3^*, a_1^*) &= a_6^* \\
        \delta^2_{\aug_i}(a_1^*,a_2^*) &= a_7^* & \delta^3_{\aug_i}(a_1^*, a_2^*, a_5^*) &= a_7^* \\
        \delta^2_{\aug_i}(a_3^*,a_1^*) &= a_6^* &&
    \end{align*}
    and vanishes otherwise. It follows that the linearized contact cohomology induced by $\aug_i$ is generated by $[a_1^*],[a_6^*],$ and $[a_2^*+a_3^*]$, with $LCH^1(\aug_i) \simeq \ff_2^2$, $LCH^{-1}(\aug_i) \simeq \ff_2$, and the cohomology vanishes in all other degrees.
\end{ex}

\subsection{Duality}
\label{ssec:old-duality}

The linearized Legendrian contact (co)homology groups obey a duality reminiscent of Poincar\'e-Lefschetz duality for manifolds with boundary.  We state the duality exact sequence for Legendrian knots in $\rr^3$, though the theorem works in more general settings.  

\begin{thm}[Linear Duality \cite{high-d-duality, duality}] \label{thm:lin-duality}
	Suppose $\aug$ is an augmentation for a Legendrian knot $\leg$.  There is a long exact sequence
\begin{equation*}
	\begin{tikzcd}[sep=scriptsize]
	\cdots \ar[r] & H^k(\leg) \ar[r,"\sigma"] & LCH^{k}(\aug) \ar[r,"\eta"] & LCH_{-k}(\aug) \ar[r,"\rho"] & H^{k+1}(\leg) \ar[r] & \cdots 
	\end{tikzcd}
\end{equation*}
	When $\leg$ is connected, the map $\sigma$ is injective in degree $1$ and zero in degree $0$.
\end{thm}

The exact sequence in the theorem above parallels the long exact sequence of a pair $(M,\partial M)$ together with Poincar\'e-Lefschetz duality, with $LCH_*(\aug)$ playing the role of the cohomology of the manifold $M$ and $H^*(\leg)$ playing the role of the cohomology of the boundary. Indeed, the Seidel isomorphism \cite{rizell:lifting, ekholm:lagr-cob} makes this parallel precise.

The central aim of this paper is to generalize the (linear) duality in Theorem~\ref{thm:lin-duality} to the (non-linear) $A_\infty$ category setting.

\subsection{Link Gradings}
\label{ssec:dga-link}

Before formulating the $A_\infty$ category framework used to generalize duality, we need to introduce one more technical tool.  When $\leg$ is a link (or, more generally, a decomposition of $\leg$ into sets of connected components), Mishachev \cite{kirill} defined an additional structure on the Chekanov-Eliashberg DGA; see also \cite{lenny:computable}.  The algebraic underpinning of this structure is a pair of maps called a link grading.

\begin{defn}
\label{defn:link-grading}
    Let $(\alg,\df)$ be a semi-free DGA generated by $\mcs = \reeb \sqcup \mct$.  An \dfn{$m$-component link grading} on $(\alg, \df)$ is a pair of maps \[ o,\upsilon: \mcs \to \{1, \ldots, m\}.\]  A word $a_1 \cdots a_k$ in $\alg$ is \dfn{composable} if $o(a_{i+1}) = \upsilon(a_{i})$ for $i = 1, \ldots, k-1$. The link grading maps must satisfy:
    \begin{enumerate}
        \item For any $a \in \reeb$ with $o(a) \neq \upsilon(a)$, each term $a_1 \cdots a_k$ in $\df a$ is composable with $o(a_1) = o(a)$  and $\upsilon(a_k) = \upsilon(a)$.
        \item For any $a \in \reeb$ with $o(a) = \upsilon(a)$, each term $a_1 \cdots a_k$ in $\df a$ is composable with $o(a_1) = o(a)$  and $\upsilon(a_k) = \upsilon(a)$ or is a constant.
        \item For any $j$, $o(t_j) = \upsilon(t_j) = o(t_j^{-1}) = \upsilon(t_j^{-1})$.
    \end{enumerate}
\end{defn}

We write $\mcs^{ij} = (o \times \upsilon)^{-1}(i,j)$; the sets $\reeb^{ij}$ is defined similarly.  A generator in $\mcs^{ii}$ is \dfn{diagonal} and a generator in $\mcs^{ij}$ with $i \neq j$ is \dfn{off-diagonal}.  All elements of $\mct$ are diagonal by definition, and we write $\base^i = (o \times \upsilon)^{-1}(i,i)$.  An augmentation \dfn{respects the link grading} if it vanishes on off-diagonal generators.

The geometry underlying this algebraic definition is that given a Legendrian link $\leg$ whose components have been decomposed into $\leg = \leg_1 \sqcup \cdots \sqcup \leg_m$, the link grading function $o$ (resp.\  $\upsilon$) sends a crossing to the index of the component of its overcrossing (resp. its undercrossing); both functions send a base point to its component.  Thus, $\reeb^{ij}$ can be thought of as the set of Reeb chords that start on $\leg_j$ and end on $\leg_i$.

\section{$A_\infty$ Background}
\label{sec:a-infty}

In order to bring non-linear information from the Chekanov-Eliashberg DGA back to the linearized Legendrian contact cohomology, we use $A_\infty$ structures. In this section, we provide a brief overview of the $A_\infty$ structures that we will need to generalize the duality expressed in Theorem~\ref{thm:lin-duality}: $A_\infty$ algebras, $A_\infty$ categories, and $A_\infty$ bimodules, as well as mapping cones and mapping cylinders. Finally, we review the Hochschild invariants for these $A_\infty$ structures, which serve as the natural framework for expressing the $A_\infty$ duality.

Our notation for $A_\infty$ algebras and categories is designed to be consistent, for the most part, with that of \cite{nrssz:aug-sheaf} which, in turn, follows \cite{keller}; when we turn to $A_\infty$ bimodules and Hochschild invariants, we follow \cite{seidel:fukaya, ganatra:thesis, ganatra:cyclic, mescher}. For simplicity, we work over $\ff_2$ and assume, when necessary, that the morphism spaces are finite dimensional; see \cite{ganatra:thesis, ganatra:cyclic, keller, nrssz:aug-sheaf} for more general constructions.  

\subsection{$A_\infty$ Algebras}
\label{ssec:a-infty-alg}

One may think of an $A_\infty$ algebra as a DGA that is associative up to homotopy, with those homotopies encoded in higher product structures. Later on, we will use $A_\infty$ algebras to encode information dual to that of augmented DGAs.  

\begin{defn}
\label{defn:a-infty-alg}
An \dfn{$A_\infty$ algebra} $A$ is a graded $\ff_2$ vector space $A$ together with a sequence of maps
\begin{equation*}
    m^k:A^{\otimes k} \to A
\end{equation*}
for $k \geq 1$ of degree $2-k$ such that the following quadratic relation holds for each $k$:
\begin{equation} \label{eq:a-infty-rel}
    \sum_{\substack{0 < r \leq k \\ 0 \leq s \leq k-r}} m^{k-s+1}(\bbid^{\otimes k-r} \otimes m^s \otimes \bbid^{\otimes r}) = 0.
\end{equation}
\end{defn}

The first of the $A_\infty$ relations implies that $m^1$ is a codifferential, while the second shows that $m^2$ descends to a product on cohomology.  The third relation is
\begin{align*}   m^2(m^2(\cdot,\cdot),\cdot)+m^2(\cdot,m^2(\cdot,\cdot)) &= m^1(m^3(\cdot,\cdot,\cdot)) + m^3(m^1(\cdot),\cdot,\cdot) \\
    & \quad + m^3(\cdot,m^1(\cdot),\cdot) + m^3(\cdot,\cdot,m^1(\cdot)),
\end{align*}
which shows that the product is associative on cohomology.

To connect with the Chekanov-Eliashberg DGA, we re-characterize $A_\infty$ algebras using the cobar construction as in \cite{gj:a-infty, kadeishvili, stasheff:h-space}. Given an augmented DGA $(\bar{T}A, \df^\aug)$, construct its dual DGCA $(\bar{T}A^*, \delta_\aug)$ as in  Section~\ref{ssec:aug}. Recall that  $\delta_\aug$ is determined by a sequence of maps $\delta_\aug^k: (A^*)^{\otimes k} \to A^*$. Let $s:A \to A[1]$ be the canonical degree $-1$ shift map, and denote $A^*[-1]$ by $A^\vee$; further, if $a$ is a generator of $A$, then we denote its dual in $A^\vee$ by $a^\vee$. We define a sequence of maps $m^k_\aug: (A^\vee)^{\otimes k} \to A^\vee$ by 
\[s \circ m^k_\aug = \delta^k_\aug \circ s^{\otimes k}.\]
This construction has a geometric characterization in the context of the Chekanov-Eliashberg DGA. A disk contributes $a^\vee$ to $m^k_\aug(a_{i_k}^\vee, \ldots, a_{i_1}^\vee)$ if it has a positive corner at $a$, negative corners at $a_{i_1}, \ldots, a_{i_k}$ (in counter-clockwise order), and possibly other negative corners that the augmentation $\aug$ maps to $1$; see Figure~\ref{fig:a-infty-disk}.

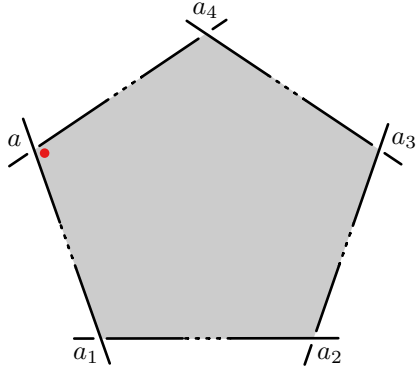
\begin{figure}
\labellist
\small\hair 2pt
 \pinlabel {$a$} [ ] at 12 89
 \pinlabel {$a_1$} [ ] at 39 8
 \pinlabel {$a_2$} [ ] at 131 8
 \pinlabel {$a_3$} [ ] at 159 89
 \pinlabel {$a_4$} [ ] at 84 138
\endlabellist
\centering
\includegraphics{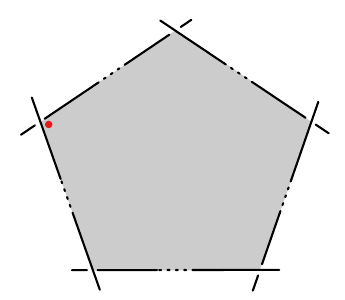}

    \caption{When all of the corners in the dotted parts of the boundary are mapped to $1$ by $\aug$, this disk contributes $a^\vee$ to $m^k_\aug(a_4^\vee, \ldots, a_1^\vee)$.}
    \label{fig:a-infty-disk}
\end{figure}

It is straightforward to check that the defining equation $(\df^\aug)^2=0$ implies the following:

\begin{prop} \label{prop:step1-dga}
    A semi-free DGA $(\alg, \df)$ with augmentation $\aug$ determines an $A_\infty$ algebra $(A^\vee, \{m^k_\aug\})$.
\end{prop}

This construction first appeared in the context of the Chekanov-Eliashberg DGA in \cite{products}, though without the correct gradings.

\begin{ex}\label{ex:running-a-infty-alg}
    The $A_\infty$ algebra induced by either augmentation $\aug_i$ of Example \ref{ex:running-aug-dga} is generated by the duals of the generators of the Chekanov-Eliashberg DGA with grading $|a_2^\vee| = |a_3^\vee| = 0, |a_4^\vee| = |a_5^\vee| = 1$, and $|a_1^\vee| = |a_6^\vee| = |a_7^\vee| = 2$ (note the grading shift compared to Example \ref{ex:running-dual-dga}). The $A_\infty$ operations match those in Example ~\ref{ex:running-dual-dga} with suitable changes in notation.  We see that $H^0(A,m_1) \simeq \ff_2$ and $H^{2}(A,m_1) \simeq \ff_2^2$.
\end{ex}

In the presence of an $m$-component link grading $o \times \upsilon: \mcs \to \{1, \ldots, m\}^2$, we may refine the associated $A_\infty$ algebra structure $(A^\vee, \{m_\aug^k\})$ further.  Following \cite[\S3.2]{nrssz:aug-sheaf}, we write $A^{ij}$ for the subspace of $A$ generated by $\reeb^{ij}$, splitting $A$ as $\bigoplus_{i,j} A^{ij}$ and $A^\vee$  as $\bigoplus_{i,j} A^\vee_{ij}$.  If the augmentation $\aug$ respects the link grading, then the products in the $A_\infty$ algebra $(A^\vee, \{m^k_\aug\})$ split, with the only non-vanishing products being the following:
\begin{equation} \label{eq:a-infty-link}
m^k_\aug: A^\vee_{i_k i_{k+1}} \otimes \cdots \otimes A^\vee_{i_1 i_2} \to A^\vee_{i_1 i_{k+1}}.
\end{equation}

\subsection{$A_\infty$ Categories}
\label{ssec:a-infty-cat}

The next step in our generalization is to move from an $A_\infty$ algebra up to an $A_\infty$ category.  In Section~\ref{sec:aug-cat}, we shall define several $A_\infty$ categories for a Legendrian $\leg$, whose objects will be augmentations and whose morphisms will be generalizations of the linearized Legendrian contact homology complex.  In this section, we set down basic notions and notations for $A_\infty$ categories, following \cite{ganatra:thesis, seidel:fukaya}.

\subsubsection{Definition of an $A_\infty$ Category}
\label{sssec:a-infty-cat}

Just as an algebra may be understood as a category with a single object, an $A_\infty$ algebra may be thought of as an $A_\infty$ category with a single object. 

\begin{defn}
\label{defn:a-infty-cat}
An \dfn{$A_\infty$ category} $\mcc$ (over $\ff_2$) consists of the following data:
\begin{itemize}
    \item a collection of objects $\obj\ \mcc$,
    \item for every pair $X,X'$ of objects, a graded vector space $\Hom_\mcc(X,X')$ over $\ff_2$, and
    \item for any set of $k+1$ objects $X_1,\ldots,X_{k+1}$, structure maps
    \begin{equation*}
        m^k_\mcc:\Hom_\mcc(X_k,X_{k+1}) \otimes \ldots \otimes \Hom_\mcc(X_1,X_2) \to \Hom_\mcc(X_1,X_{k+1})
    \end{equation*}
    of degree $2-k$ that satisfy the relation \eqref{eq:a-infty-rel}, with appropriate changes of notation.
\end{itemize}
\end{defn}

For convenience, we shall henceforth denote by $\Hom_\mcc(X_{k+1}, \ldots, X_1)$ the tensor product $\Hom_\mcc(X_{k},X_{k+1}) \otimes \ldots \otimes \Hom_\mcc(X_1,X_2)$.

As for $A_\infty$ algebra, the $k = 1$ instance of the relations for an $A_\infty$ category $\mcc$ says that $m^1_\mcc$ is a differential on each graded vector space $\Hom_\mcc(X,X')$. Thus, we may define the \dfn{cohomology category} $H(\mcc)$ of $\mcc$ whose objects are the same as $\mcc$ and whose morphisms are given by the cohomology $H^*(\Hom_\mcc(X,X'))$. That the cohomology category is actually a category stems from the $k = 2$ and $k = 3$ instances of the $A_\infty$ relations for $\mcc$, which yield composition of morphisms and associativity for that composition, respectively.

Given an $A_\infty$ category $\mcc$, a natural construction is a \dfn{wide subcategory} $\mcd$ of $\mcc$, which is an $A_\infty$ category that has the same objects at $\mcc$, whose morphisms are subspaces $\Hom_{\mcd}(X,X') \subset \Hom_{\mcc}(X,X')$, and whose structure maps are restrictions of those on $\mcc$, assuming the latter have image in the appropriate $\Hom_{\mcd}$ spaces. A wide subcategory $\mcd$ of $\mcc$ induces a \dfn{quotient category} $\mcc/\mcd$ that has the same objects as $\mcc$,  whose morphisms are defined by $\Hom_{\mcc/\mcd}(X,X') = \Hom_{\mcc}(X,X')/\Hom_{\mcd}(X,X')$, and whose structure maps are defined by $m_\mcc^k$ modulo elements in $\Hom_{\mcd}(X_i,X_{i+1})$.

The theory of $A_\infty$ categories requires a notion of unitality for $A_\infty$ equivalence to make sense. The most relevant notion of unitality for this paper is the following.

\begin{defn}
\label{defn:a-infty-cat-unital}
An $A_\infty$ category $\mcc$ is \dfn{strictly unital} if for any object $X$, there is an element $e_X \in \Hom_\mcc(X,X)$ of grading $0$ such that
\begin{itemize}
    \item $m^1_\mcc(e_X) = 0$,
    \item for any objects $X,X'$ and $a \in \Hom_\mcc(X,X')$,
    \begin{equation*}
        m^2_\mcc(e_{X'},a) = m^2_\mcc(a,e_X) = a,
    \end{equation*}
    \item $m^k_\mcc(\ldots,e_X,\ldots) = 0$ for all $k \ge 3$.
\end{itemize}
\end{defn}

\subsubsection{Definition of an $A_\infty$ Functor}
\label{sssec:a-infty-func}

The definition of an $A_\infty$ functor between $A_\infty$ categories generalizes the standard definition of a functor as follows:

\begin{defn}
\label{defn:a-infty-functor}
An \dfn{$A_\infty$ functor} $F:\mcc \to \mcd$ of degree $d$ consists of the following data:
\begin{itemize}
    \item A map $F: \obj\ \mcc \to \obj\ \mcd$;
    \item For any $k+1$ objects $X_1,\ldots,X_{k+1} \in \obj\ \mcc$, a linear map
    \begin{equation*}
        F^k: \Hom_\mcc(X_{k+1},\ldots,X_1) \to \Hom_{\mcd}(F(X_1),F(X_{k+1})),
    \end{equation*}
    of degree $1 - k + d$ that satisfies the relations
    \begin{multline*} \label{eq:a-infty-morphism}
    \sum_{j,i_1+\cdots+i_j = k} m_\mcd^j(F^{i_j} \otimes \cdots \otimes F^{i_1}) = 
    \sum_{\substack{0 < s \leq k \\ 0 \leq r \leq k-s}} F^{k-s+1}(\bbid^{\otimes k-r-s} \otimes m_\mcc^s \otimes \bbid^{\otimes r}).
 \end{multline*}
\end{itemize}
\end{defn}

The first two relations show that an $A_\infty$ functor $F:\mcc \to \mcd$ induces an ordinary functor $[F]$ between the corresponding cohomology categories. 

Note that if $\mcd$ is a subcategory of $\mcc$, there is a canonical inclusion functor $i: \mcd \to \mcc$ where $i^1$ is the inclusion of linear subspace and $i^k = 0$ for $k > 1$. Similarly, there is a canonical projection functor $\pi:\mcc \to \mcc/\mcd$ where $\pi^1$ is the projection of vector spaces and $\pi^k = 0$ for $k > 1$.  Functors such as these, whose higher order terms vanish, are termed \dfn{na\"ive functors}.

Strict unitality for functors follows a similar pattern as for strictly unital categories. 

\begin{defn}
    An $A_\infty$ functor $F: \mcc \to \mcd$ between strictly unital categories is  \dfn{strictly unital} if for all objects $X \in \obj\ \mcc$, $F^1(e_X) = e_{F(X)}$ and $F^k(\ldots, e_{X_i}, \ldots) = 0$ for any $1 \leq i \leq k$ and $k \geq 2$.
\end{defn}

Again, it is straightforward to check that a strictly unital $A_\infty$ functor induces a unital functor on the cohomology categories. An $A_\infty$ functor between two strictly unital $A_\infty$ categories is an \textbf{$A_\infty$ equivalence} between if the induced functor on cohomology categories is an equivalence of categories in the usual sense, in particular preserving units. $A_\infty$ equivalence is indeed an equivalence relation; see \cite[Theorem 2.9]{seidel:fukaya}

\subsection{$A_\infty$ Bimodules}
\label{ssec:a-infty-bimod}

Theorem~\ref{thm:main} sets up a duality relationship between $A_\infty$ bimodules over augmentation categories.  In this section, we lay out the basic definitions and constructions of $A_\infty$ bimodules that we will need to describe and prove that relationship.  We will not develop the most general theory in this section, as we work in a setting where all of the underlying vector spaces are finite dimensional.

\subsubsection{The Definition of an $A_\infty$ Bimodule}
\label{sssec:defn-bimod-cat}

We begin with the definition of an $A_\infty$ bimodule over two $A_\infty$ categories.

\begin{defn} \label{defn:a-infty-bimod-cat}
Given $A_\infty$ categories $\mcc$ and $\mcd$ over $\ff_2$, an \dfn{$\mcc$-$\mcd$ bimodule} $\mcm$ consists of the following data:
\begin{itemize}
    \item for $X \in \obj\ \mcc$ and $Y \in \obj\ \mcd$, a graded vector space $\mcm(X,Y)$;
    \item for $r,s \ge 0$, objects $X_1,\ldots,X_{r+1} \in \obj\ \mcc$, and $Y_1,\ldots,Y_{s+1} \in \obj\ \mcd$, bimodule structure maps
    \begin{gather*}
        n_\mcm^{r|s}: \Hom_\mcc(X_{r+1},\ldots,X_1) \otimes \mcm(X_1,Y_1) \otimes \Hom_\mcd(Y_1,\ldots,Y_{s+1}) \\ \to \mcm(X_{r+1},Y_{s+1}),
    \end{gather*}
    of degree $1-r-s$, satisfying the relations
    \begin{equation} \label{eq:a-infty-bimod}
\begin{split}
    0&=\sum n_\mcm^{r-i|s-j}(\bbid_\mcc^{\otimes r-i} \otimes n_\mcm^{i|j} \otimes \bbid_\mcd^{\otimes s-j}) \\ 
    &\quad +  \sum n_\mcm^{r-i+1|s}(\bbid_\mcc^{\otimes r-i-j} \otimes m_\mcc^i \otimes \bbid_\mcc^{\otimes j} \otimes \bbid_\mcm \otimes \bbid_\mcd^{\otimes s}) \\
    &\quad + \sum n_\mcm^{r|s-j+1}(\bbid_\mcc^{\otimes r} \otimes \bbid_\mcm \otimes \bbid^{\otimes i} \otimes m_\mcd^j \otimes \bbid^{\otimes s-i-j})
\end{split}
\end{equation} 
\end{itemize}
\end{defn}

In the case when the base categories $\mcc = \mcd$ are the same, we simply refer such a bimodule as a $\mcc$-bimodule. Further, an $A_\infty$ bimodule may come with a \dfn{grading shift}, denoted by $\mcm[n]$, meaning that its graded vector spaces are the shifted spaces $\mcm[n](X,Y) = \mcm(X,Y)[n]$, where $[n]$ means shifting the grading of a vector space down by $n \in \zz$. Since we work over $\ff_2$, the structure maps are unchanged.  

If $\mcc$ and $\mcd$ are strictly unital, then, following \cite[\S3.7]{fooo1} or \cite{seidel:a-infty-subalgebra}, we say that an $\mcc$-$\mcd$ bimodule $\mcm$ is \dfn{strictly unital} if for any objects $X \in \obj\ \mcc$ and $Y \in \obj\ \mcd$, the units $e_X \in \Hom_\mcc(X,X)$ and $e_Y \in \Hom_\mcd(Y,Y)$ satisfy the following relations for any $\mbz \in \mcm(X,Y)$:
\begin{itemize}
    \item $n_\mcm^{1|0}(e_X, \ul{\mbz}) = \mbz = n_\mcm^{0|1}(\ul{\mbz},e_Y)$, and
    \item $n_\mcm^{r|s}(\ldots,e_X,\ldots,\ul{\mbz},\ldots) = 0 =  n_\mcm^{r|s}(\ldots,\ul{\mbz},\ldots,e_Y,\ldots)$ for $r+s \ge 2$.
\end{itemize}

\subsubsection{Fundamental Constructions of $A_\infty$ Bimodules}
\label{sssec:a-infty-bimod-construct}

We review some general constructions of $A_\infty$ bimodules, including submodules, linear duals, pullbacks, diagonal bimodules, and Serre bimodules, which combine to give the language necessary for the relative Calabi-Yau structure promised in Theorem~\ref{thm:main}.

A \dfn{submodule} $\mcn$ of an $A_\infty$ $\mcc$-$\mcd$ bimodule $\mcm$ is simply an assignment, for each $X \in \obj\ \mcc$ and $Y \in \obj\ \mcd$, of a subspace $\mcn(X,Y) \subset \mcm(X,Y)$ that is preserved under the bimodule action.  

The \dfn{linear dual} $\mcm^\vee$ of an $A_\infty$ $\mcc$-$\mcd$ bimodule $\mcm$  is a $\mcd$-$\mcc$ bimodule whose underlying vector spaces are defined by 
\[\mcm^\vee(X,Y) = \hom(\mcm(Y,X),\ff_2).\]
Note the reversal of the order of the actions of $\mcc$ and $\mcd$.  Further, the grading of $\mcm^\vee(X,Y)$ should also be carefully noted:   the base field is assumed to lie in grading zero and the linear functionals in $\hom(\mcm(Y,X),\ff_2)$ are meant to have degree $0$.  Thus, if $\{b_i\}$ is a basis for $\mcm(X,Y)$ with dual basis $\{b_i^\vee\}$, then $|b_i^\vee| = -|b_i|$.  The $A_\infty$ operations on $\mcm^\vee$ are defined by
\begin{equation} \label{eq:linear-dual}
	n_{\mcm^\vee}^{s|r}(b_s, \ldots, b_1,\ul{\boldsymbol{\beta}},a_1, \ldots, a_r)(\mbz) =  \boldsymbol{\beta}(n_{\mcm}^{r|s}(a_1, \ldots, a_r,\ul{\mbz},b_s, \ldots, b_1)).
\end{equation}

We will also need to change the base categories for an $A_\infty$ bimodule by pulling back along a pair of functors.  Let $F:\mcc \to \mcc'$ and $G:\mcd \to \mcd'$ be two functors of $A_\infty$ categories of degree $0$. For a $\mcc'$-$\mcd'$ bimodule $\mcm$, the \dfn{pullback} $(F\otimes G)^*\mcm$ of $\mcm$ along $F$ and $G$ is an $\mcc$-$\mcd$ bimodule with graded vector spaces $(F \otimes G)^*\mcm(X,Y) = \mcm(F(X),G(Y))$
and structure maps
\begin{multline}\label{eq:pullback}
    n_{(F\otimes G)^*\mcm}^{r|s} = \sum_{\substack{i,p_1+\cdots+p_i \leq r \\j, q_1+\cdots+q_j \leq s}} n_\mcm^{i|j}(F^{p_i} \otimes \cdots \otimes F^{p_1} \otimes \bbid_\mcm \otimes G^{q_1} \otimes \cdots \otimes G^{q_j}).
\end{multline}
If $F = G$, we use the shorthand $F^*\mcm = (F\otimes F)^*\mcm$. It follows directly from the definition that the pullback construction commutes with the linear dual.

Finally, every $A_\infty$ category $\mcc$ induces a  $\mcc$-bimodule $\mcc_\Delta$ --- called the \dfn{diagonal bimodule} of $\mcc$ --- with underlying vector spaces $\mcc_\Delta(X,Y) = \Hom_\mcc(X,Y)$
and structure maps given by $n^{r|s}_{\mcc_\Delta} = m^{r+s+1}_\mcc$. The linear dual of a diagonal bimodule of $\mcc_\Delta$ is also a $\mcc$-$\mcc$ bimodule, called the \dfn{Serre bimodule} of $\mcc$ and denoted by $\mcc^\vee$.

\subsubsection{Morphisms of $A_\infty$ Bimodules}
\label{sssec:a-infty-bimod-cat-morphism}

We finish our recollection of $A_\infty$ bimodule concepts by considering morphisms between $A_\infty$ bimodules. A key step in generalizing duality in \cite{high-d-duality, duality} will be promoting the duality map $\eta$ in Theorem~\ref{thm:lin-duality} to an $A_\infty$ bimodule quasi-isomorphism. To get there, we need to first define a weaker notion.

\begin{defn}
\label{defn:pre-morphism} A \dfn{pre-morphism} of $\mcc$-$\mcd$ bimodules $f:\mcm \to \mcn$ of degree $k$ consists of a collection of linear maps
\begin{gather*}
    f^{r|s}: \Hom_\mcc(X_{r+1},\ldots,X_1)
    \otimes \mcm(X_1,Y_1) \otimes \Hom_\mcd(Y_1,\ldots,Y_{s+1}) \\ \to \mcn(X_{r+1},Y_{s+1}),
\end{gather*}
of degree $k-r-s$ for all $r,s \ge 0$, and objects $X_1,\ldots,X_{r+1} \in \obj\ \mcc$ and $Y_1,\ldots,Y_{s+1} \in \obj\ \mcd$.
\end{defn}

For two pre-morphisms $f:\mcm \to \mcn$ and $g:\mcn \to \mcp$ of degrees $k$ and $l$, their \dfn{composition} $g f:\mcm_0 \to \mcm_2$ is a pre-morphism of degree $k+l$ consisting of maps
\begin{equation}\label{eqn:composition}
    (gf)^{r|s} := \sum_{i,j} g^{r-i|s-j}(\bbid_\mcc^{\otimes r-i} \otimes f^{i|j} \otimes \bbid_\mcd^{\otimes s-j}).
\end{equation}

We will use three special types of pre-morphisms. First, a bimodule structure map is a degree $1$ pre-morphism. 

To define the second type, let the \dfn{differential} $\delta(f)$ of a pre-morphism $f$ be the pre-morphism of degree $k+1$ consisting of maps given by
\begin{align*}
    (\delta(f))^{r|s} &= \sum_{i,j} n_\mcm^{r-i|s-j} (\bbid_\mcc^{\otimes r-i} \otimes f^{i|j} \otimes \bbid_\mcd^{\otimes s-j}) \\
    &\quad + \sum_{i,j} f^{r-i|s-j}(\bbid_\mcc^{\otimes r-i} \otimes n_\mcm^{i|j} \otimes \bbid_\mcd^{\otimes s-j})\\
    & \quad + \sum_{i,j} f^{r-i+1|s}(\bbid_\mcc^{\otimes j} \otimes m^i_\mcc \otimes \bbid_\mcc^{\otimes r-i-j} \otimes \bbid_\mcm \otimes \bbid_\mcd^{\otimes s}) \\
    & \quad + \sum_{i,j} f^{r|s-j+1}(\bbid_\mcc^{\otimes r} \otimes \bbid_\mcm \otimes \bbid_\mcd^{\otimes j} \otimes m^i_\mcd \otimes \bbid_\mcd^{\otimes s-i-j}),
\end{align*}
for all $r,s \ge 0$ and objects $X_1,\ldots,X_{r+1} \in \obj\ \mcm$, $Y_1,\ldots,Y_{s+1} \in \obj\ \mcn$.

Note that for fixed bimodules $\mcm,\mcn$, the collection $\hom_{\mcc-\mcd}(\mcm,\mcn)$ of pre-morphisms $f:\mcm \to \mcn$ forms a cochain complex under $\delta$. 

\begin{defn}
\label{defn:bimod-morphism}
A pre-morphism of $\mcc$-$\mcd$ bimodules $f:\mcm \to \mcn$ is a \dfn{bimodule morphism} if $\delta(f) = 0$.
\end{defn}

In particular, an $A_\infty$ bimodule morphism is a chain map on the linear level. Conversely, a simple, but useful, example of a bimodule morphism arises from a family of chain maps $f_{X,Y}: \mcm(X,Y) \to \mcn(X,Y)$. In this case, we may define a bimodule morphism $f: \mcm \to \mcn$ by setting $f^{0|0}$ to be $f_{X,Y}$ on $\mcm(X,Y)$ and all higher maps vanish; this is called a \dfn{na\"ive morphism}. In particular, for a $\mcc$-$\mcd$ bimodule $\mcm$, the \dfn{identity morphism} is a na\"ive morphism obtained by extending the identity maps on $\mcm(X,Y)$.

Other useful examples of bimodule morphisms arise from the constructions in Section~\ref{sssec:a-infty-bimod-construct}.  First, for an $A_\infty$ functor $F:\mcc \to \mcd$ of degree $0$, the \dfn{induced bimodule morphism} $f:\mcc_\Delta \to F^*\mcd_\Delta$ of $F$ is a $\mcc$-bimodule morphism of degree $0$ defined by $f^{r|s} = F^{r+s+1}$. In general, if $f: \mcm \to \mcn$ is a $\mcd$-bimodule morphism, its \dfn{pullback morphism} $F^*f:F^*\mcm \to F^*\mcn$ is a $\mcc$-bimodule morphism defined using a formula analogous to Equation~\eqref{eq:pullback}. Further, we may form the \dfn{adjoint} $f^\vee: \mcn^\vee \to \mcm^\vee$ using a formula analogous to Equation~\eqref{eq:linear-dual}.

The last type of pre-morphism we will use generalizes the notion of a chain homotopy.  

\begin{defn}
\label{defn:a-infty-homotopy}
Let $f:\mcm \to \mcn$ and $g:\mcm \to \mcn$ be two $\mcc$-$\mcd$ bimodule morphisms of degree $k$. An \dfn{$A_\infty$ homotopy} between $f$ and $g$ is a pre-morphism $H:\mcm \to \mcn$ of degree $k-1$ such that
\begin{equation}
    f+g = \delta(H).
\end{equation}
In this case, $f$ is homotopic to $g$ and we write $f \sim g$.
\end{defn}

The following lemma, which will be useful in proving Proposition~\ref{prop:alg-gadget}, is straightforward to check.

\begin{lem} 
\label{lem:bimod-homotopy-prop}
    The $A_\infty$ homotopy relation is an equivalence relation.  Further, if $f,f': \mcm \to \mcn$ are homotopic and if $j: \mcn \to \mcp$ and $g: \mcl \to \mcm$ are $A_\infty$ morphisms, then $jf$ and $jf'$ are homotopic, as are $fg$ and $f'g$.
\end{lem}

A key ingredient in the generalization of duality is the invertibility of morphisms. First, a bimodule morphism $f$ is a \dfn{quasi-isomorphism} if $f^{0|0}$ is a chain quasi-isomorphism (i.e.\ $[f^{0|0}]$ is an isomorphism on homology). A bimodule morphism $f:\mcm \to \mcn$ is an \dfn{$A_\infty$ homotopy equivalence} if there is an \dfn{$A_\infty$ homotopy inverse} $g:\mcn \to \mcm$ such that
\begin{equation*}
    fg \sim \bbid_\mcn\ \ \text{and}\ \ gf \sim \bbid_\mcm.
\end{equation*}
By noting that $A_\infty$ homotopy equivalence implies ordinary homotopy equivalence at the linear level, we see that an $A_\infty$ homotopy equivalence is automatically a quasi-isomorphism. Conversely, if $f:\mcm \to \mcn$ is a quasi-isomorphism and all $\mcm(X,Y)$ and $\mcn(X,Y)$ are finite-dimensional as graded vector spaces, then we can explicitly construct a quasi-inverse to $f^{0|0}$ and bootstrap up to the following proposition.  The bootstrapping proof appears in a variety of sources, but the source closest to the our setting of modules over $A_\infty$ categories is \cite[Lemma 8.1]{gps:sectorial-descent}.

\begin{prop}
\label{prop:homotopy-equiv}
If $f:\mcm \to \mcn$ is a quasi-isomorphism of $A_\infty$ bimodules with all $\mcm(X,Y)$ and $\mcn(X,Y)$ finite-dimensional as graded vector spaces, then $f$ is an $A_\infty$ homotopy equivalence.
\end{prop}

This proposition is a key feature of the $A_\infty$ bimodules, and we will use it repeatedly in our work to come.

Finally, suppose $\mcc$ and $\mcd$ are strictly unital, with units $e_X \in \Hom_\mcc(X,X)$ and $e_Y \in \Hom_\mcd(Y,Y)$ for every object $X \in \obj\ \mcc$ and $Y \in \obj\ \mcd$. We say that a bimodule pre-morphism $f:\mcm \to \mcn$ is \dfn{strictly unital} if for all $r+s \ge 1$, $f^{r|s}(\ldots, e_{X},\ldots,\ul{\mbz},\ldots) = 0$ and $f^{r|s}(\ldots,\ul{\mbz},\ldots, e_{Y},\ldots) = 0$. 

\subsection{Mapping Cones and Cylinders}
\label{ssec:mapping-cone-defn}

The proof of Theorem~\ref{thm:lin-duality} in \cite{duality} identifies the linearized LCH of a ``separated $2$-copy'' as the mapping cone of the duality map $\eta$  to deduce the duality exact sequence. The proof of the main theorem in this paper will generalize this structure to an $A_\infty$ mapping cone of an $A_\infty$ bimodule morphism.  Following \cite[\S3s]{seidel:fukaya}, we recall the definition of these objects here, beginning with the mapping cone. 

\begin{defn}
\label{defn:a-infty-mapping-cone}
Given a degree $0$ $\mcc$-$\mcd$ bimodule morphism $f:\mcm \to \mcn$, the \dfn{$A_\infty$ mapping cone} $\cone(f)$ of $f$ is the $\mcc$-$\mcd$ bimodule consisting of the following data:
\begin{itemize}
    \item for $X \in \obj\ \mcc$ and $Y \in \obj\ \mcd$, a graded vector space
    \begin{equation*}
        \cone(f)(X,Y) = \mcm(X,Y)[1] \oplus \mcn(X,Y).
    \end{equation*}
    \item For $r,s \geq 0$, 
    the structure map of degree $1-r-s$ is defined by
    \begin{align*}
    n^{r|s}_{\cone(f)}(\ldots,\ul{(\mbx,\mby)},\ldots) =  \bigl(& n_\mcm^{r|s}(\ldots,\ul{\mbx},\ldots), \\ & n_\mcn^{r|s}(\ldots,\ul{\mby},\ldots) + f^{r|s}(\ldots,\ul{\mbx},\ldots) \bigr).
    \end{align*}
\end{itemize}
\end{defn}

Note that $\mcn$ is a submodule of $\cone(f)$. Further, $\cone(f)$ comes with two degree $0$ bimodule morphisms, the inclusion $i_\mcn:\mcn \to \cone(f)$ and the projection $\pi_\mcm:\cone(f) \to \mcm[1]$, both defined as the na\"ive morphisms extending the corresponding linear map.

The \dfn{mapping cocone} of $f$ is given by $\cocone(f) = \cone(f^\vee)[-1]$. The cocone is also equipped with two degree $0$ bimodule morphisms $i_\mcn^\vee:\cocone(f) \to \mcn^\vee$ and $\pi^\vee_{\mcm}:\mcm^\vee[-1] \to \cocone(f)$ corresponding to the duals of $i_\mcn$ and $\pi_\mcm$ respectively. 

As stated in \cite[\S2]{seidel:a-infty-subalgebra}, a mapping cone structure naturally arises when a bimodule contains a submodule.  In particular, suppose $\mcq$ is a submodule of an $\mcc$-bimodule $\mcm$.  Let $\mcp$ be the quotient $\mcm/\mcq$.  We get a na\"ive short exact sequence of $\mcc$-bimodules
\[
\begin{tikzcd}
    0 \ar[r] & \mcq \ar[r,"i"] & \mcm \ar[r,"\pi"] & \mcp \ar[r] &0
\end{tikzcd}
\]
From this data, we may define a morphism $\varphi: \mcp[-1] \to \mcq$.  For any choice of a na\"ive pre-morphism $\sigma: \mcp \to \mcm$ whose underlying vector space maps split $\pi$,  define $\varphi$ by $(id-\sigma \pi)\circ n_\mcm \circ \sigma$.  

\begin{lem}\label{lem:submodule-to-cone}
    In the setting above, $\varphi:\mcp[-1] \to \mcq$ is a bimodule morphism and $\mcm$ is quasi-isomorphic to $\cone(\varphi)$.
\end{lem}

Finally, we recall the relevant concepts of the mapping cylinder in the $A_\infty$ bimodule setting.

\begin{defn}
\label{defn:a-infty-mapping-cyl}
The \dfn{$A_\infty$ mapping cylinder} $\cyl(f)$ of a $\mcc$-$\mcd$ bimodule morphism $f:\mcm \to \mcn$ is the shifted cone $\cone(\pi_\mcm)[-1]$ of the canonical projection $\pi_\mcm:\cone(\pi) \to \mcm[1]$. The \dfn{$A_\infty$ mapping cocylinder} $\cocyl(f)$ of a $\mcc$-$\mcd$ bimodule morphism $f:\mcm \to \mcn$ is the shifted cocone $\cocone(\pi_M)[1]$.
\end{defn}

A useful fact about these constructions is encapsulated in the following lemma, whose proof follows from the fact that the na\"ive inclusion $I_\mcn: \mcn \to \cyl(f)$ and na\"ive projection $\Pi_\mcn: \cyl(f) \to \mcn$ are homotopy inverses.

\begin{lem} 
\label{lem:mapping-cyl-homotopy}
The mapping cylinder $\cyl(f)$ of a $\mcc$-$\mcd$ bimodule morphism $f:\mcm \to \mcn$ is homotopy equivalent to $\mcn$ and the mapping cocylinder $\cocyl(f)$ is homotopy equivalent to $\mcn^\vee$. 
\end{lem}

\subsection{Hochschild Homology for $A_\infty$ Bimodules}
\label{ssec:HH}

We end our review of $A_\infty$ concepts by recalling the Hochschild invariants for $A_\infty$ bimodules. As analogues of polyvector fields and de Rham differential forms \cite{kontsevich-soibelman}, these complexes provide a natural framework for expressing duality in non-commutative and derived settings. Following the approach and notation of \cite{ganatra:thesis,ganatra:cyclic}, as well as \cite{mescher}, we formulate these invariants using explicit cochain level models based on the $A_\infty$ bar complex, maintaining a cohomological grading throughout. 

\subsubsection{Hochschild Homology}
\label{sssec:HH}

We begin by defining the Hochschild homology of a $\mcc$-bimodule $\mcm$. Throughout, we assume that the base category $\mcc$ is strictly unital.

\begin{defn}[{\cite[Definition 2.27 and 2.28]{ganatra:thesis}}]
\label{defn:HH}
The \dfn{Hochschild cochain complex} of $\mcc$ with coefficients in $\mcm$ is
\begin{equation*}
    \CH^*(\mcc,\mcm) := \bigoplus_{X_{1},\ldots,X_{k+1} \in \obj\ \mcc} \hom_{*-k}(\Hom_\mca(X_{k+1},\ldots,X_1),\mcm(X_{k+1},X_1)),
\end{equation*}
with differential $b^*:\CH^*(\mcc,\mcm) \to \CH^{*+1}(\mcc,\mcm)$ given by
\begin{equation*}
    (b^*\phi)^k := \sum_{i,j} n_\mcm^{i|k-i-j} (\bbid_\mcc^{\otimes i} \otimes \phi^{j} \otimes \bbid_\mcc^{\otimes k-i-j}) + \sum_{i,j} \phi^{k-j+1}(\bbid_\mcc^{\otimes i} \otimes m_\mcc^{j} \otimes \bbid_\mcc^{\otimes k-i-j}).
\end{equation*}
The \dfn{Hochschild chain complex} of $\mcc$ with coefficients in $\mcm$ is the direct sum of all cyclically composable sequences of morphism spaces
\begin{equation*}
    \CH_*(\mcc,\mcm) := \bigoplus_{X_{1},\ldots,X_{k+1} \in \obj\ \mcc} (\mcm(X_{k+1},X_1) \otimes \Hom_\mca(X_{k+1},\ldots,X_1))_{*+k-1},
\end{equation*}
with differential $b:\CH_*(\mcc,\mcm) \to \CH_{*+1}(\mcc,\mcm)$ given by
\begin{align*}
    b(\mbz \otimes a_k \otimes \cdots \otimes a_1) &:= \sum_{i,j} n_\mcm^{i|j} (a_i,\ldots,a_1,\ul{\mbz},a_k,\ldots,a_{k-j+1}) \otimes a_{k-j} \otimes \cdots \otimes a_{i+1} \\
    &+ \sum_{i,j} \mbz \otimes a_k \otimes \cdots \otimes m_\mcc^j (a_{i+j} \otimes \cdots \otimes a_{i+1}) \otimes a_i \otimes \cdots \otimes a_1.
\end{align*}
\end{defn}

Explicitly, a Hochschild cochain is a sequence $(\phi^0,\phi^1,\ldots)$ of multi-linear maps where $\phi^0 \in \mcm(X_1,X_1)$ and $\phi^k \not = 0$ for only finitely many $k$. The notation $b^*$ is justified by the duality between $\CH^*(\mcc,\mcm^\vee)$ and $\CH_*(\mcc,\mcm)$ \cite[Theorem 3.5]{mescher}. Specifically, consider the pairing $\langle -, - \rangle:\CH^*(\mcc,\mcm^\vee) \otimes \CH_*(\mcc,\mcm) \to \ff_2$ given by
\begin{equation}
\label{eq:HH-pairing}
    \langle \phi, \mbz \otimes a_k \otimes \cdots \otimes a_1 \rangle := \phi(a_k,\ldots,a_1)(\mbz). 
\end{equation}
It is immediate to check that $b^*$ is the adjoint of $b$ with respect to \eqref{eq:HH-pairing} and that this pairing descends to cohomology. Thus, the map 
\begin{equation}
\label{eq:hoch-dual}
    \CH^*(\mcc,\mcm^\vee) \xrightarrow{\sim} \CH_*(\mcc,\mcm)^\vee:\quad \phi \mapsto \langle \phi, - \rangle,
\end{equation}
is a quasi-isomorphism by the tensor-Hom adjunction. For brevity, we denote $\CH_*(\mcc) := \CH_*(\mcc,\mcc_\Delta)$ and $\CH^*(\mcc) := \CH^*(\mcc,\mcc^\vee)$. The cohomologies of $\CH_*(\mcc)$ and $\CH^*(\mcc)$ are referred to as the \dfn{Hochschild homology} and \dfn{Hochschild cohomology} of $\mcc$, denoted by $\HH_*(\mcc)$ and $\HH^*(\mcc)$, respectively.

\subsubsection{Functoriality of Hochschild Homology}
\label{sssec:HH-functor}

Hochschild invariants exhibit natural functoriality with respect to $A_\infty$ morphisms. We briefly review two aspects of this functoriality for Hochschild cohomology that are central to our purpose. These are the analogues to the functoriality of Hochschild homology detailed in \cite[\S2.9]{ganatra:thesis}. 

First, an $A_\infty$ bimodule morphism $f:\mcm \to \mcn$ induces a chain map $f^*:\CH^*(\mca,\mcm) \to \CH^*(\mca,\mcn)$ defined by
\begin{equation}
\label{eq:hoch-pushfwd}
    (f^* \phi)^k(a_k,\ldots,a_1) := \sum_{i,j} f^{k-i|j}(a_k,\ldots,a_{i+1},\ul{\phi(a_i,\ldots,a_{j+1})},a_j,\ldots,a_1).
\end{equation}
Moreover, $f^*$ is a quasi-isomorphism whenever $f$ is; see \cite{gj:a-infty, mescher} for a spectral sequence proof of this fact. If $f$ and $g$ are homotopic bimodule morphisms via a homotopy $h$, then $h^*$ is a chain homotopy between $f^*$ and $g^*$.

In fact, Hochschild cochains can be identified with $A_\infty$ bimodule pre-morphisms. Recall that $\hom_{\mcc-\mcc}(\mcc_\Delta,\mcm)$ is the cochain complex of pre-morphisms from $\mcc_\Delta$ to $\mcm$ over $\mcc$, equipped with the differential $\delta$. This is also referred to as the two-pointed complex for Hochschild cohomology in \cite[Definition 2.31]{ganatra:thesis}.  The following proposition allows for comparisons between the two-pointed complex, which allows us to use the $A_\infty$ language developed in previous subsections, with the standard Hochschild homology language.

\begin{prop}[{\cite[Proposition 2.5]{ganatra:thesis}}]
    \label{prop:hoch-a-infty}
    The map $\Psi:\CH^*(\mcc,\mcm) \to \hom_{\mcc-\mcc}(\mcc_\Delta,\mcm)$ defined by
    \begin{multline*}
        \Psi(\phi)(a_r,\ldots,a_1,\ul{z},b_1,\ldots,b_s) := \\
        \sum_{i,j} n_\mcm^{i|s+j+1}(a_r,\ldots,a_{r-i+1},\ul{\phi( a_{r-i},\ldots,a_{j+1})},a_j,\ldots,a_1,z,b_1,\ldots,b_s)
    \end{multline*}
    is a quasi-isomorphism. Moreover, if $f:\mcm \to \mcn$ is a bimodule morphism and $\phi$ is closed, then $f\Psi(\phi)$ is homotopic to $\Psi(f^*\phi)$ in $\hom_{\mcc-\mcc}(\mcc_\Delta,\mcm)$.
\end{prop}

The proof that $\Psi$ is a quasi-isomorphism requires, as above, work with spectral sequences.  Once this has been accomplished, however, the use of $\Psi$ to translate between two different algebraic languages is straightforward.  For example, the chain homotopy $H_f$ asserted in the final sentence is obtained by replacing $n_\mcm^{i|s+j+1}$ with $f^{i|s+j+1}$ in the formula for $\Psi$. That $H_f$ is, indeed, a chain homotopy follows from rearranging the bimodule morphism equation for the composition of $f$ and $\phi$ into 
\begin{equation}
\label{eq:hoch-htpy}
    \Psi(f^*\phi) + f\Psi(\phi) = \delta H_f(\phi) + H_f(b^*\phi).
\end{equation}

\begin{rem}
    The formula in Proposition~\ref{prop:hoch-a-infty} applies $\phi$ to a block of the $a_i$ inputs, whereas Ganatra's formula applies it to a block of the $b_i$ inputs. Both conventions give quasi-isomorphisms. We use the former to agree the cap-with-fundamental-class map in \cite{duality}. The corresponding two choices for that map are chain homotopic; see \cite{products}.
\end{rem}

Second, there is a form of functoriality with respect to morphisms of $A_\infty$ categories, following \cite[\S2.9]{ganatra:thesis}. Suppose $F:\mcc \to \mcd$ is an $A_\infty$ functor. Then, it induces a chain map $F^\sharp:\CH^*(\mcd,\mcd^\vee) \to \CH^*(\mcc,F^*\mcd^\vee)$ defined by
\begin{align}
    (F^\sharp \phi)^k(a_k,\ldots,a_1) &:=\\
     \sum_{i_1+\cdots+i_s = k} &\phi(F^{i_1}(a_k,\ldots,a_{k-i_1+1}),\ldots,F^{i_s}(a_{i_s},\ldots,a_1)). \nonumber
\end{align}
As before, $F^\sharp$ is a quasi-isomorphism if $F$ is.

By combining the two preceding constructions, we establish the functoriality of Hochschild cohomology for $A_\infty$ categories. Specifically, let $f:\mcc_\Delta \to F^*\mcd_\Delta$ be the induced bimodule morphism of $F$. An $A_\infty$ functor $F:\mcc \to \mcd$ induces a chain map $F^*:\CH^*(\mcd) \to \CH^*(\mcc)$ defined by $F^* = (f^\vee)^* \circ F^\sharp$:
\begin{align}
\label{eq:hoch-cat-functor}
    (F^* \phi)^k(a_k,\ldots,a_1) &= \\  \sum_{i,j} (f^\vee)^{k-i|j}&(a_k,\ldots,a_{i+1}, \ul{F^\sharp\phi(a_i,\ldots,a_{j+1})},a_j,\ldots,a_1). \nonumber
\end{align}

If $\phi$ is closed, then $\Psi(F^* \phi) \simeq f^\vee(F^* \Psi(\phi)) f$ in $\hom_{\mcc-\mcc}(\mcc_\Delta,\mcc^\vee)$ with a chain homotopy $H_F(\phi) := H_{f^\vee}(F^\sharp \phi)$, where $H_{f^\vee}$ is defined after Proposition~\ref{prop:hoch-a-infty}. This follows from applying Equation~\eqref{eq:hoch-htpy} for $H_{f^\vee}$ to the Hochschild cochain $F^\sharp \phi$; we also require the identity $\Psi(F^\sharp\phi) = (F^*\Psi(\phi))f$. In particular, we have:
\begin{equation}
\label{eq:hoch-cat-htpy}
    \Psi(F^* \phi) + f^\vee(F^* \Psi(\phi)) f = \delta H_F(\phi) + H_F(b^*\phi).
\end{equation}

Moreover, the homotopy for a composite $F \circ G$ is the sum of the homotopy for $F$ and the transported homotopy for $G$. Specifically, if $g$ is the induced bimodule morphism of $G$, then
\begin{equation}
\label{eq:hoch-cat-htpy-composite}
    H_{FG}(\phi) = H_G(F^*\phi)+g^\vee(G^*H_F(\phi))g.
\end{equation}
Indeed, it follows from the identities $H_{fg}(\phi) = H_f(g^*\phi) + fH_g(\phi)$ and $H_{F^*g}(F^\sharp\phi) = (F^*H_g(\phi))f$ that
\begin{align*}
H_{FG}(\phi) &= H_{g^\vee(G^*f^\vee)}(G^\sharp F^\sharp\phi) \\
&= H_{g^\vee}((G^*f^\vee)^*G^\sharp F^\sharp\phi) + g^\vee H_{G^*f^\vee}(G^\sharp F^\sharp\phi) \\
&= H_{g^\vee}(G^\sharp F^*\phi) + g^\vee(G^*H_{f^\vee}(F^\sharp\phi))g \\
&= H_G(F^*\phi) + g^\vee(G^*H_F(\phi))g.
\end{align*}

\begin{rem}
     With respect to the pairing~\eqref{eq:HH-pairing}, note that $F^*$ acts as the adjoint to the map $F_*:\CH_*(\mcc) \to \CH_*(\mcd)$ defined by
\begin{align*}
    F_*(\mba \otimes a_k \otimes \cdots \otimes a_1) &= \\
    \sum_{i,j,s} &F^{j + 1 + (k-i)}(a_j, \ldots, a_1, \mba, a_k, \ldots, a_{i+1}) \otimes F^{i_1}(a_i, \ldots) \\ &\otimes \cdots \otimes F^{i_s}(\ldots, a_{j+1}).
\end{align*}
\end{rem}

\subsubsection{Relative Hochschild Homology}
\label{sssec:rel-HH}

Adapting the approach of \cite[\S2.1]{bd:rel-cy}, we can now define explicit cochain level models for the relative versions of the Hochschild and cyclic invariants.

\begin{defn}
    \label{defn:rel-hochschild}
    The \dfn{relative Hochschild chain complex} $\CH_*(F)$ of an $A_\infty$ functor $F:\mcc \to \mcd$ is the mapping cone $\cone(F_*)$ in the category of Hochschild  chain complexes. The \dfn{relative Hochschild  cochain complex} $\CH^*(F)$ of $F$ is $\cone(F^*)$ in the category of Hochschild cochain complexes. The \dfn{relative cyclic chain complex} $\CC_*(F)$  and \dfn{relative cyclic cochain complex} $\CC^*(F)$ of a strictly unital functor $F:\mcc \to \mcd$ are defined similarly.
\end{defn}

Here, the mapping cone is defined analogously to Definition~\ref{defn:a-infty-mapping-cone}. Specifically, a relative cochain in $\CH^*(F)$ is a tuple $(\phi,\sigma) \in \CH^*(\mcc)[1] \oplus \CH^*(\mcd)$. Such a pair is closed if $\sigma$ is a cocycle in $\CH^*(\mcc)$ and $F^\sharp \sigma$ is a coboundary in $\CH^*(\mcd)$ bounded by $\phi$. By Proposition~\ref{prop:hoch-a-infty}, this is equivalent to asserting that $\Psi(\sigma)$ is an $A_\infty$ morphism and $\Psi(\phi)$ is an $A_\infty$ null-homotopy for $\Psi(F^* \sigma)$. Similar characterizations apply to the other complexes in the definition. Leveraging the isomorphism~\eqref{eq:hoch-dual}, along with the adjunction between $F_*$ and $F^*$, we obtain the duality
\begin{equation*}
    \CH^*(F) \simeq \CH_*(F)^\vee.
\end{equation*}

\section{Relative Calabi-Yau Structures}
\label{sec:wrcy}

The goal of this section is to introduce the algebraic structures used to generalize duality. As above, we will not develop the most general theory, but rather work over $\ff_2$ and under the assumption that the morphism spaces of the $A_\infty$ categories and bimodules are finite-dimensional. Based on work of Brav and Dyckerhoff \cite{bd:rel-cy} (see also Chen \cite{chen:lsft-dual}), we first define the notions of weak Calabi-Yau (CY) structures in terms of Hochschild cohomology. We then introduce a closely related concept that we term a very weak relative CY structure. Formulated in terms of $A_\infty$ bimodules, this definition aligns more closely with our preliminary setup. Finally, we prove that, under mild assumptions, a very weak relative CY structure gives rise to a weak one; the transition from very weak to weak structures will be the algebraic key to the proof of the main theorem.

\begin{defn}[\cite{bd:rel-cy}]
\label{defn:rel-cy}
Let $\mca$ and $\mcc$ be strictly unital $A_\infty$ categories and let $F: \mca \to \mcc$ be a strictly unital $A_\infty$ functor with induced bimodule morphism $f$. A \dfn{weak right relative Calabi-Yau structure of dimension $n$} on $F:\mca \to \mcc$ is a relative Hochschild cohomology class $([\phi,\sigma]) \in \HH^*(F)[n-1] \simeq \HH_*(F)^\vee[n-1]$ that satisfies the following non-degeneracy property.  Using the universal property of $\cocone(f)$ and $\cone(f)$, we obtain morphisms $\zeta$ and $\zeta'$ that fit together into a coherent diagram in the derived category $D(\mca$-$\mathrm{bimod})$ via the null-homotopy $\phi$ of $f(F^*\Psi(\sigma))f^\vee$:
    \begin{equation}
    \label{eq:weak-right-cy}
 \begin{tikzcd}[column sep=small]
	\cdots \arrow[r] & \mca_\Delta[n-1] \arrow[r,"f"] \arrow[d,"\zeta", dashed] & F^*\mcc_\Delta[n-1] \arrow[r,"i"] \arrow[d,"F^*\Psi(\sigma)"] & \cone(f)[n-1] \arrow[d,"\zeta'",dashed] \arrow[r] & \cdots \\	
	\cdots \arrow[r] &\cocone(f) \arrow[r, "i^\vee"] & F^*\mcc^\vee \arrow[r,"f^\vee"] & \mca^\vee \arrow[r] & \cdots
\end{tikzcd}
\end{equation}
The non-degeneracy condition is that the morphisms $F^*\Psi(\sigma)$, $\zeta$, and $\zeta'$ are all quasi-isomorphisms.
\end{defn}

Several comments are in order to clarify the definition above and to connect it to other notions in the literature. First, the diagram in the definition above mirrors the Poincaré-Lefschetz Duality for a manifold with boundary $(M,\partial M)$, where $\cone(f)$ plays the role of $H^*(M)$, $F^*\mcc_\Delta$ plays the role of $H^*(\partial M)$, and $\mca_\Delta$ plays the role of $H^*(M,\partial M)$. Second, an \dfn{(absolute) weak right CY structure of dimension $n$} on an $A_\infty$ category $\mcc$ is defined as a class in $\HH^*(\mcc)[n-1] \simeq \HH_*(\mcc)^\vee[n-1]$ such that its image in $\hom_{\mcc-\mcc}(\mcc_\Delta,\mcc^\vee)$ under $\Psi$ is a quasi-isomorphism. Hence, a weak relative CY structure on $F:\mca \to \mcc$ incorporates a weak absolute CY structure on $\mcc$.

By Proposition~\ref{prop:hoch-a-infty}, (absolute) weak right CY structures of dimension $n$ are sometimes defined simply as bimodule quasi-isomorphisms $\mcc_\Delta \xrightarrow{\sim} \mcc^\vee[n]$. The relative version can be defined in a similar vein. Specifically, the structure consists of
\begin{itemize}
\item An $A_\infty$ quasi-isomorphism $\kappa: \mcc_\Delta[n-1] \to \mcc^\vee$ and
\item An $A_\infty$ pre-morphism $\phi: \mca_\Delta[n] \to \mca^\vee$ 
\end{itemize}
so that $\phi$ is a null-homotopy for the composition
\[ \begin{tikzcd}
	& \mca_\Delta[n-1] \arrow[r,"f"]  & F^*\mcc_\Delta[n-1] \arrow[r,"F^*\kappa"]  & F^*\mcc^\vee \arrow[r,"f^\vee"] & \mca^\vee,
\end{tikzcd}
\]
and that it induces quasi-isomorphisms $\zeta$ and $\zeta'$, which fit into the exact same coherent diagram as Diagram~\eqref{eq:weak-right-cy}, with $F^*\Psi(\sigma)$ replaced by $F^*\kappa$. This is the definition of a weak right relative CY structure in \cite{chen:lsft-dual}.

Finally, a ``right'' CY structure is also called a ``proper'' CY structure in the literature; see \cite{ganatra:cyclic, gps:mirror, sheridan:fano}. Our terminology agrees with \cite{bd:rel-cy}. There is a parallel notion of a left or smooth CY structure defined using the inverse dualizing bimodule, but we will not discuss these structures in this paper.

We will not extract a relative CY structure directly from the geometry under our consideration. Rather, we will use an even weaker structure as an intermediate step. This definition is motivated by the earlier formulation of the weak relative right CY structure using $A_\infty$ bimodules.

\begin{defn}
\label{defn:v-rel-cy}
Let $\mca$ be a (strictly) unital $A_\infty$ category, let $\mcm, \mcn$ be $\mca$-bimodules, and let $f: \mcn \to \mcm$ be an $A_\infty$ morphism. A \dfn{very weak relative Calabi-Yau structure of dimension $n$} on $f$ consists of $A_\infty$ quasi-isomorphisms $\eta:\mcm^\vee[n] \to \cone(f)$, $\theta:\mcn^\vee[n] \to \mcn[1]$, and $\eta': \cocone(f)[n+1] \to \mcm[1]$ so that the diagram
\begin{equation} 
\label{eq:very-weak-CY}
 \begin{tikzcd}[column sep=scriptsize]
	\cdots \arrow[r] & \mcm^\vee[n] \arrow[r,"f^\vee"] \arrow[d,"\eta"] & \mcn^\vee[n] \arrow[r, "\pi^\vee_\mcn"] \arrow[d,"\theta"] & \cocone(f)[n+1] \arrow[d,"\eta'"] \arrow[r] & \cdots \\	
	\cdots \arrow[r] &\cone(f) \arrow[r, "\pi_\mcn"] & \mcn[1] \arrow[r,"f"] & \mcm[1] \arrow[r] & \cdots
\end{tikzcd}
\end{equation}
commutes up to homotopy.
\end{defn}

It turns out that the existence of the leftmost coherent square in Diagram~\eqref{eq:very-weak-CY} is sufficient for constructing a very weak relative CY structure.

\begin{lem}
    \label{lem:very-weak-shortcut}
    Given a morphism $f:\mcn \to \mcm$ and quasi-isomorphisms $\eta:\mcm^\vee[n] \to \cone(f)$ and $\theta:\mcn^\vee[n] \to \mcn[1]$ such that $\pi_\mcn \eta \sim \theta f^\vee$, there exists a quasi-isomorphism $\eta':\cocone(f)[n+1] \to \mcm[1]$ such that $(\eta,\theta,\eta')$ constitutes a very weak relative Calabi-Yau structure on $f$.
\end{lem}

\begin{proof}
    Let  $H:\mcm^\vee[n] \to \mcn[1]$ be the homotopy between $\pi_\mcn \eta$ and $\theta f^\vee$.  Define $\eta' = (\pi_{\mcm}\eta + f H)i_{\mcm}^\vee + f \theta i_{\mcn}^\vee$. That $\eta'$ is a bimodule morphism follows from $\delta$ being a derivation on the category of pre-morphisms together with the facts that $\eta$, $i_{\mcm}^\vee$, $f$, and $\theta$ are morphisms, that $\delta(\pi_{\mcm}) = f \pi_\mcn$, and that $\delta(i_{\mcn}^\vee) = f^\vee i_{\mcm}^\vee$.

    With $\eta'$ defined above, it is immediate that the second square in Diagram~\eqref{eq:very-weak-CY} commutes strictly. That the next (unseen) square commutes up to homotopy follows a similar idea as \cite[Lemma 4.9]{yeung}. Define a pre-morphism $K:\cocone(f)[n+1] \to \cone[1]$ by
    \begin{equation*}
        K = (\eta',(\theta f^\vee+H) i_\mcm^\vee + \theta i_\mcn^\vee).
    \end{equation*}
    Since 
    \begin{equation*}
        \eta i_\mcm^\vee + i_\mcm \eta' = (\pi_\mcm\eta i_\mcm^\vee + \eta', \pi_\mcn \eta i_\mcm^\vee),
    \end{equation*}
    the equation $\eta i_\mcm^\vee + i_\mcm \eta' = \delta(K)$ holds in the first component due to the definition of $\eta'$ and the fact that $\eta'$ is a bimodule morphism. It holds in the second component since $\theta, \theta f^\vee$ are bimodule morphisms and $H$ is a homotopy between $\pi_\mcn \eta$ and $\theta f^\vee$. Thus, $K$ is a homotopy between $\eta i_\mcm^\vee$ and $i_\mcm \eta'$ as desired.

    Finally, the Five Lemma guarantees that $\eta'$ is a quasi-isomorphism.
\end{proof}

We end this section by introducing an algebraic tool to produce a weak right relative CY structure from a very weak relative CY structure.  In Section \ref{sec:main-pf}, we will use geometric considerations to construct a very weak relative CY structure on $A_\infty$ bimodules associated to a Legendrian link. The algebraic tool described in this section will then exploit the extra structure in this relationship to derive the main theorem. We emphasize that the form and proof of the algebraic tool rely on the assumption that the underlying vector spaces of all $A_\infty$ structures involved are finite-dimensional, though a more general statement may hold.


In order to state the algebraic tool, we first specify the extra information needed to pass from very weak to weak relative CY structures.  Begin with strictly unital $A_\infty$ categories $\mca$ and $\mcc$, an $\mca$-bimodule $\mcm$, a functor $\pi: \mca \to \mcc$, and a morphism $\rho: \pi^*\mcc_\Delta[-1] \to \mcm$.  For convenience, we set $\mcn = \pi^*\mcc_\Delta$.  We say that the data $(\alg \stackrel{\pi}{\to} \mcc, \mcn[-1] \stackrel{\rho}{\to} \mcm)$ is a \dfn{conical pair} if the diagonal bimodule $\mca_\Delta$ is $\cone(\rho)$ and if the functor $\pi$ induces the canonical projection $\pi_\mcn: \mca_\Delta \to \mcn$ of $\cone(\rho)$.

\begin{prop}
\label{prop:alg-gadget}
    Let $(\alg \stackrel{\pi}{\to} \mcc, \mcn[-1] \stackrel{\rho}{\to} \mcm)$ be a conical pair. If there is a very weak relative Calabi-Yau structure of dimension $-n$ on $\rho$ where $\theta:\mcn^\vee \to \mcn[n-1]$ is the pullback of a bimodule morphism $\mcc^\vee \to \mcc_\Delta$ along $\pi$, then there exists a weak right relative Calabi-Yau structure of dimension $n$ on $\pi$. 
\end{prop}

\begin{proof} 
To set notation, suppose that the given very weak CY structure on $\rho$ is captured by the following homotopy commutative diagram:
\begin{equation} \label{eq:weak-CY-pf}
 \begin{tikzcd}
	\cdots \arrow[r] & \mcm^\vee \arrow[r,"\rho^\vee"] \arrow[d,"\eta"] & \mcn^\vee[1] \arrow[r, "\pi^\vee_\mcn"] \arrow[d,"\theta"] & \mca_\Delta^\vee[1] \arrow[d,"\eta'"] \arrow[r] & \cdots \\	
	\cdots \arrow[r] &\mca_\Delta[n] \arrow[r, "\pi_\mcn"] & \mcn[n] \arrow[r,"\rho"] & \mcm[n+1]\arrow[r] & \cdots
\end{tikzcd}
\end{equation}
Compare with Diagram~\eqref{eq:very-weak-CY} and note that the degree shifts result from the the assumptions of dimension $-n$ and that $\mca_\Delta \simeq \cone(\rho)$ by the canonical pair construction.

Our goal is to construct $[\phi] \in \HH^*(\mca)[n]$ and $[\sigma] \in \HH^*(\mcc)[n-1]$ that satisfy the requirements of a weak right relative CY structure. As a first step, note that the units $e$ of $\mca$ and of $\mcc$ determine canonical Hochschild cocycles $e = (e,0,0,\ldots)$ in $\CH^*(\mca,\mca_\Delta)$ and in  $\CH^*(\mcc,\mcc_\Delta)$. Using Proposition~\ref{prop:homotopy-equiv}, let $\xi:\mca_\Delta[n] \to \mcm^\vee$ and $\kappa:\mcn[n] \to \mcn^\vee[1]$ denote the homotopy inverses of the quasi-isomorphisms $\eta$ and $\theta$, respectively. Since $\theta$ is the pullback of a bimodule morphism $\mcc^\vee \to \mcc_\Delta$ along $\pi$, note that $\kappa$ is also the pullback of a bimodule morphism $\mcc_\Delta \to \mcc^\vee$ along $\pi$. We abuse the notation by using $\kappa$ to denote both morphisms. Then, we obtain canonical Hochschild cohomology classes 
\begin{equation*}
    [\lambda] := [\xi^* e] \in \HH^*(\mca,\mcm^\vee)\quad \text{and}\quad [\sigma] := [\kappa^* e] \in \HH^*(\mcc).
\end{equation*}
Using the coherence of the left square of Diagram~\eqref{eq:weak-CY-pf} and the functoriality of Hochschild cohomology, we have 
\begin{equation*}
    [(\rho^\vee)^*\lambda] = [\kappa^*\pi_\mcn^*\eta^*\lambda] = [\kappa^*\pi_\mcn^*e].
\end{equation*}
Further, a computation using the definition of $\pi^\sharp$ shows that $[\kappa^*\pi_\mcn^*e] = [\pi^\sharp \sigma]$. Let $s \in \CH^*(\mca,\mcn^\vee)$ be a bounding cochain between $(\rho^\vee)^*\lambda$ and $\pi^\sharp \sigma$ in $\CH^*(\mca,\mcn^\vee)$.  Using the mapping cone structure of $\mca^\vee$, consider the Hochschild cochain
\begin{equation}
\label{eq:phi}
    \phi := (\lambda, s) \in \CH^*(\mca).
\end{equation}
We claim that $ (\phi, \sigma) \in \CH^*(\pi)$ constitutes a weak relative right Calabi-Yau structure on $\pi:\mca \to \mcc$.

To verify, we first show that $(\phi,\sigma)$ is closed in $\CH^*(\pi)$. We already know by construction that $\sigma$ is a cocycle in $\CH^*(\mcc)$. Thus, it suffices to compute that
\begin{equation*}
    b^*\phi + \pi^*\sigma = (b^*\lambda, (\rho^\vee)^*\lambda + b^*s + \pi^\sharp \sigma) = 0,
\end{equation*}
where we used the fact that $\lambda$ is a cocycle in $\CH^*(\mca,\mcm^\vee)$ and that $s$ is the bounding cochain between $(\rho^\vee)^*\lambda$ and $\pi^\sharp \sigma$.

Next, consider the following coherent diagram induced by $(\phi,\sigma)$.
\begin{equation} \label{eq:weak-CY-pf-2} \begin{tikzcd}
	\cdots \arrow[r] & \mca_\Delta[n-1] \arrow[r,"\pi_\mcn"] \arrow[d,"\zeta"] & \mcn[n-1] \arrow[r, "I_\mcn"] \arrow[d,"\pi^*\Psi(\sigma)"] & \cyl(\rho)[n] \arrow[d,"\zeta'"] \arrow[r] & \cdots \\
	\cdots \arrow[r] &\cocyl(\rho)[-1] \arrow[r, "I_\mcn^\vee"] & \mcn^\vee \arrow[r,"\pi^\vee_\mcn"] & \mca_\Delta^\vee \arrow[r] & \cdots
\end{tikzcd}
\end{equation}
In this case, an explicit chain model for $\zeta$ is given by $\zeta = ((\pi^* \Psi(\sigma))\pi_\mcn, \Psi(\lambda))$. Then, $I_{\mcm^\vee}\Pi_{\mcm^\vee}\zeta  = I_{\mcm^\vee}\Psi(\lambda)$, where $I_{\mcm^\vee}$ and $\Pi_{\mcm^\vee}$ are the naïve inclusion and projection from Lemma~\ref{lem:mapping-cyl-homotopy}. Since $I_{\mcm^\vee}\Pi_{\mcm^\vee}$ is homotopic to the identity $\bbid_{\cocyl(\rho)}$, it follows that $\zeta$ is homotopic to $I_{\mcm^\vee}\Psi(\lambda)$. By Proposition~\ref{prop:hoch-a-infty}, since $\eta\Psi(\lambda) \simeq \Psi(\eta^*\xi^*e) \simeq \Psi(e) = \bbid_{\mca_\Delta}$, we see that $\Psi(\lambda)$ is a quasi-isomorphism with $\eta$ its homotopy inverse. This implies that $\eta \Pi_{\mcm^\vee}$ is a homotopy inverse for $\zeta$ and $\zeta$ is a quasi-isomorphism. Finally, the Five Lemma guarantees that $\zeta'$ is a quasi-isomorphism.
\end{proof}

Conceptually, the proof of Proposition~\ref{prop:alg-gadget} can be viewed as an algebraic abstraction and generalization of the constructions of the fundamental class and the inverse to the duality map from \cite{duality}. In particular, $[\lambda]$ and $[\sigma]$ represent the \dfn{fundamental classes} in the Hochschild cohomology groups $\HH^*(\mca,\mcm^\vee)$ and $\HH^*(\mcc)$ as the preimages of the unit $e$ in $\HH^*(\mca,\mca_\Delta)$ and  $\HH^*(\mcc,\mcc_\Delta)$ under the Hochschild cohomology isomorphisms induced by $\eta$ and $\theta$, respectively. The homotopy inverses $\xi$ and $\kappa$, in turn, can be realized as cap products with $\lambda$ and $\sigma$. We refer the reader to the end of Section~\ref{ssec:weak-CY-final-proof} for more details.

\section{Pushoffs, Reeb Chords, and Disks}
\label{sec:copies}

In this section, we examine the geometric models that underlie the main algebraic objects in the paper.  The $n$-copy of a Legendrian link, described in Section~\ref{ssec:n-copy}, will lead to the definition of the augmentation categories $\Aug_\pm(\leg)$ and $\mcc$ in Section~\ref{ssec:aug-cat}.  The $(s,r)$-copy of the $2$-copy of a Legendrian knot, described in Sections~\ref{ssec:sr-copy} and \ref{ssec:2-copy}, will lead to the definition of the augmentation bimodules $\mcm_\pm$ and $\mcn$ in Section~\ref{ssec:aug-bimod}.  Finally, the $(s,r)$-copy of the separated $2$-copy of a Legendrian knot, described in Section~\ref{ssec:sep-2-copy}, will yield the structures that play a key role in the proof of the main theorem.

For each of these geometric models, we begin by examining the Reeb chords and immersed disks.  We then discuss how to identify disks in the copied links with enrichments of disks in the original Lagrangian diagram.   These ``disk identification lemmas'' will be important tools in the proof of the main theorem.

\subsection{The n-Copy of a Link}
\label{ssec:n-copy}

We begin by setting notation for the Reeb chords of the $n$-copy of a Legendrian link, and analyze the behavior of immersed disks in the $n$-copy.  These constructions will lead to the definition of the augmentation categories in Section~\ref{sec:aug-cat}.

\subsubsection{The n-Copy Construction}
\label{sssec:n-copy-geometry}

Given a Legendrian link $\leg$, we form its \dfn{$n$-copy $\nleg{n}$} by taking $n$ copies of $\leg$, pushing each successive copy in the negative Reeb ($z$) direction by a small amount, and then perturbing the results by the $1$-jet of a Morse function $f: \leg \to \rr$ in a small Weinstein neighborhood, with an additional perturbation near the critical points of $f$; see Figure~\ref{fig:n-copy}. The $n$-copy depends on the choice of $f$, though we do not explicitly include $f$ in the notation for readability.  We assume that $f$ has a unique maximum and a unique minimum on each component, though the theory in this section may be developed for more general perturbations.

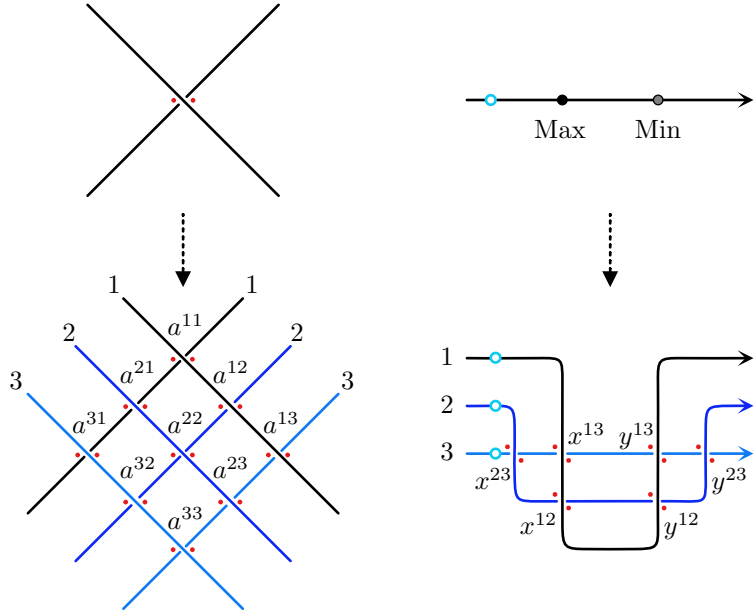
\begin{figure}
\labellist
\small\hair 2pt
 \pinlabel {$1$} [ ] at 40 132
 \pinlabel {$2$} [ ] at 23 114
 \pinlabel {$3$} [ ] at 3 96
 \pinlabel {$1$} [ ] at 92 132
 \pinlabel {$2$} [ ] at 109 114
 \pinlabel {$3$} [ ] at 128 96
 \pinlabel {$1$} [r] at 170 105
 \pinlabel {$2$} [r] at 170 87
 \pinlabel {$3$} [r] at 170 69 
 \pinlabel {$a^{11}$} [b] at 67 111
 \pinlabel {$a^{21}$} [b] at 49 93
 \pinlabel {$a^{31}$} [b] at 31 74
 \pinlabel {$a^{12}$} [b] at 84 93
 \pinlabel {$a^{22}$} [b] at 67 74
 \pinlabel {$a^{32}$} [b] at 49 57
 \pinlabel {$a^{33}$} [b] at 67 38
 \pinlabel {$a^{23}$} [b] at 84 57
 \pinlabel {$a^{13}$} [b] at 103 74
 \pinlabel {$x^{23}$} [t] at 183 65
 \pinlabel {$x^{12}$} [t] at 200 48
 \pinlabel {$x^{13}$} [t] at 218 82
 \pinlabel {$y^{13}$} [t] at 237 82
 \pinlabel {$y^{12}$} [t] at 254 48
 \pinlabel {$y^{23}$} [t] at 272 65
 \pinlabel {Max} [t] at 208 196
 \pinlabel {Min} [t] at 245 196
\endlabellist

\centerline{\includegraphics{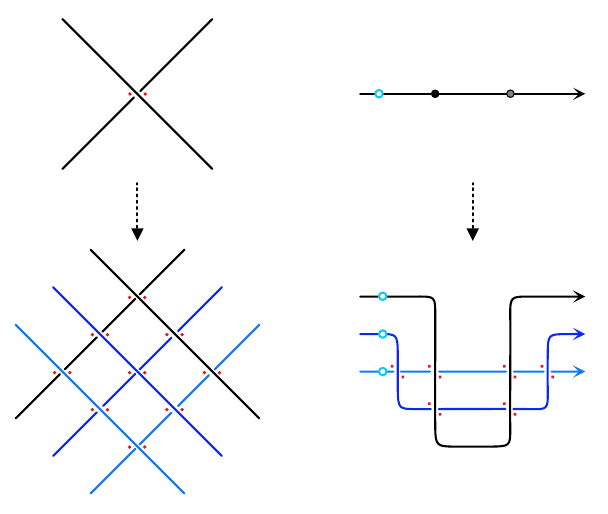}}
\caption{(a) For every crossing of $\leg$, there are $n^2$ crossings of $\nleg{n}$.  (b) For every critical point of the perturbing function $f$, there are $\frac{1}{2}n(n-1)$ crossings of $\nleg{n}$.}
\label{fig:n-copy}
\end{figure}

To set notation for the Reeb chords of $\nleg{n}$, we label the copies of $\nleg{n}$ from top to bottom by $\nleg{n}_1, \ldots, \nleg{n}_n$. There is a natural link grading on $\nleg{n}$ in which the set of Reeb chords splits into the union of sets $\reeb^{ij}$, each of which consists of Reeb chords that begin on the (lower) component $\nleg{n}_j$ and end on the (upper) component $\nleg{n}_i$, as in Section~\ref{ssec:dga-link}.

As shown in Figure~\ref{fig:n-copy}(a), for every Reeb chord $a_k$ of $\leg$, there are $n^2$ Reeb chords of $\nleg{n}$, one in each of the sets $\reeb^{ij}$.  We label these chords by $a^{ij}_k \in \reeb^{ij}$. For every critical point of the perturbing function $f$, there are $\frac{1}{2}n(n-1)$ Reeb chords of $\nleg{n}$, one in each set $\reeb^{ij}$ for $i>j$; see Figure~\ref{fig:n-copy}(b).  We denote by $x_l^{ij}$ the chord in $\reeb^{ij}$ at a maximum of $f$, and by $y_l^{ij}$ the chord in $\reeb^{ij}$ at a minimum. 

Since the $n$ copies of $\nleg{n}$ are separated by a small distance in the $z$ direction, the lengths of the Reeb chords $x_l^{ij}$ and $y_l^{ij}$ are much smaller than the lengths of the Reeb chords $a^{ij}_k$. Thus, we call the former chords \dfn{short chords}, collected in the set $\reeb_f$, and the latter chords \dfn{long chords}, which comprise the set $\reeb_\leg$.  We further divide the short chords into $\reeb_x$, which correspond to the maxima of $f$, and $\reeb_y$, which correspond to the minima of $f$.

We next place a base point on the original link $\leg$ at each maximum of $f$. We then place base points on every copy $\nleg{n}_i$ of $\nleg{n}$ preceding (with respect to the orientation of $\leg$) the corresponding half twist. Recall from Section~\ref{ssec:dga-link} that these base points correspond to invertible symbols  $(t_l^i)^{\pm 1}$, with $\base^{i}$ denoting the collection of symbols $(t_l^i)^{\pm 1}$ associated with base points on $\nleg{n}_i$.   As in Section~\ref{ssec:dga-link}, let $\gens^{ij} = \reeb^{ij} \cup \base^{ij}$, where $\base^{ij} = \emptyset$ if $i \neq j$ and $\base^i$ otherwise.

The sets of Reeb chords and base points are all \emph{graded} sets, with the gradings coming from \cite{nrssz:aug-sheaf}:
\begin{itemize}
    \item invertible generators $(t^i_k)^{\pm 1}$ have grading $|t^i_k| = |(t^i_k)^{-1}| = 0$;
    \item long chords $a^{ij}_k$ have grading $|a_k^{ij}| = |a_k|$;
    \item maxima $x_l^{ij}$ have grading $|x_l^{ij}| = 0$; and
    \item minima $y_l^{ij}$ have grading $|y_l^{ij}| = -1$.
\end{itemize} 

As in Section~\ref{ssec:ce-dga}, the Chekanov-Eliashberg DGA $\alg^n$ of the $n$-copy is generated by $\bigcup_{i,j} \gens^{ij}$.  

\subsubsection{Disks in the $n$-Copy}
\label{sssec:n-copy-disks}

The disks that define the differential of the Chekanov-Eliashberg DGA of an $n$-copy $\nleg{n}$ may be identified with ``thin'' disks that collapse into $\pi_{xy}(\leg)$ and ``thick'' disks that correspond to disks that define the Chekanov-Eliashberg differential of $\leg$.  This idea goes back to \cite{s1bundles} and the original proof of duality for linearized LCH \cite{duality}, with further developments most relevant to this paper in \cite{nrssz:aug-sheaf}.  The notation and work needed for this identification serves as a model for later, more subtle, disk identification lemmas that will be essential for the proof of the main theorem in Section~\ref{sec:main-pf}.

To enable translation between chords in the $n$-copy and chords in the original Legendrian, we introduce the \dfn{stick-together map} $s: \rr^2 \to \rr^2$, which retracts the Lagrangian diagram of $\nleg{n}$ onto the diagram of $\leg$.  The stick-together map induces the \dfn{algebraic projection} map  $\check{\alpha}: \alg^n \to \alg$, which is defined on the generators as follows:
\begin{itemize}
    \item A Reeb chord in $\reeb_{\leg}$ is mapped to the associated Reeb chord in $\reeb$.
    \item A Reeb chord in $\reeb_{f}$ is mapped to $1$.
    \item A base point in $\base^{i}$ is mapped to the associated base point in $\base$.
\end{itemize}
See Figure~\ref{fig:n-copy} once again. For compactness of notation, we will, at times, denote the image of $a \in \alg^n$ under $\check{\alpha}$ by $\check{a}$.  In particular, we have $\check{a}^{ij}_k = a_k$ and $\check{x}_l^{ij} = 1 = \check{y}_l^{ij}$.

The stick-together map divides the disks that define the differential in the $n$-copy into two types. A disk $u \in \Delta_{\nleg{n}}(a, \mba)$ is \dfn{thick} if $s \circ u$ is still immersed. In contrast, a disk $u \in \Delta_{\leg^n}(a, \mba)$ is \dfn{thin} if the image of $s \circ u$ lies entirely inside $\pi_{xy}(\leg)$.  A geometric argument using the convexity of the corners of disks in $\Delta_{\nleg{n}}$ and Figure~\ref{fig:n-copy}, or an examination of the formulae of \cite[Proposition 4.14]{nrssz:aug-sheaf} shows that thick and thin disks are the only two possibilities. 

Before analyzing thick and thin disks more precisely, we need to introduce some notation.  Given a choice of $n-1$ indices $1 \leq i_1<  \cdots < i_{n-1} \leq N$, we write the word $\mba = a_1 \cdots a_N$ as
\begin{equation} \label{eq:copy-form}
\mba = \mba_1 a_{i_1} \mba_2 a_{i_2} \mba_3 \cdots \mba_{n-1} a_{i_{n-1}} \mba_n.   
\end{equation}
This word could come from the generators for a single Legendrian, or could come from the generators for an $n$-copy.  When working with words in $\alg^n$, we say that a pair $(a,\mba)$ is in \dfn{positive $n$-copy form} with respect to the indices $\{i_1, \ldots, i_{n-1}\}$ if 
\begin{itemize}
    \item $a \in \reeb^{1,n}$,  
    \item $a_{i_k} \in \reeb^{k,k+1}$, and 
    \item $\mba_k$ a word in the alphabet $\gens^{kk}$.
\end{itemize}
Similarly, the pair $(a,\mba)$ is in \dfn{negative $n$-copy form} if $a \in \reeb^{n,1}$, $a_{i_k} \in \reeb^{n-k+1, n-k}$, and $\mba_k$ a word in the alphabet $\gens^{n-k+1, n-k+1}$. 

We turn to describing the thin disks. Propositions 4.14 and 4.26 of \cite{nrssz:aug-sheaf} or an examination of Figure~\ref{fig:n-copy-thin-disks} yield the following characterization of thin disks.  Note that the formula is simplified --- as in Proposition 4.26 of \cite{nrssz:aug-sheaf} --- by the assumption that $\nleg{n}$ is \dfn{simply perturbed}. Namely, $\leg^n$ is perturbed by a Morse function $f:\leg \to \rr$ with one maximum and one minimum on each component that are adjacent in the Lagrangian diagram of $\leg$, preceded (with respect to the orientation of $\leg$) by the unique base point on $\leg$.

\begin{lem}[$n$-Copy Thin Disk Identification] \label{lem:thin-disk}
    Suppose $\nleg{n}$ is simply perturbed and that $(a,\mba)$ is in positive $n$-copy form. The set $\Delta_{\leg^{n}}(a,\mba)$ consists of thin disks precisely when $(a,\mba)$ is in one of the following forms; see Figure~\ref{fig:n-copy-thin-disks}.  We suppress the indices of the $x$, $y$, and $t$ generators for readability.
    \begin{description}
    \item[Flowlines] 
    \[\begin{array}{ll}
        (x^{12}, y^{12}) & (x^{12}, t^{-1} y^{12} t)  \\
        (x^{13}, x^{12} y^{23}) & (x^{13}, t^{-1} y^{12} t x^{23}) 
    \end{array}\]
    \item[Partial Flowlines]
    \[\begin{array}{ll}
         (a^{12}, y^{12} a^{22})  &  (a^{13}, y^{12} a^{23})  \\
         (a^{12}, a^{11} y^{12})
         & (a^{13}, a^{12} y^{23}) 
    \end{array}\]
    \item[Constant Flowlines] $(y^{13}, y^{12} y^{23})$
    \end{description}
    In particular, a disk in $\Delta_{\leg^{n}}(a,\mba)$ is thin if and only if $\mba$ contains a $y$ generator. 
    
    If $(a,\mba)$ is in negative $n$-copy form, then no disk in $\Delta_{\leg^{n}}(a,\mba)$ is thin.
\end{lem}

\begin{figure}
\labellist
\small\hair 2pt
 \pinlabel {$1$} [r] at 8 284
 \pinlabel {$2$} [r] at 8 275
 \pinlabel {$3$} [r] at 8 266
 \pinlabel {$1$} [b] at 81 301
 \pinlabel {$2$} [b] at 90 301
 \pinlabel {$3$} [b] at 99 301
 \pinlabel {$x$} [t] at 28 253
 \pinlabel {$y$} [t] at 61 253
 \pinlabel {$a$} [ ] at 107 290
 \pinlabel {Flowlines} [t] at 123 169
 \pinlabel {Partial Flowlines} [t] at 123 12
 \pinlabel {Constant Flowlines} [t] at 307 110
\endlabellist
    \centering
    \includegraphics[width=4.9in]{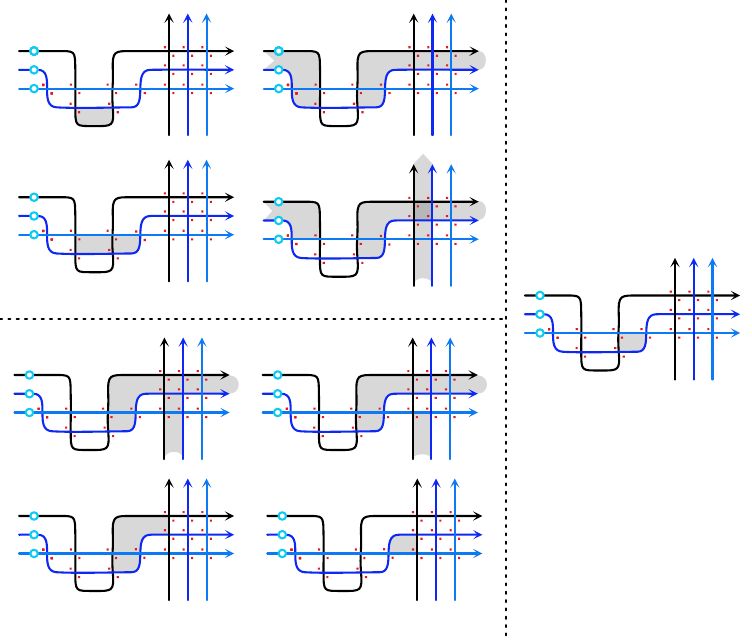}
    \caption{The thin disks in the $n$-copy of a Legendrian link take one of these forms, as listed in Lemma~\ref{lem:thin-disk}. We will use the red dots to indicate crossing information in this and subsequent figures that display disks in Lagrangian diagrams.}
    \label{fig:n-copy-thin-disks}
\end{figure}
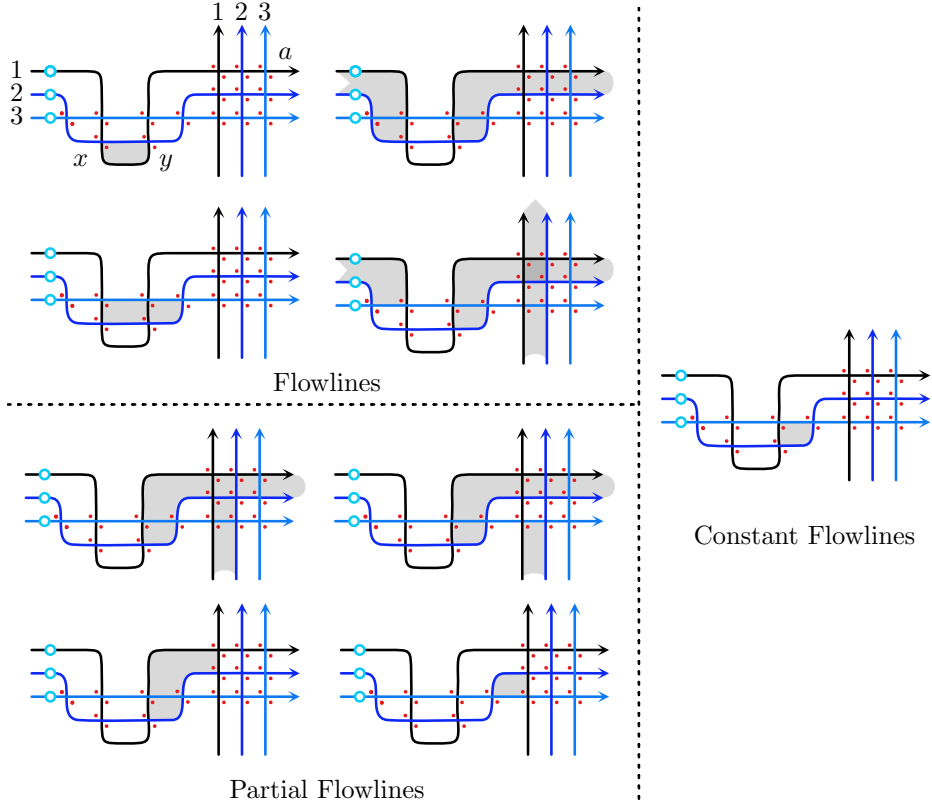

We next describe thick disks, building on the definition of $\Delta_{\leg}(a, \mba)$.

\begin{defn} \label{defn:enriched-disk}
    An \dfn{$m$-enriched disk} is a pair $(u,\mathbf{w})$ consisting of a disk $u \in \Delta_\leg(a, \mba = a_1 \cdots a_N)$ and a multiset of indices $\mbw$ of total multiplicity $m$ whose elements are from $\{1, \ldots, N\}$, where an element $i$ may have multiplicity greater than $1$ only if $a_i \in \base_-$. The set of $m$-enriched disks with corners $(a,\mba)$, up to smooth reparametrization, is denoted $\Delta^m_\leg(a, \mba)$.
\end{defn}

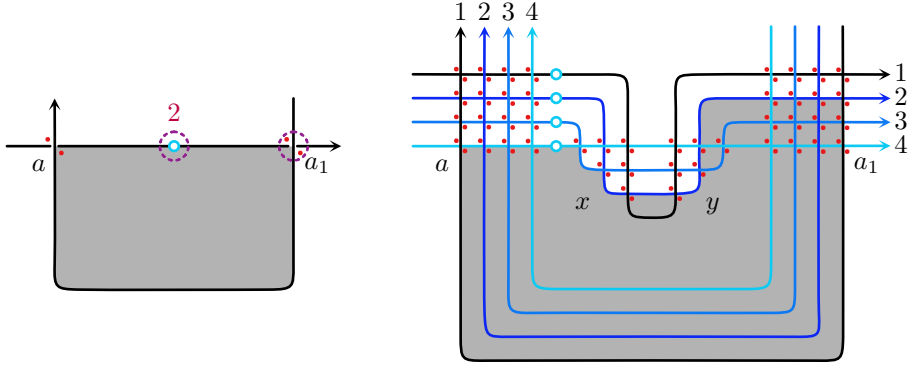
\begin{figure}
\labellist
\small\hair 2pt
 \pinlabel {$a$} [ ] at 17 80
 \pinlabel {$a_1$} [ ] at 122 80
 \pinlabel {\textcolor{purple}{$2$}} [ ] at 68 100
 \pinlabel {$a$} [ ] at 169 80
 \pinlabel {$a_1$} [ ] at 329 80
 \pinlabel {$x$} [t] at 222 70
 \pinlabel {$y$} [t] at 271 70
 \pinlabel {$1$} [ ] at 176 138
 \pinlabel {$2$} [ ] at 185 138
 \pinlabel {$3$} [ ] at 194 138
 \pinlabel {$4$} [ ] at 203 138
 \pinlabel {$1$} [ ] at 342 115
 \pinlabel {$2$} [ ] at 342 106
 \pinlabel {$3$} [ ] at 342 97
 \pinlabel {$4$} [ ] at 342 88
\endlabellist
    \centerline{\includegraphics{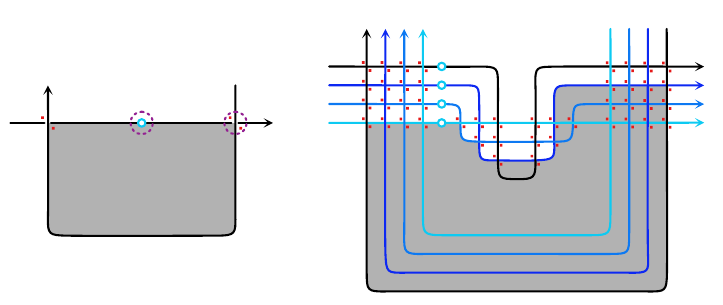}}
    \caption{The disk at left is an enriched disk in $\Delta_\leg^3(a, a_1 t^{-1})$ with $\mathbf{w}=\{1, 2,2\}$. The disk at right is its lift to a disk in $\Delta_{\nleg{4}}(a^{14}, a_1^{12} x^{23}x^{34} (t^4)^{-1})$, following  Lemma~\ref{lem:n-copy-dil}.}
    \label{fig:n-copy-dil}
\end{figure} 

Enriched disks allow us to identify thick disks in the diagram of the $n$-copy with objects in the diagram of original knot, as shown in Figure~\ref{fig:n-copy-dil}.  One direction of such an identification involves the stick-together map.  Given $u \in \Delta_{\nleg{n}}(a, \mba)$, with $(a,\mba)$ in positive $n$-copy form as in Equation \eqref{eq:copy-form}, let $s \circ u$ be the stick-together disk in the Lagrangian diagram of $\leg$.  We describe how $s \circ u$ induces an enriched disk $(\check{u}, \mbw(u))$ in $\Delta^{n-1}_{\leg}(\check{a}, \check{\mba})$, with enrichment $\mbw(u)$ to be determined, called the \dfn{projection} of $u$. 

If a convex corner of $u$ is labeled with $a \in \reeb_\leg$, then an examination of the configuration of crossings in Figure~\ref{fig:n-copy} shows that $s \circ u$ has a convex corner labeled by $\check{a} \in \reeb$ decorated with the same sign as $a \in \reeb_\leg$; we enrich $\check{u}$ by adding the position of $\check{a}$ in $\check{\mba}$ to $\mbw(u)$ if $a$ is one of the $a_{i_k}$, but not if $a$ is in one of the $\mathbf{a}_i$. Similarly, if the boundary of $u$ passes through a basepoint $t^{\pm 1} \in \base^{i}$, then $s \circ u$ passes through the corresponding basepoint $\check{t}^{\pm 1}$; we do not yet enrich these marked points.

We are left to understand the relationship between $u$ and $s \circ u$ near negative corners of $u$ at maxima $x \in \reeb^{i,i+1}_x$.  Since $\check{x} = 1$, we record the corner of $u$ at $x$ in $\check{u}$ by enriching the position in $\check{a}$ of the associated basepoint (which necessarily appears in $\mba$).  As can be seen in Figure~\ref{fig:n-copy}, if the orientations of $\nleg{n}$ and the boundary of $u$ agree at $x$, then at most one $x$ can serve as a negative corner; this corresponds to the associated marked point being enriched once.  On the other hand, if the orientations disagree, then multiple generators $x$ can serve as negative corners, and the position of the marked point may be enriched multiple times. See Figure~\ref{fig:n-copy-dil} for an example. This completes the description of the enriched disk $\check{u}$.

Conversely, we may also describe a \dfn{(geometric) lift} of an enriched disk $(v, \{m_1 i_1, \ldots, m_l i_l\}) \in \Delta^{n-1}_\leg(a, \mba)$ to a thick disk in the $n$-copy.  It suffices to specify the corners of the lift $\hat{v}$ of $v$ in $\nleg{n}$, as this will also specify the lift of the boundary.  We start by lifting $a$ to $a^{1n}$.  For each $k \in \{1, \ldots, l\}$, let $r_k = \sum_{j=1}^k m_j$.  We lift an enriched Reeb chord $a_{i_k}$ to $a_{i_k}^{r_k,r_k+1}$.  Similarly, we lift an enriched marked point $t_{i_k}^{\pm 1}$ with multiplicity $m_k$ to $m_k$ corners between strands $r_{k-1}+1$ and $r_{k-1}+2$ up to corners between strands $r_k$ and $r_k+1$ at the nearby maxima, preceded by the marked point $t^{r_{k-1}+1}_{i_k}$ or succeeded by $\left(t_{i_k}^{r_k+1}\right)^{-1}$. Note that these additional corners at the maxima add $m_k$ to the strand number, and hence the instructions for non-enriched generators and subsequent enriched corners or marked points are consistent.  Further, we add corners for all non-enriched generators in $\mba_i$ to the corresponding self-crossings (or marked points) in $\gens^{r_{i-1} r_{i-1}}$. An examination of the configuration of crossings in Figure~\ref{fig:n-copy} shows that the lifted corners yield an immersed disk in $\pi_{xy}(\nleg{n+1})$ with one positive corner at $a^{1n}$. Again, see Figure~\ref{fig:n-copy-dil} for an example of this process.

It is straightforward to check that the lifting process $v \mapsto \hat{v}$ inverts the projection process $u \mapsto \check{u}$ and vice versa. A similar projection / lifting framework works for disks in negative $n$-copy form, though this time, there are no corners at maxima in the $n$-copy, and no enrichments at the marked points for the enriched disks. We summarize these findings in the following lemma.

\begin{lem}[$n$-Copy Thick Disk Identification] 
\label{lem:n-copy-dil}
    Suppose $(a,\mba)$ is in positive $n$-copy form with no generators at the minima of $f$. There is a bijective correspondence between thick disks $u \in \Delta_{\nleg{n}}(a, \mba)$ and $(n-1)$-enriched disks $(s\circ u, \mbw(u))$ in $\Delta^{n-1}_{\leg}(\check{a}, \check{\mba})$.  An identical statement holds when $(a,\mba)$ is in negative $n$-copy form. 
\end{lem}

\begin{rem}
The strands of the $n$-copy traversed by the boundary of the disk need not be consecutive.  That is, the lemma holds for indices $1 \leq j_1 < \cdots < j_{k+1} \leq n$, with $a \in \reeb_\leg^{j_1,j_{k+1}}$ and $\mba$
a word in the alphabet $\gens^n$ with $a_{i_k} \in \reeb^{j_k,j_{k+1}}$ and $\mba_k$ a word in the alphabet $\gens^{j_k,j_k}$ with no generators at the minima of $f$.  The proof is the same as the one above, with somewhat more work when taking care of indices. Such a characterization is useful for directly checking $A_\infty$ relations, but will not be discussed further in this paper.
\end{rem}

\subsection{The (s,r)-Copy of a Link}
\label{ssec:sr-copy}

Our next step in setting the geometric underpinnings of the augmentation categories and their associated bimodules is to understand the $(s,r)$-copy of an ordered two-component Legendrian link.  This construction  involves taking the $s$-copy of one component and the $r$-copy of the other. Of particular interest are the $(s,r)$-copies of the $2$-copy of a Legendrian knot and of the separated $2$-copy of a Legendrian knot, both of which will be examined in greater detail in subsequent subsections. 

\subsubsection{The (s,r)-Copy Construction}
\label{sssec:sr-copy}

We begin by refining the link grading notation of Section~\ref{ssec:dga-link} in the presence of a pair of link gradings.  Given an ordered $2$-component Legendrian link $\leg = \leg_1 \sqcup \leg_2$ with the canonical $2$-component link grading $O \times \Upsilon$, we denote the \dfn{$(s,r)$-copy} link formed by the $s$-copy of the first component (with perturbation function $f_1$) and the $r$-copy of the second component (with perturbation function $f_2$) by $\nleg{s,r}$. Note that the ordering of the components of $\leg$ matters. 

Label the components of the $s$-copy from bottom to top and those of the $r$-copy from top to bottom.  Finally, place base points on $\nleg{s,r}$ just before the maxima of the functions $f_i$, as in the case of an $n$-copy. Note that there could be crossings between copies of $\leg_1$ and $\leg_2$ between the maximum and minimum of the perturbation functions.  See Figure~\ref{fig:s-r-copy}.

There is a $\max(s,r)$-component link grading $o \times \upsilon$ on $\nleg{s,r}$ so that, along with $O \times \Upsilon$, the Reeb chords of $\nleg{s,r}$ split into the union of the sets $\reeb^{ij|kl}$ which consists of Reeb chords that start on the $l^{th}$ component of the $r$-copy of the $j^{th}$ component of $\leg$ and end on the $k^{th}$ component of the $s$-copy of the $i^{th}$ component of $\leg$. We denote generators in $\reeb^{ij|kl}$ by $a^{ij|kl}$ when convenient. Similarly, the link gradings $O \times \Upsilon$ and $o \times \upsilon$ also split the marked points of $\nleg{s,r}$ into sets $\base^{i|k}$, and we denote a generator therein by $t^{i|k}$. As usual, we also define $\gens^{ij|kl} = \reeb^{ij|kl} \cup \base^{ij|kl}$, with $\base^{ij|kl} = \emptyset$ if $i \neq j$ or $k \neq l$. Call the elements of $\reeb^{ij|kl}$ with $i \neq j$ \dfn{mixed}, and those with $i=j$ \dfn{pure}.  

We extend the algebraic projection map $\check{\alpha}$ to the $(s,r)$-copy setting as a map the Chekanov-Eliashberg DGA of $\leg^{s,r}$ to that of $\leg$ induced from the stick-together map.  On the generators, we define 
\begin{align*}
\check{\alpha}(a^{ij|kl}) &= \check{a}^{ij|kl} = a^{ij} \\  \check{\alpha}(t^{i|k}) &= \check{t}^{i|k} = t^i
\end{align*}
Thick and thin disks in $\leg^{s,r}$ are defined  as in Section \ref{ssec:n-copy}. 

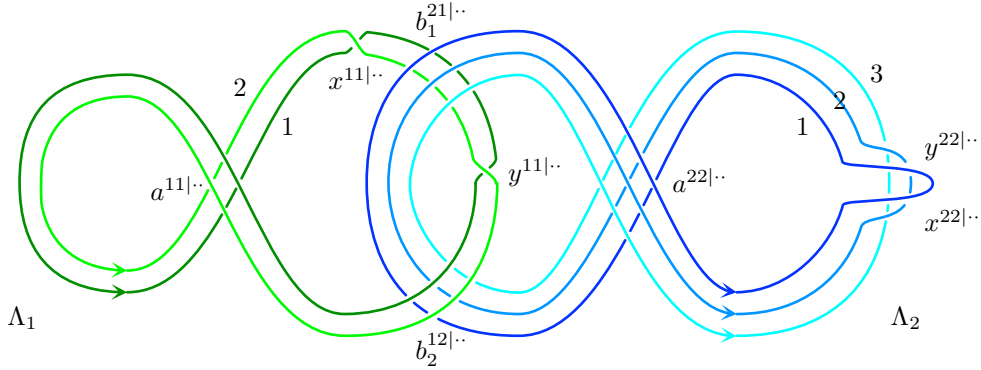
\begin{figure}
\labellist
\small\hair 2pt
 \pinlabel {$\leg_1$} [ ] at 14 13
 \pinlabel {$\leg_2$} [ ] at 347 13
 \pinlabel {$x^{11|\cdot \cdot}$} [t] at 140 110
 \pinlabel {$y^{11|\cdot \cdot}$} [l] at 195 68
 \pinlabel {$a^{11|\cdot \cdot}$} [r] at 85 62
 \pinlabel {$b_1^{21|\cdot \cdot}$} [b] at 172 117
 \pinlabel {$b_2^{12|\cdot \cdot}$} [t] at 172 10
 \pinlabel {$a^{22|\cdot \cdot}$} [l] at 257 65
 \pinlabel {$x^{22|\cdot \cdot}$} [l] at 352 50
 \pinlabel {$y^{22|\cdot \cdot}$} [l] at 352 78
 \pinlabel {$2$} [ ] at 96 100
 \pinlabel {$1$} [ ] at 114 85
 \pinlabel {$3$} [ ] at 336 105
 \pinlabel {$2$} [ ] at 322 95
 \pinlabel {$1$} [ ] at 308 85
\endlabellist
\centerline{\includegraphics{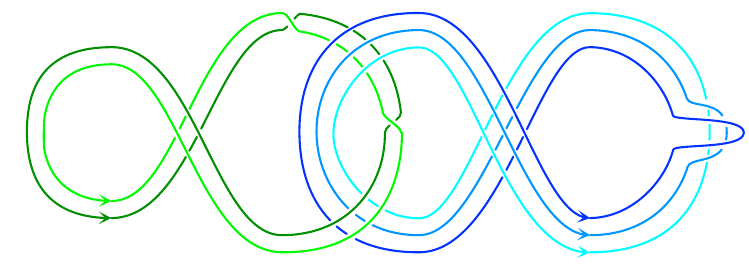}}
\caption{A Lagrangian diagram of the $(2,3)$-copy $\leg^{2,3}$ of a $2$-component link $\leg$. }
\label{fig:s-r-copy}
\end{figure}

\subsubsection{Disks in the $(s,r)$-Copy}
\label{sssec:sr-copy-disks}

The analysis of thin and thick disks for the $(s,r)$-copy of a Legendrian link is similar to that for the $n$-copy, though the notation is more involved.  Given a choice of $s-1$ indices $1 \leq i_1 < \cdots < i_{s-1} \leq M$ and $r-1$ indices $1 \leq j_1 < \cdots < j_{r-1} \leq N$, we write the word $\mba \ul{w} \mbb = a_M \cdots a_1 \ul{w} b_1 \cdots b_N$ as
\begin{equation} \label{eq:s-r-word}
    \mba \ul{w} \mbb = \mba_{s}a_{i_{s-1}} \mba_{s-1} \cdots \mba_2 a_{i_1}\mba_{1} \ul{w} \mbb_{1}b_{j_1} \mbb_2 \cdots \mbb_{r-1}b_{j_{r-1}}\mbb_{r}.
\end{equation}
We say that the position index of $\ul{w}$ is $0$. As before, the word could come from the generators for a single Legendrian link, or from the generators for the $(s,r)$-copy. When working with the $(s,r)$-copy, we say a pair $(\ul{z},\mba \ul{w} \mbb)$ is in \dfn{(s,r)-copy form} if \begin{itemize}
    \item $z \in \reeb^{12|sr}$,
    \item $w \in \reeb^{12|11}$,
    \item $a_{i_k} \in \reeb^{11|k+1,k}$ and $\mba_k$ are words in the alphabet $\gens^{11|kk}$, and
    \item $b_{j_k} \in \reeb^{22|k,k+1}$ and $\mbb_k$ are words in the alphabet $\gens^{22|kk}$. 
\end{itemize}
We underline $z$ and $w$ in the pair to emphasize the location of the mixed chords, i.e.\ chords in $\gens^{ij|kl}$ for $i \neq j$. Note that all generators in $\mba$ come from chords or basepoints of $\leg_1$, while all generators in $\mbb$ come from $\leg_2$.  As can be seen in Figure~\ref{fig:s-r-disks}, disks in $\Delta_{\nleg{s,r}}(\ul{z},\mba \ul{w} \mbb)$ have a positive corner at a mixed chord $z$, a negative corner at a mixed chord $w$ sandwiched by a collection of negative corners lying in $\nleg{s}_1$, $s-1$ of which are mixed relative to $\nleg{s}_1$, and a similar collection for $\nleg{r}_2$. 

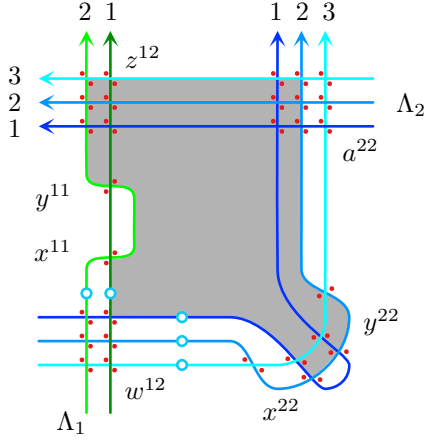
\begin{figure}
\labellist
\small\hair 2pt
 \pinlabel {$\leg_1$} [ ] at 22 5
 \pinlabel {$\leg_2$} [ ] at 150 125
 \pinlabel {$2$} [b] at 27 155
 \pinlabel {$1$} [b] at 37 155
 \pinlabel {$1$} [r] at 5 116
 \pinlabel {$2$} [r] at 5 126
 \pinlabel {$3$} [r] at 5 136
 \pinlabel {$1$} [b] at 99 155
 \pinlabel {$2$} [b] at 109 155
 \pinlabel {$3$} [b] at 119 155
 \pinlabel {$z^{12}$} [l] at 40 143
 \pinlabel {$w^{12}$} [l] at 40 18
 \pinlabel {$x^{11}$} [r] at 24 70
 \pinlabel {$y^{11}$} [r] at 24 90
 \pinlabel {$a^{22}$} [l] at 122 108
 \pinlabel {$x^{22}$} [r] at 110 10
 \pinlabel {$y^{22}$} [l] at 130 43
\endlabellist
\centerline{\includegraphics{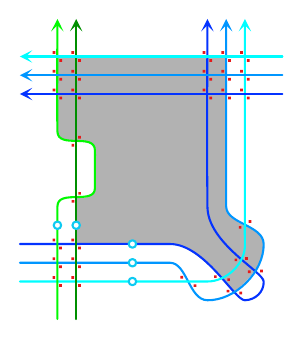}}
\caption{A thick disk in $(s,r)$-copy form in the set $\Delta_\leg(\ul{z}^{12|23}, x^{11|21} (t^{1|1})^{-1} \ul{w}^{12|11} t^{2|1} x^{22|12} a^{22|23})$.}
\label{fig:s-r-disks}
\end{figure}

Rather than identifying thin and thick disks at this stage using flowlines and enriched disks, we shall describe the process in detail only when the link in question is the $2$-copy and separated $2$-copy of a Legendrian knot.  That said, there is no obstruction to working with enriched disks for Legendrian links.

\subsection{The 2-Copy}
\label{ssec:2-copy}

Two special cases of the $(s,r)$-copy construction  will be important for the definition and analysis of the ``augmentation bimodules'' in Section~\ref{sec:aug-cat}. The first arises when the underlying link is the $2$-copy of a Legendrian knot $\leg$, ordered either from top to bottom (we call this the \dfn{$2$-copy} $\tcleg$) or bottom to top (the \dfn{opposite $2$-copy} $\otcleg$). The second special case, which we will encounter in the next subsection, involves separating the two copies by a large distance in the $z$ direction. In both cases, there is an additional ``stick-together'' map from the two-component link to a Legendrian knot, allowing us to identify disks in the $(s,r)$-copy with disks in the \emph{knot}, not just the $2$-copy \emph{link}. 

\subsubsection{The 2-Copy Construction}
\label{sssec:2-copy}

For each crossing $a$ of $\leg$, there are four types of crossings in the $(s,r)$-copy of the $2$-copy $\tcleg =  \tcleg_1 \sqcup \tcleg_2$:  those of the form $a^{12|kl}$ (which, again, we call mixed chords), those of the form $a^{21|kl}$ (which we will ignore), those of the form $a^{11|kl}$, and those of the form $a^{22|kl}$ (both of which we refer to as pure chords). Similarly, there are two kinds of crossings induced by critical points: those produced by perturbing the original $2$-copy of $\leg$ and those produced by perturbing the $(s,r)$-copy of $\tcleg$. The former type yields mixed Reeb chords in $\overline{\reeb}^{12|kl}$ (for any $i,j$) and the latter type yields pure Reeb chords in $\overline{\reeb}^{11|kl}$ (for $l<k$) and $\overline{\reeb}^{22|kl}$ (for $k<l$).  We denote both types by $x$ and $y$. 

From this point on, we assume that $\tcleg$ is simply perturbed and its $(s,r)$-copy is perturbed by a pair of Morse functions with the maxima and minima placed as shown in Figure~\ref{fig:2-copy-crit}. We call this a \dfn{simply perturbed $(s,r)$-copy of $\tcleg$}.  In this configuration, the combinatorics of the Lagrangian diagram of $\tcleg^{s,r}$ matches that of the diagram of a simply perturbed $\nleg{s+r}$, up to relabeling.

\begin{figure}
\labellist
\small\hair 2pt
 \pinlabel {$\leg$} [t] at 45 10
 \pinlabel {$\tcleg$} [t] at 170 10
 \pinlabel {$\tcleg^{s,r}$} [t] at 300 10
 \pinlabel {Max} [r] at 47 52
 \pinlabel {Min} [l] at 60 52
 \pinlabel {Maxima} [r] at 147 49
 \pinlabel {Minima} [l] at 185 49
 \pinlabel {$1$} [r] at 240 100
 \pinlabel {$2$} [r] at 240 110
 \pinlabel {$1$} [r] at 240 82
 \pinlabel {$2$} [r] at 240 72
 \pinlabel {$3$} [r] at 240 62
\endlabellist
    \centering
    \includegraphics[width=4.9in]{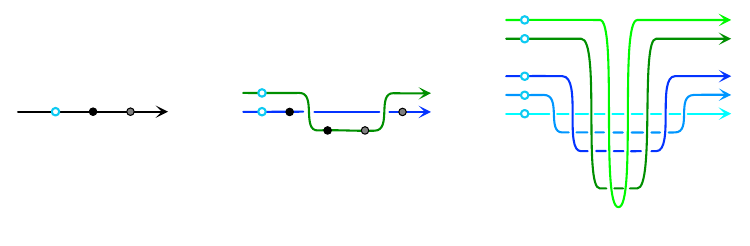}
    \caption{The construction of a simply perturbed $(s,r)$-copy of the $2$-copy near the critical points.}
    \label{fig:2-copy-crit}
\end{figure}
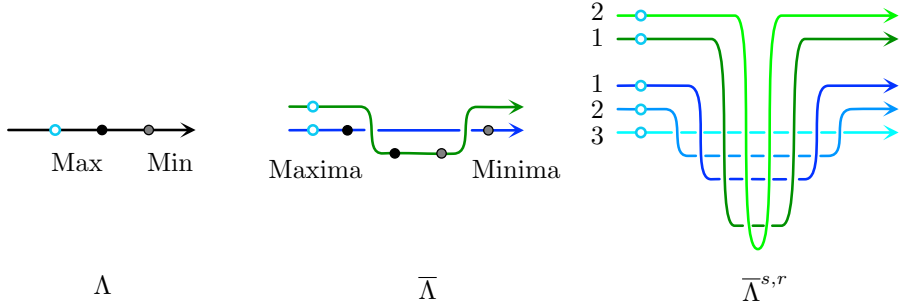

The description of the Reeb chords for the opposite $2$-copy $\otcleg$ is similar, though the critical points of the perturbing function induce generators in $\ul{\reeb}^{21|kl}$ rather than $\ul{\reeb}^{12|kl}$.  This is because in the opposite $2$-copy, the bottom copy is labeled $1$, and there are no Reeb chords that go from the top to the bottom copy at a critical point.

\subsubsection{Disks in the $2$-Copy}
\label{sssec:2-copy-disks}

Since either $2$-copy can be further retracted onto the original knot, disks in $\tcleg^{s,r}$ and $\otcleg^{s,r}$ can be classified thin or thick with respect to the original Legendrian $\leg$. As before, we can classify thin and thick disks in the $(s,r)$-copy of the two $2$-copies.
 
\begin{lem}[$2$-copy Thin Disk Identification] 
\label{lem:s-r-2-copy-thin-disk}
    Suppose $(\ul{z}, \mba \ul{w} \mbb)$ is in $(s,r)$-copy form for $\tcleg$.   The set of disks $\Delta_{\tcleg^{s,r}}(\ul{z}, \mba \ul{w}\mbb)$ consists of thin disks precisely when $(\ul{z}, \mba \ul{w} \mbb)$ is in of one of the following forms:
    \begin{description}
        \item[Flowlines] 
        \begin{subequations}
        \begin{align}
         &(\ul{x}^{12|11}, t^{-1} \ul{y}^{12|11} t) && (\ul{x}^{12|11}, \ul{y}^{12|11}) \\
        &(\ul{x}^{12|21}, t^{-1}y^{11|21} t \ul{x}^{12|11}) && (\ul{x}^{12|12}, \ul{x}^{12|11} y^{22|12}) \\
        &(\ul{x}^{12|12}, t^{-1} \ul{y}^{12|11} x^{12|11}) &&(\ul{x}^{12|21}, x^{11|21} \ul{y}^{12|11}).
        \end{align}
        \end{subequations}
        
        \item[Partial Flowlines] 
        \begin{subequations}
        \begin{align}
        &(\ul{a}^{12|12}, \ul{a}^{12|11} y^{22|12}) && (\ul{a}^{12|21}, y^{11|21} \ul{a}^{12|11}) \\
        &(\ul{a}^{12|11}, a^{11|11} \ul{y}^{12|11}) && (\ul{a}^{12|11}, \ul{y}^{12|11} a^{22|11}) \label{eq:2-thin-partial-2}\\
         &(\ul{a}^{12|21}, a^{11|21} \ul{y}^{12|11}) && (\ul{a}^{12|12}, \ul{y}^{12|11} a^{22|12}). \label{eq:2-thin-partial-3}
           \end{align}
        \end{subequations}
        \item[Constant Flowlines] $(\ul{y}^{12|21}, y^{11|21} \ul{y}^{12|11})$ or $(\ul{y}^{12|12}, \ul{y}^{12|11} y^{22|12}$).
    \end{description}
    In particular, a disk in $\Delta_{\tcleg^{s,r}}(\ul{z},\mba \ul{w} \mbb)$ is thin with respect to $\leg$ if and only if $\mba$, $w$, or $\mbb$ contains a $y$ generator.

    Similarly, the moduli spaces consist of thin disks in one of the following forms, both of which are partial flowlines:
\[(\ul{a}^{12|12},\ul{a}^{12|11}y^{22|12}),\quad (\ul{a}^{12|21},y^{11|21} \ul{a}^{12|11}).\]
    In particular, 
    if $(\ul{z}, \mba \ul{w} \mbb)$ is in $(s,r)$-copy form for $\otcleg$, a disk in $\Delta_{\otcleg^{s,r}}(\ul{z},\mba \ul{w} \mbb)$ is thin with respect to $\leg$ if and only if  $\mba$ or $\mbb$ contains a $y$ generator. 
\end{lem}

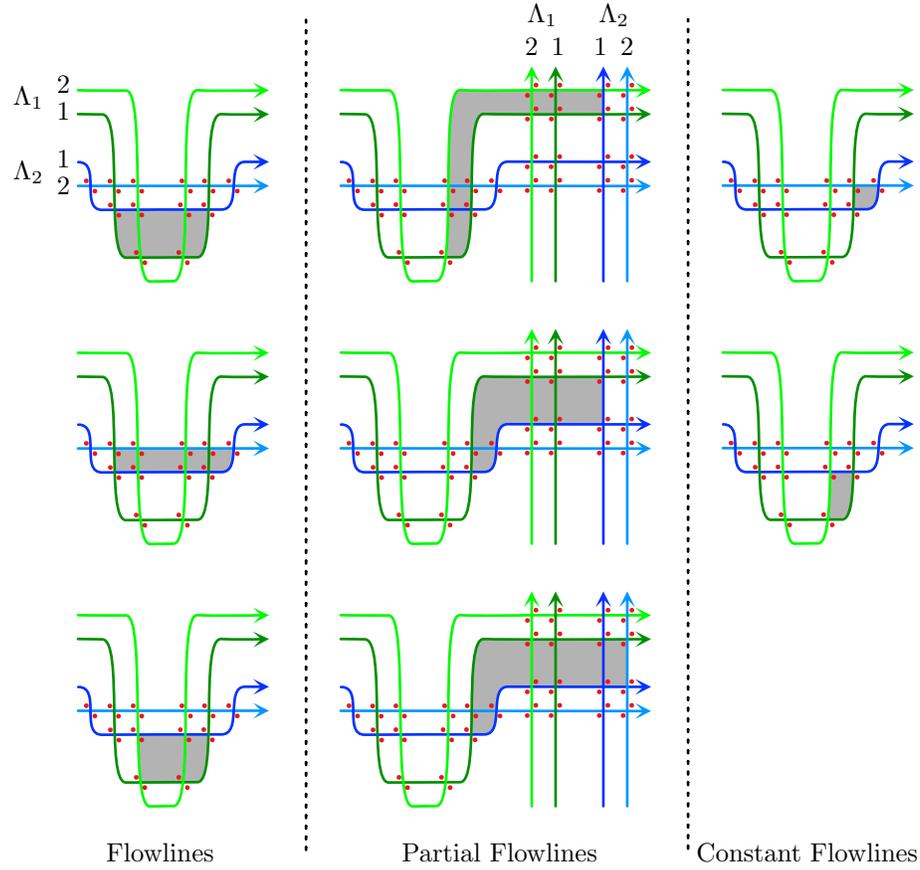
\begin{figure}
\labellist
\small\hair 2pt
 \pinlabel {$\leg_1$} [r] at 7 286
 \pinlabel {$\leg_2$} [r] at 7 258
 \pinlabel {$2$} [r] at 17 291
 \pinlabel {$1$} [r] at 17 281
 \pinlabel {$1$} [r] at 17 263
 \pinlabel {$2$} [r] at 17 253
 \pinlabel {$\leg_1$} [b] at 193 311
 \pinlabel {$\leg_2$} [b] at 220 311
 \pinlabel {$2$} [b] at 189 300
 \pinlabel {$1$} [b] at 199 300
 \pinlabel {$1$} [b] at 215 300
 \pinlabel {$2$} [b] at 225 300
 \pinlabel {Flowlines} [t] at 49 7
 \pinlabel {Partial Flowlines} [t] at 177 7
 \pinlabel {Constant Flowlines} [t] at 293 7
\endlabellist
    \centering
    \includegraphics{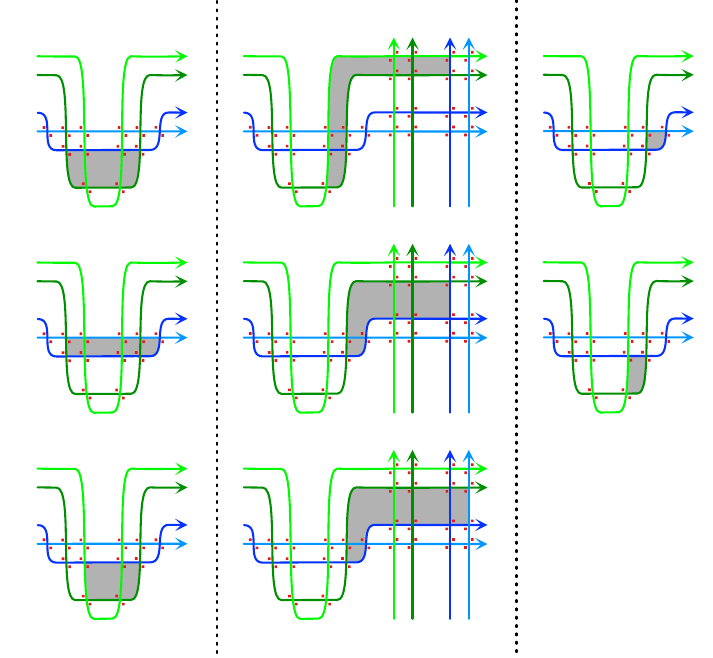}
    \caption{The thin disks in the $2$-copy listed in the second column of Lemma~\ref{lem:s-r-2-copy-thin-disk}. The disks listed in the first column are similar, but pass around larger portions of the knot as in Figure~\ref{fig:n-copy-thin-disks}.  }
    \label{fig:2-copy-thin}
\end{figure}

The proof of the lemma comes down to direct combinatorial enumeration, half of which is carried out in Figure~\ref{fig:2-copy-thin}.

\begin{rem}
    Notice that there are more thin disks for the $(s,r)$-copy of $\tcleg$ than for the $n$-copy.  A thin disk in the $n$-copy with more than one negative corner corresponds to two thin disks in the $(s,r)$-copy of $\tcleg$, depending on which of the two corners are pure and which are mixed.  These will contribute to the $A_\infty$ bimodule maps in different ways in Section~\ref{sec:aug-cat}.
\end{rem}

Next, we use enriched disks to identify the thick disks in an $(s,r)$-copy of the $2$-copy with disks in the original Legendrian.  First, we need to slightly generalize the notion of an enriched disk to incorporate enrichments on each component.  Recall the notation in Equation~\eqref{eq:s-r-word} for $\mba \ul{w} \mbb$.

\begin{defn} \label{defn:doubly-enriched}
    An \dfn{$(m,n)$-doubly enriched disk} is a triple $(u, \mbw_1, \mbw_2)$ consisting of a disk $u \in \Delta_\leg(\ul{z}, \mba \ul{w} \mbb)$ with $w \in \reeb \cup \base$ and two multisets of indices $\mbw_1$ and $\mbw_2$, termed the \dfn{pure enrichments}, so that
    \begin{itemize}
        \item $\mbw_1$ has total multiplicity $m$ and draws elements from $\{0, \ldots, M\}$, while
        \item $\mbw_2$ has total multiplicity $n$ and draws elements from $\{0, \ldots, N\}$.
    \end{itemize}
    The index $0$ of $\ul{w}$ carries an additional enrichment with multiplicity $1$; we term this the \dfn{mixed enrichment}. An index $i$ may have multiplicity greater than $1$ in $\mbw_j$  only if the corresponding generator is $t^{-1}$.
    
    The set of $(m,n)$-enriched disks with corners $(\ul{z},\mba \ul{w} \mbb)$ is denoted by $\Delta^{m,n}_\leg(\ul{z},\mba \ul{w} \mbb$).
\end{defn}

Before describing the projection and lifting operations, we modify how the algebraic projection map on a word in $(s,r)$-copy form for the $2$-copy interacts with the splitting of the word into $\mba$, $\ul{w}$, and $\mbb$. Examining the word $\check{\mba} \check{\ul{w}} \check{\mbb}$, we see that if $w$ arises from a Reeb chord of $\leg$, then the stuck-together word has the same tripartite structure as the original word.  On the other hand, if $w = x^{12|11}$, then $\check{\mba} \check{\ul{w}} \check{\mbb}$ no longer appears to have an underlined generator. In this case, a word arising from an immersed disk in $(s,r)$-copy form must either have a $t$ at the end of $\check{\mba}$ or $t^{-1}$ at the beginning of $\check{\mbb}$.  We simply remove that $t^{\pm 1}$ from one of the side words and underline it.

In parallel to Section~\ref{sssec:n-copy-disks}, we describe how to project a thick disk $u$ in $(s,r)$-copy form for $\tcleg$ to a doubly enriched disk for $\leg$.  We first apply the stick-together map twice, once to pass from $\tcleg^{s,r}$ to $\tcleg$, and again to reach $\leg$.  Composition with this double stick-together map yields an immersed disk with convex corners in $\leg$.  As before, we need only understand the enrichment process.  The pure enrichments $\mbw_1(u)$ and $\mbw_2(u)$ come from the positions of the $a_{i_k}$ in $\mba$ and $b_{j_l}$ in $\mbb$, while the mixed enrichment is at the central position $0$; see Figure~\ref{fig:2-copy-thick-1}. One further adjustment may be necessary: if the underlined generator of $\check{\alpha}(\mba \ul{w} \mbb)$ is $t^{-1}$, as in Figure~\ref{fig:2-copy-thick-2}, then we also enrich the position $0$ in $\mbw_1$ with the number of $x^{11|\cdot \cdot}$ generators that precede $\ul{w}$ in $\mba \ul{w} \mbb$ and enrich the position $0$ in $\mbw_2$ with the number of $x^{22|\cdot \cdot}$ generators that succeed $\ul{w}$.

Lifting a doubly enriched disk $(u,\mbw_1, \mbw_2)$ to the $(s,r)$-copy of $\tcleg$ is a straightforward generalization of lifting an enriched disk to the $n$-copy.  As before, positions that are purely enriched in $\mbw_i$ lift to Reeb chords of the $s$- or $r$-copy of $\tcleg_i$.  The mixed enriched position lifts to $w^{12|11}$ if $w \in \reeb$, and to $x^{12|11}$ if $w = t^{\pm 1}$. Note that if $0$ is enriched in either $\mbw_1$ or $\mbw_2$, then the lifting process will yield $x$ chords on either side of $x^{12|11}$.

As before, it is straightforward to check that the projection and enrichment processes are inverses of each other, and we obtain the following lemma.

\begin{lem}[$2$-copy Thick Disk Identification] 
\label{lem:s-r-2-copy-dil}
    Suppose $(\ul{z},\mba \ul{w} \mbb)$ is in $(s,r)$-copy form for either $\tcleg$ or $\otcleg$ without any generator at a minimum of a perturbing function.  There is a bijective correspondence between thick disks $u \in \Delta_{\tcleg^{s,r}}(\ul{z},\mba \ul{w} \mbb)$ and doubly enriched disks $(s\circ u, \mbw_1(u),\mbw_2(u)) \in \Delta^{s,r}_\leg(\ul{\check{z}}, \check{\mba} \ul{\check{w}} \check{\mbb})$, where $0$ has pure enrichments only if the mixed enrichment is at $t^{-1}$. The same holds true for the opposite $2$-copy, though without any complications coming from $w = x^{12|11}$.
\end{lem}

We emphasize that for the $2$-copy, the Reeb chords $z$ and $w$ start at the bottom copy and end at the top copy, while the opposite holds for the opposite $2$-copy, even though the notation for the chords is forced to be identical.  

\begin{figure}
\labellist
\small\hair 2pt
 \pinlabel {$z$} [br] at 20 115
 \pinlabel {$a$} [tr] at 20 48
 \pinlabel {$w$} [bl] at 122 115
 \pinlabel {$z^{12}$} [tr] at 168 104
 \pinlabel {$z^{11}$} [br] at 168 148
 \pinlabel {$z^{21}$} [bl] at 212 148
 \pinlabel {$z^{22}$} [tl] at 212 104
 \pinlabel {$a^{11}$} [tr] at 168 15
 \pinlabel {$a^{12}$} [br] at 168 60
 \pinlabel {$a^{22}$} [bl] at 212 60
 \pinlabel {$a^{21}$} [tl] at 212 15
 \pinlabel {$w^{21}$} [br] at 294 148
 \pinlabel {$w^{22}$} [tr] at 294 104
 \pinlabel {$w^{12}$} [tl] at 337 104
 \pinlabel {$w^{11}$} [bl] at 337 148
\endlabellist
    \centering
    \includegraphics{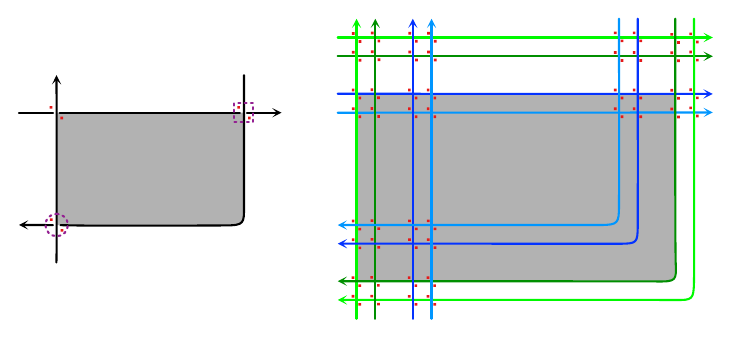}
    \caption{The disk at left is an enriched disk in $\Delta_\leg^2(\ul{z}, a\ul{w})$ with pure enrichments $\mbw_1 = \{1\}$ and $\mbw_2=\emptyset$.  The disk at right is its lift to a disk in $\Delta_{\tcleg^{2,2}}(\ul{z}^{12|21}, a^{11|21}\ul{w}^{12|11})$, following Lemma~\ref{lem:s-r-2-copy-dil}.}
    \label{fig:2-copy-thick-1}
\end{figure}

\begin{figure}
\labellist
\small\hair 2pt
 \pinlabel {$z$} [br] at 21 105
 \pinlabel {$a$} [bl] at 123 106
 \pinlabel {$z^{11}$} [br] at 170 142
 \pinlabel {$z^{12}$} [tr] at 170 94
 \pinlabel {$z^{22}$} [tl] at 209 94
 \pinlabel {$z^{21}$} [bl] at 209 142
 \pinlabel {$a^{21}$} [br] at 295 142
 \pinlabel {$a^{22}$} [tr] at 295 94
 \pinlabel {$a^{12}$} [tl] at 337 94
 \pinlabel {$a^{11}$} [bl] at 337 142
 \pinlabel {$x^{12}$} [tr] at 238 77
 \pinlabel {$y^{12}$} [tl] at 266 77
\endlabellist
    \centering
    \includegraphics{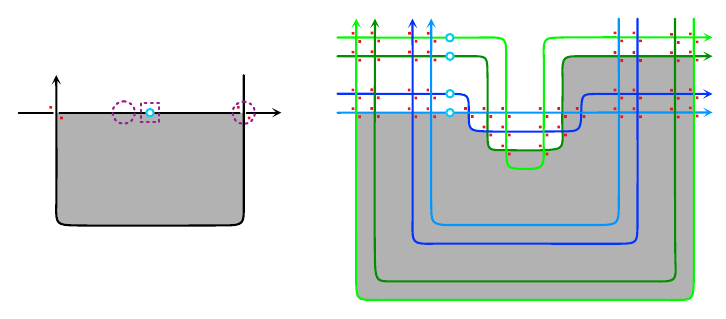}
    \caption{The disk at left is an enriched disk in $\Delta_\leg^2(\ul{z}, a \ul{t}^{-1})$ with pure enrichments $\mbw_1 = \{1\}$ and $\mbw_2 = \{0\}$.  The disk at right is its lift to a disk in $\Delta_{\tcleg^{2,2}}(\ul{z}^{12|22}, a^{11|21}\ul{x}^{12|11} x^{22|12} (t^{2|2})^{-1})$.}
    \label{fig:2-copy-thick-2}
\end{figure}

\subsection{The Separated 2-Copy}
\label{ssec:sep-2-copy}

The second important special case of the $(s,r)$-copy construction occurs when the underlying link is the the \emph{separated} $2$-copy of a Legendrian knot $\leg$. 

\subsubsection{The Separated $2$-Copy Construction}
\label{sssec:sep-2-copy}

Given a Legendrian link $\leg$, we form its \dfn{separated $2$-copy} $\stcleg$ by taking the 2-copy $\tcleg$ and pushing the bottom copy in the negative Reeb ($z$) direction by a distance $Z$ larger than the maximum length of the Reeb chords of $\leg$.  See Figure~\ref{fig:sep-2-copy}, which should be compared and contrasted with the standard $2$-copy in Figure~\ref{fig:n-copy}. We distinguish the Chekanov-Eliashberg DGA of $\stcleg$ with that of the standard $2$-copy $\tcleg$ by decorating related notation with a hat.

\begin{figure}[htbp]
   \labellist
\small\hair 2pt
 \pinlabel {$1$} [br] at 20 105
 \pinlabel {$2$} [tl] at 39 86
 \pinlabel {$a^{22}$} [r] at 102 64
 \pinlabel {$a^{11}$} [l] at 141 64
 \pinlabel {$q^{12}$} [b] at 122 94
 \pinlabel {$p^{12}$} [t] at 122 34
 \pinlabel {$x^{12}$} [tl] at 245 57
 \pinlabel {$y^{12}$} [tl] at 243 84
\endlabellist \centerline{\includegraphics{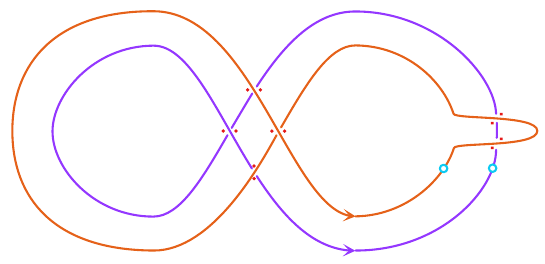}}
    \caption{The separated $2$-copy of a Legendrian knot. }
    \label{fig:sep-2-copy}
\end{figure}

For $i = 1,2$, pure chords and base points in $\widehat{\gens}^{ii}$ are in bijection with Reeb chords and base points of the original Legendrian $\leg$, and thus in bijection with pure chords and base points $\overline{\gens}^{ii}$ of $\tcleg$. We label such chords by $a^{ii|kl}$ or $t^{i|k}$. All mixed chords in the separated $2$-copy begin on the component $\stcleg_2$ and end on $\stcleg_1$. That is, all mixed chords lie in $\widehat{\reeb}^{12}$, and $\widehat{\reeb}^{21} = \emptyset$. Following \cite{duality}, we split chords of the separated $2$-copy into two classes based on a correspondence with chords of $\tcleg$; our formulation matches that in \cite[Proposition 5.4]{nrssz:aug-sheaf}.

\begin{prop}\label{prop:p-q-chords}
    There is a bijective correspondence between Reeb chords of $\widehat{\reeb}^{12}$ (resp. $\widehat{\gens}^{ii}$, for $i = 1,2$) and those of $\overline{\reeb}^{12} \cup \overline{\reeb}^{21}$  (resp. $\overline{\gens}^{ii}$). 
\end{prop}

Under this correspondence, Reeb chords in $\overline{\reeb}^{21}$ are mapped to those in $\widehat{\reeb}^{12}$ with length smaller than $Z$ and chords in $\overline{\reeb}^{12}$ are mapped to those in $\widehat{\reeb}^{12}$ with length larger than $Z$. Following \cite{duality}, we call the former type of chords \dfn{$p$ chords} and the latter type \dfn{$q$ chords}, denoted by $\widehat{\reeb}_p$ and $\widehat{\reeb}_q$, respectively. Note that the chords arising from the critical points of the perturbation are $q$ chords. We also use these letters to distinguish these two kinds of chords when necessary, though we continue to label the $q$ chords at critical points by $x$ and $y$.  The gradings of elements in $\widehat{\gens}^{ij}$ align with those in $\overline{\gens}^{ij}$ of the standard $2$-copy except that $|p_k| = -|a_k^{21}|-1$; again, see \cite{duality}.

\subsubsection{Disks in the Separated $2$-Copy}
\label{sssec:sep-2-copy-disks}

When analyzing immersed disks in $\stcleg$, we extend concepts such as the stick-together map, thin and thick disks, and enriched disks from the standard $n$-copy to the separated $2$-copy.  The set of immersed disks in the separated $2$-copy is noticeably more complicated than in the settings above. Again, we shall always assume that the separated $2$-copy and its $(s,r)$-copy are simply perturbed in the same sense as before. We begin by identifying thin disks of three different types. 

\begin{lem}[Separated $2$-Copy Thin Disk Identification I] 
\label{lem:s-r-sep-thin-disk-1} 
    Suppose that $(\ul{z},\mba \ul{w} \mbb)$ is in $(s,r)$-copy form for $\stcleg$. A disk in the set $\Delta_{\stcleg^{s,r}}(\ul{z},\mba \ul{w} \mbb)$ with both $z$ and $w$ in $\widehat{\reeb}_q$ is thin with respect to $\leg$ if and only if $(\ul{z},\mba \ul{w} \mbb)$ is one of the following types: 
    \begin{description}
        \item[Flowlines] 
        \[\begin{array}{ll} (\ul{x}^{12|11}, t^{-1} \ul{y}^{12|11} t), & (\ul{x}^{12|11}, \ul{y}^{12|11}), \\
        (\ul{x}^{12|21}, t^{-1}y^{11|21} t \ul{x}^{12|11}), & (\ul{x}^{12|12}, \ul{x}^{12|11} y^{22|12}), \\
        (\ul{x}^{12|12}, t^{-1} \ul{y}^{12|11} x^{12|11}), & (\ul{x}^{12|21}, x^{11|21} \ul{y}^{12|11}),
        \end{array}\]
        
        \item[Partial Flowlines] 
         \[\begin{array}{ll} (\ul{q}^{12|12}, \ul{q}^{12|11} y^{22|12}), & (\ul{q}^{12|21}, y^{11|21} \ul{q}^{12|11}), \\
        (\ul{q}^{12|11}, a^{11|11} \ul{y}^{12|11}), & (\ul{q}^{12|11}, \ul{y}^{12|11} a^{22|11}), \\
         (\ul{q}^{12|21}, a^{11|21} \ul{y}^{12|11}), & (\ul{q}^{12|12}, \ul{y}^{12|11} a^{22|12}),
         \end{array}\]
      
        \item[Constant Flowlines] $(\ul{y}^{12|21}, y^{11|21} \ul{y}^{12|11})$ or $(\ul{y}^{12|12}, \ul{y}^{12|11} y^{22|12}$).
    \end{description}
\end{lem}

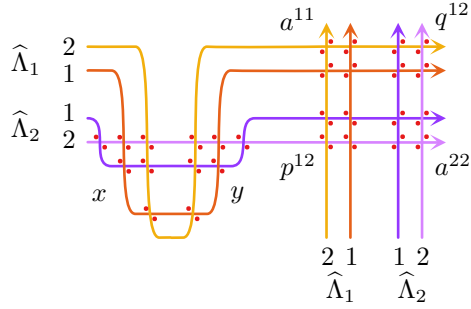
\begin{figure}
\labellist
\small\hair 2pt
 \pinlabel {$\stcleg_1$} [r] at 3 85
 \pinlabel {$\stcleg_2$} [r] at 3 60
 \pinlabel {$2$} [r] at 16 90
 \pinlabel {$1$} [r] at 16 80
 \pinlabel {$1$} [r] at 16 65
 \pinlabel {$2$} [r] at 16 55
 \pinlabel {$x$} [ ] at 23 34
 \pinlabel {$y$} [ ] at 75 34
 \pinlabel {$\stcleg_1$} [t] at 114 7
 \pinlabel {$\stcleg_2$} [t] at 141 7
 \pinlabel {$2$} [t] at 109 16
 \pinlabel {$1$} [t] at 118 16
 \pinlabel {$1$} [t] at 136 16
 \pinlabel {$2$} [t] at 145 16
 \pinlabel {$p^{12}$} [tr] at 106 51
 \pinlabel {$q^{12}$} [bl] at 148 95
 \pinlabel {$a^{11}$} [br] at 106 95
 \pinlabel {$a^{22}$} [tl] at 148 51
\endlabellist
    \centering
    \includegraphics{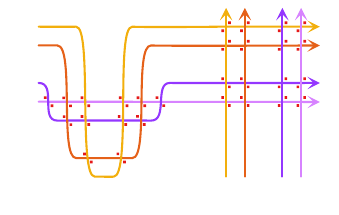}
    \caption{In Lemma~\ref{lem:s-r-sep-thin-disk-1}, the configuration of the crossing signs for the $x$, $y$, and $q$ generators is the same as those for the $x$, $y$, and $a^{ij}$ ($i \geq j$) generators in Figure~\ref{fig:2-copy-thin}.}
    \label{fig:sep-2-copy-qq}
\end{figure}

\begin{proof}
    The thin disks for the separated $2$-copy with both $z$ and $w$ in $\widehat{\reeb}_q$ are combinatorially identical to those in the ordinary $2$-copy.  To see this, first note that the configuration of the crossing signs for the flowlines and constant flowlines in Figure~\ref{fig:2-copy-crit} is identical for the separated $2$-copy; see Figure~\ref{fig:sep-2-copy-qq}.  Next, the partial flowlines in Figure~\ref{fig:2-copy-crit} only have corners at the $a^{12}$ crossings, which are configured identically to the $q^{12}$ crossings in the separated $2$-copy. Thus, the combintorial count in the proof of Lemma~\ref{lem:s-r-2-copy-thin-disk} applies here with appropriate changes in notation.
\end{proof}

\begin{lem}[Separated $2$-Copy Thin Disk Identification II] \label{lem:s-r-sep-thin-disk-2}
    Suppose that $(\ul{z},\mba \ul{w} \mbb)$ is in $(s,r)$-copy form for $\stcleg$. A disk in the set $\Delta_{\stcleg^{s,r}}(\ul{z},\mba \ul{w} \mbb)$ with both $z$ and $w$ in $\widehat{\reeb}_p$ is thin with respect to $\leg$ if and only if $(\ul{z},\mba \ul{w} \mbb)$ is one of the following types:
    \[(\ul{p}^{12|21}, y^{11|21} \ul{p}^{12|11}) \text{ or } (\ul{p}^{12|12},  \ul{p}^{12|11} y^{22|12}).\]
\end{lem}

\begin{proof}
    As shown in Figure~\ref{fig:sep-2-copy-qq}, the thin disks for the separated $2$-copy with both $z$ and $w$ in $\widehat{\reeb}_p$ are combinatorially identical to those in the $(s,r)$-copy of an opposite $2$-copy in Lemma~\ref{lem:s-r-2-copy-thin-disk}.
\end{proof}

\begin{rem} \label{rmk:strict-unital-dil}
    Aside from the disks identified in Lemmas~\ref{lem:s-r-sep-thin-disk-1} and \ref{lem:s-r-sep-thin-disk-2}, Figure~\ref{fig:sep-2-copy-qq} shows that no other disk --- thick or thin --- can have a negative corner at a $y$ generator.
\end{rem}

The final set of thin disks for the separated $2$-copy that we will need are those that have a positive corner at an $x$ generator and a negative corner at a $p$ generator.  

\begin{lem}[Separated $2$-Copy Thin Disk Identification III] \label{lem:s-r-sep-thin-disk-3}
    Suppose that $(\ul{z},\mba \ul{w} \mbb)$ is in $(s,r)$-copy form for $\stcleg$. All disks in the set $\Delta_{\stcleg^{s,r}}(\ul{z},\mba \ul{w} \mbb)$ with $z \in \widehat{\reeb}_x$ and $w$ in $\widehat{\reeb}_p$ are thin with respect to $\leg$. In particular, $(\ul{z},\mba \ul{w} \mbb)$ is one of the following half-flowlines.  

    \begin{subequations}
    \begin{align}
        &(\ul{x}^{12|11}, t^{-1} a^{11|11} \ul{p}^{12|11} t) && (\ul{x}^{12|11}, t^{-1} \ul{p}^{12|11} a^{22|11}t) \label{eq:sep-thin-1} \\
        &(\ul{x}^{12|21}, t^{-1} a^{11|21} \ul{p}^{12|11} t) && (\ul{x}^{12|12}, t^{-1} \ul{p}^{12|11} a^{22|12} t) \label{eq:sep-thin-2}\\
        &(\ul{x}^{12|12}, t^{-1}a^{11|11} \ul{p}^{12|11} t x^{22|12}) && (\ul{x}^{12|12}, t^{-1}\ul{p}^{12|11} a^{22|11}t x^{22|12}) \label{eq:sep-thin-3}\\
        &(\ul{x}^{12|22}, t^{-1} a^{11|21} \ul{p}^{12|11} t x^{22|12}) && (\ul{x}^{12|13}, t^{-1} \ul{p}^{12|11} a^{22|12} t x^{22|23}). \label{eq:sep-thin-4}
    \end{align}
     \end{subequations}
\end{lem}

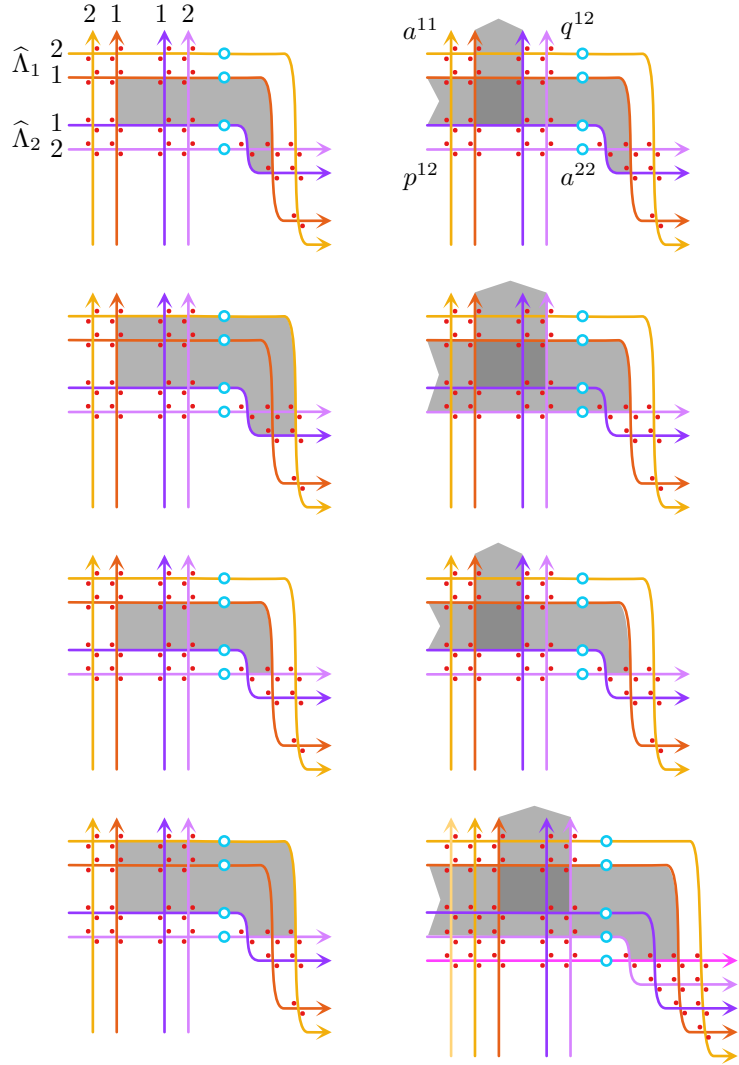
\begin{figure}
\labellist
\small\hair 2pt
 \pinlabel {$\stcleg_1$} [r] at 1 384
 \pinlabel {$\stcleg_2$} [r] at 1 356
 \pinlabel {$2$} [r] at 9 389
 \pinlabel {$1$} [r] at 9 379
 \pinlabel {$1$} [r] at 9 361
 \pinlabel {$2$} [r] at 9 351
 \pinlabel {$2$} [b] at 17 397
 \pinlabel {$1$} [b] at 27 397
 \pinlabel {$1$} [b] at 44 397
 \pinlabel {$2$} [b] at 54 397
 \pinlabel {$a^{11}$} [br] at 150 391
 \pinlabel {$p^{12}$} [tr] at 150 347
 \pinlabel {$q^{12}$} [bl] at 193 391
 \pinlabel {$a^{22}$} [tl] at 193 347
\endlabellist
    \centering
 \includegraphics{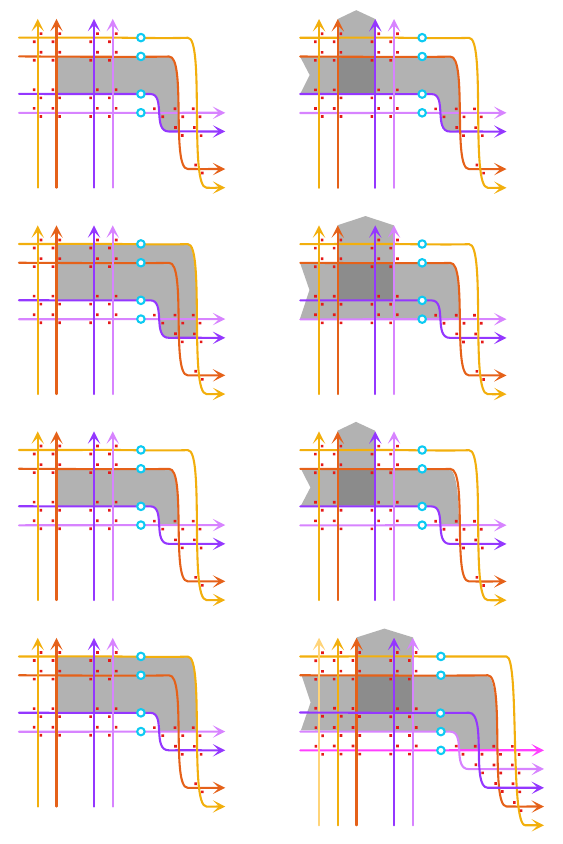}
    \caption{The thin disks in the $2$-copy listed in Lemma~\ref{lem:s-r-sep-thin-disk-3}. }
    \label{fig:sep-2-copy-xp}
\end{figure}

The proof is, once again, a straightforward combinatorial enumeration, which is carried out in Figure~\ref{fig:sep-2-copy-xp}.

There are many more thin disks with $z \in \widehat{\reeb}_q$ and $w \in \widehat{\reeb}_p$, including new configurations that lie between two crossings, or even lie inside a single crossing, but these will not need to be explicitly enumerated for the main proofs in Section~\ref{sec:main-pf}.

We are now ready to describe the thick disks.  The precise details of the projection and lifting operations depend on the nature of the $z$ and $w$ generators in $(\ul{z},\mba \ul{w} \mbb)$, divided into the following three cases.

\begin{figure}
\labellist
\small\hair 2pt
 \pinlabel {$1$} [ ] at 18 206
 \pinlabel {$2$} [ ] at 18 130
 \pinlabel {$3$} [ ] at 126 206
 \pinlabel {$q_1^{12}$} [br] at 171 290
 \pinlabel {$a_2^{11}$} [br] at 171 200
 \pinlabel {$q_3^{12}$} [br] at 325 290
 \pinlabel {$p_3^{12}$} [tr] at 296 134
 \pinlabel {$p_1^{12}$} [tr] at 197 134
 \pinlabel {$a_2^{22}$} [br] at 197 56
\endlabellist
    \centering
    \includegraphics{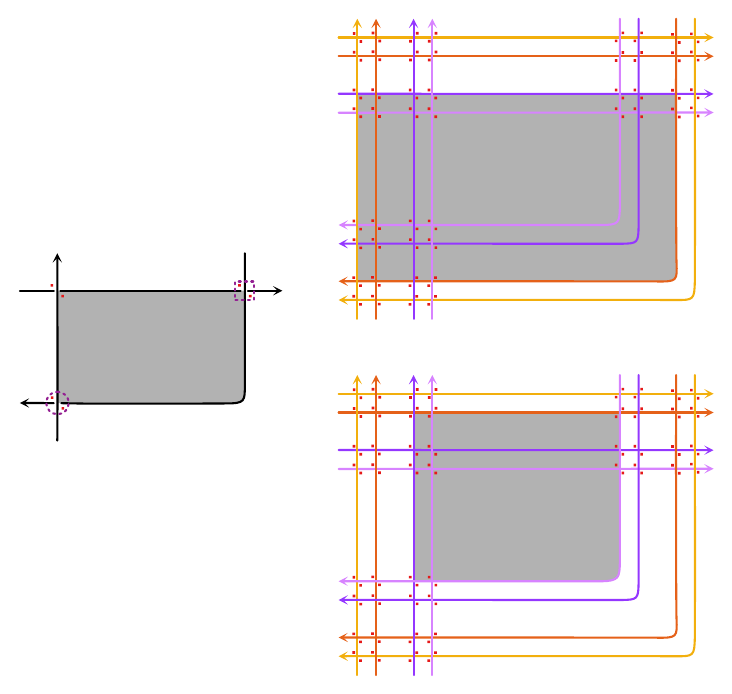}
    \caption{The thick disks captured by Lemma~\ref{lem:separated+dil}(1) on the top and (2) on the bottom come from the same underlying enriched disk in $\Delta_\leg^{1,0}(\ul{a}_1, a_2 \ul{a}_3)$ with pure enrichment $\mbw_1=\{1\}$ and $\mbw_2= \emptyset$.  At top, the enriched disk corresponds to a disk in $\Delta_{\stcleg^{2,2}}(\ul{q}_1^{12|21}, a_2^{12|21} \ul{q}_3^{12|11})$.  At bottom, the enriched disk corresponds to a disk  in $\Delta_{\stcleg^{2,2}}(\ul{p}_3^{12|12}, \ul{p}_1^{12|11} a_2^{22|12})$.}
    \label{fig:sep-copy-thick-1}
\end{figure}

\begin{description}
    \item[$z,w \in \widehat{\reeb}_q$] The projection and lifting operations are identical to those for the ordinary $2$-copy. See Figure~\ref{fig:sep-copy-thick-1}.   
    
    \item[$z \in \widehat{\reeb}_p$, $w \in \widehat{\reeb}_p$] The stick-together map sends a corner of a disk in the diagram of $\stcleg^{s,r}$ at a $p$ chord to a corner of a disk in the diagram of $\leg$ with the opposite sign. Thus, a disk in $\Delta_{\leg^{s,r}}(\ul{z}, \mba \ul{w} \mbb)$, which has a positive $p$-corner at $\check{z}$ and a negative $p$-corner at $\check{w}$, projects to a doubly enriched disk in $\Delta^{r,s}_\leg(\check{\ul{w}}, \check{\mbb} \check{\ul{z}} \check{\mba})$.  The pure enrichments follow the cyclic movement of $\mba$ and $\mbb$, and the mixed enrichment lies at $\ul{z}$. See Figure~\ref{fig:sep-copy-thick-1} once again.

    Once we specify that the target lifts of $z$ and $w$ lie in $\widehat{\reeb}_p$, the lifting procedure follows the same recipe as before and provides an inverse operation to the projection.
    
    \item[$z \in \widehat{\reeb}_y$, $w \in \widehat{\reeb}_p$]  While thick disks with a positive corner at a minimum did not exist for the ordinary $2$-copy, as dictated by Lemma~\ref{lem:s-r-2-copy-dil}, they are possible when the negative corner lies at a $p$ chord.  Similarly to the previous case, a disk in this case projects to a disk with a positive corner at $\check{w}$ and a marked point in the position occupied by $z$. That is, the projection lies in $\Delta^{r,s}(\check{\ul{w}}, \check{\mbb} \ul{t}^{\pm 1} \check{\mba})$. Care must be taken with the enrichment.  If the mixed enrichment is at $t$, then $0$ is enriched at most once in $\mbw_1$ and not enriched in $\mbw_2$.  If the mixed enrichment is at $t^{-1}$, then $0$ is not enriched in $\mbw_1$ but may be enriched multiple times in $\mbw_2$. See Figure~\ref{fig:sep-copy-thick-2}.

    Again, once we specify that the lifts of $z$ and $w$ lie in $\widehat{\reeb}_y$ and $\widehat{\reeb}_p$ respectively, the lifting procedure follows the same recipe as before and provides an inverse operation to the projection.
\end{description}

\begin{figure}
\labellist
\small\hair 2pt
 \pinlabel {$1$} [br] at 22 106
 \pinlabel {$2$} [bl] at 120 106
 \pinlabel {$y^{12}$} [tl] at 275 88
 \pinlabel {$p_1^{12}$} [tr] at 197 124
 \pinlabel {$a_2^{22}$} [tl] at 308 96
 \pinlabel {$x^{11}$} [tr] at 240 77
\endlabellist
    \centering
    \includegraphics{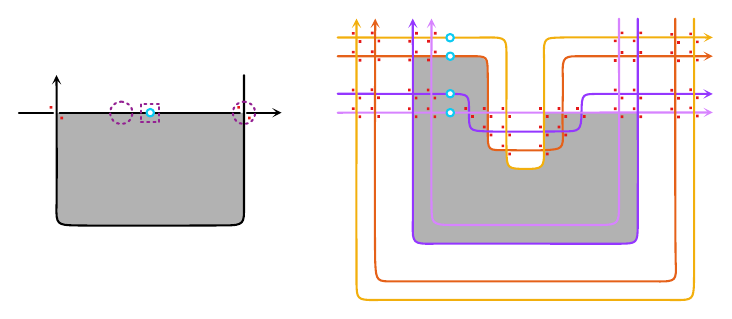}
    \caption{The thick disks captured by Lemma~\ref{lem:separated+dil}(3). The enriched disk at left lies in $\Delta_{\leg}^{1,0}(a_1, a_2\ul{t}^{-1})$ with pure enrichments $\mbw_1=\{1\}$ and $\mbw_2 = \{0\}$ corresponds to the disk at right in $\Delta_{\stcleg^{2,2}}(\ul{y}^{12|22}, x^{11|21} (t^{1|1})^{-1} \ul{p}_1^{12|11} a_2^{22|12})$.}
    \label{fig:sep-copy-thick-2}
\end{figure}
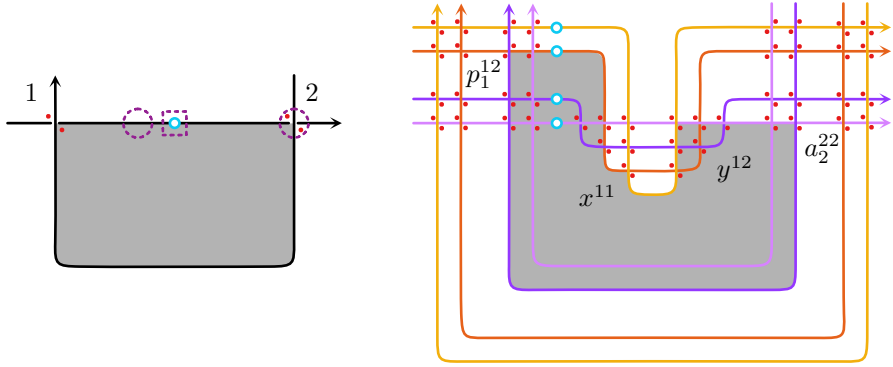

\begin{rem}\label{rem:qp-thick-disks}
    We need not formally identify the remaining thick disks with $z\in \widehat{\reeb}_q^{12|s,r}$ and $w \in \widehat{\reeb}_p^{12|s,r}$, though one can observe that they project to disks with precisely \emph{two} positive corners in $\leg$.  By Lemma~\ref{lem:s-r-sep-thin-disk-3}, there are no thick disks with $z \in \widehat{\reeb}_x$ and $w \in \widehat{\reeb}_p$. Further, action considerations imply that there are no disks in $\stcleg^{s,r}$ with $z\in \widehat{\reeb}_p^{12|s,r}$ and $w \in \widehat{\reeb}_q^{12|s,r}$.
\end{rem}

The descriptions of the projection and lifting operations above yield the following identification.

\begin{lem}[Separated $2$-Copy Thick Disk Identification] \label{lem:separated+dil}
    Suppose the pair $(\ul{z},\mba \ul{w} \mbb)$ is in $(s,r)$-copy form for $\stcleg$ without any generator in $\mba\ul{w}\mbb$ arising from a minimum of a perturbing function.

    \begin{enumerate}
        
    \item If $z\in \widehat{\reeb}_q^{12|s,r}$ and $w \in \widehat{\reeb}_q^{12|11}$, there is a bijective correspondence between the disks in $\Delta_{\stcleg^{s,r}}(\ul{z},\mba \ul{w} \mbb)$ and the doubly enriched disks in $\Delta^{s-1,r-1}_\leg(\ul{\check{z}}, \check{\mba} \ul{\check{w}} \check{\mbb})$, where $0$ has pure enrichments only if the mixed enrichment is at $t^{-1}$.

    \item If $z\in \widehat{\reeb}_p^{12|s,r}$ and $w \in \widehat{\reeb}_p^{12|11}$, there is a bijective correspondence between the disks in $\Delta_{\leg^{s,r}}(\ul{z},\mba \ul{w} \mbb)$ and the doubly enriched disks in $\Delta^{r-1,s-1}_\leg(\ul{\check{w}}, \check{\mbb} \ul{\check{z}} \check{\mba})$.

    \item If $y \in \widehat{\reeb}_q^{12|s,r}$ arises from a minimum and $w \in \widehat{\reeb}_p^{12|11}$, there is a bijective correspondence between the disks in $\Delta_{\leg^{s,r}}(\ul{y},\mba \ul{w} \mbb)$ and the doubly enriched disks in $\Delta^{r-1,s-1}_\leg(\ul{\check{w}}, \check{\mbb}\ul{t}^{\pm 1}  \check{\mba})$, where $0$ is enriched at most once in $\mbw_1$ and not enriched in $\mbw_2$ if the mixed enrichment is at $t$, and $0$ is not enriched in $\mbw_1$ if the mixed enrichment is at $t^{-1}$.
    \end{enumerate}
\end{lem}

\section{Augmentation Categories, Bimodules, and Isomorphisms}
\label{sec:aug-cat}

In this section, we define the central algebraic objects in the paper using the geometric models established in the previous section.  We begin by reviewing the construction from \cite{nrssz:aug-sheaf} of the augmentation $A_\infty$ categories $\Aug_\pm(\leg)$ in Section~\ref{ssec:aug-cat}; these constructions rely on the $n$-copy. Next, we define a variety of bimodules over $\Aug_+(\leg)$ in Sections~\ref{ssec:aug-bimod} and \ref{ssec:aug-bimod-defn}; these bimodules are constructed using the $(s,r)$-copy.  Finally, we elucidate relationships between the augmentation bimodules  using bimodules associated to the separated $2$-copy in Section~\ref{ssec:sep-module}.

\subsection{Augmentation Categories}
\label{ssec:aug-cat}

We begin with a summary of the definition of the augmentation categories $\Aug_\pm(\leg)$ associated to a Legendrian $\leg$, along with a third ``circle category'' $\mcc$. The first appearances of the categories $\Aug_-(\leg)$ and $\Aug_+(\leg)$ were in \cite{bc:bilinear} and \cite{nrssz:aug-sheaf} respectively; see also the earlier paper \cite{products}.  All three categories are built on the geometric model of the $n$-copy $\nleg{n}$.

\subsubsection{Positive and Negative Augmentation Categories}
\label{sssec:pm-aug-cat}

The objects of the $A_\infty$ categories $\Aug_{\pm}(\leg)$ are simply the augmentations of the Chekanov-Eliashberg DGA of $\leg$.  We emphasize once again that we restrict ourselves to augmentations over $\ff_2$.

To define the morphisms, differentials, and higher products, we generalize the algebraic structures presented in Section~\ref{ssec:aug}. Given an $n$-tuple of augmentations $(\aug_1, \ldots, \aug_n)$, define the \dfn{diagonal augmentation} $\hat{\aug}$ on the $n$-copy $\nleg{n}$ to equal $\aug_i$ on generators in $\gens^{ii}$ and to vanish on $\gens^{ij}$ for $i \neq j$. This augmentation respects the link grading on the $n$-copy by definition. Hence, using Proposition~\ref{prop:step1-dga}, we obtain the $A_\infty$ algebra $(A^\vee, \{m_{\hat{\aug}}^k\})$ from the DGA of the $n$-copy $\nleg{n}$.  Further, as noted in the discussion after Proposition~\ref{prop:step1-dga}, the vector space and products in the $A_\infty$ algebra $(A^\vee, \{m_{\hat{\aug}}^k\})$ split, with the only non-vanishing products given by   
\begin{equation} 
m^k_{\hat{\aug}}: A^\vee_{i_k i_{k+1}} \otimes \cdots \otimes A^\vee_{i_1 i_2} \to A^\vee_{i_1 i_{k+1}}.
\end{equation}

We are now in position to define the morphism spaces and structure maps for the augmentation categories.  First, given a pair of augmentations $(\aug_1, \aug_2)$, we define
\begin{align*}
	\Hom_+(\aug_1,\aug_2) &= A^\vee_{12}, \\
	\Hom_-(\aug_2,\aug_1) &= A^\vee_{21}.
\end{align*}
The morphism spaces depend neither on the augmentations $\aug_i$ nor on the number $n$ of copies of $\leg$ (so long as $n \geq 2$). In the case of a simply perturbed $n$-copies, we follow the conventions in Section~\ref{sssec:n-copy-geometry} to denote the generators of $A^\vee_{12}$ by $a_k^+, x^+, y^+$ and those of $A^\vee_{21}$ by $a_k^-$. Further, we may split the vector space $A^\vee_{12}$ into $A^\vee_{12,\leg} \oplus A^\vee_{12, f}$, where $A^\vee_{12,\leg}$ is generated by the $a^+_k$ and $A^\vee_{12,f}$ is generated by $x^+$ and $y^+$.

To define the maps $m^k_+: \Hom_+(\aug_{k+1}, \ldots, \aug_1) \to \Hom_+(\aug_1, \aug_{k+1})$, we first note that implicit in the translation of \cite[Lemma 3.15]{nrssz:aug-sheaf} to the algebras of $n$-copies in \cite[\S4]{nrssz:aug-sheaf} is the fact that, for $i< j$, there is a chain isomorphism $h_{ij}: A^\vee_{12} \to A^\vee_{ij}$ obtained by an identification of generators, and similarly for $i>j$ using $A^\vee_{21}$. We then compose the structure map $m^k_{\hat{\aug}}$ with appropriate identifications $h_{ij}$ to define $m_+^k$:
\begin{equation} \label{eq:m+}
	m_+^k = h^{-1}_{1, k+1}  \circ m_{\hat{\aug}}^k \circ (h_{k,k+1} \otimes \cdots \otimes h_{12}).
\end{equation}
We will occasionally abuse notation for convenience and abbreviate $h_{ij}$ by $h$ and extend it to $\alg$.  As discussed in \cite{nrssz:aug-sheaf}, the definition of $m^+_k$ is independent of the number of knots in the $n$-copy, or even the choice of $k+1$ components inside a given $n$-copy. The disks used in the definition of $m_+^k(a_k^+, \ldots, a_1^+)$ are those in $\Delta_{\nleg{k+1}}(h(a), h(\mba))$, where the associated word is in positive $(k+1)$-copy form with $a_{i_j} = a_j$ and $\aug_i(\mba_i) = 1$.

Similarly, the  maps $m^k_-: \Hom_-(\aug_1, \ldots, \aug_{k+1}) \to \Hom_-(\aug_{k+1},\aug_1)$ are defined as the composite 
\begin{equation} \label{eq:m-}
	m_-^k = h^{-1}_{k+1,1}  \circ m_{\hat{\aug}}^k \circ (h_{21} \otimes \cdots \otimes h_{k+1,k}).
\end{equation}
Note that since $A^\vee_{ij}$ for $i>j$ does not depend on the perturbation $f$, neither does $m^k_-$.  The disks used in the definition of $m_-^k(a_k^-, \ldots, a_1^-)$ are those in $\Delta_{\nleg{k+1}}(h(a), h(\mba))$, where $(a,\mba)$ is in negative $(k+1)$-copy form with $a_{i_j} = a_j$ and $\aug_i(\mba_i) = 1$.

Summarizing, we state the following definition and its accompanying theorem.

\begin{defn} \label{defn:aug-cat}
    The \dfn{positive and negative augmentation categories} $\Aug_\pm(\leg)$ of a Legendrian knot $\leg$ have augmentations of the associated Chekanov-Eliashberg DGA $(\alg, \df)$ as objects, $\Hom_\pm$ as morphism spaces, and $\{m_\pm^k\}$ as structure maps.
\end{defn}

\begin{thm}[\cite{bc:bilinear,nrssz:aug-sheaf}]
    The categories $\Aug_\pm(\leg)$ are $A_\infty$ categories. Up to $A_\infty$ equivalence, the categories $\Aug_\pm(\leg)$ do not depend on the perturbation $f$ and are invariant under Legendrian isotopy.
\end{thm}

The reason we prefer $\Aug_+$ to $\Aug_-$ is that the former is a strictly unital category.  The unit is straightforward to describe: the element $y^+ \in \Hom_+(\aug,\aug)$ is a strict unit \cite{nrssz:aug-sheaf}. Concretely, the disks that yield the action of the unit are the thin disks with positive corner at $b^{13}$ (for $b=a,x,y$) in Lemma~\ref{lem:thin-disk}.

It is useful to characterize the structure maps $m_+^k$ in terms of thin disks and enriched disks in the original Lagrangian diagram of $\leg$, building on the ideas of \cite[Proposition 4.26]{nrssz:aug-sheaf}. For our purpose, we consider the case when $\leg^n$ is simply perturbed, though the model applies to any choice of perturbation function with  changes to some details. The augmentation category defined in this manner is also termed \dfn{simply perturbed}.

Recall from Lemma~\ref{lem:n-copy-dil} that there is a bijective correspondence between disks in the $(k+1)$-copy referenced above and $k$-enriched disks, with the correspondence given by projection from disks in the $(k+1)$-copy to enriched disks and geometric lifting of enriched disks to disks in the $(k+1)$-copy. To complement the geometric lifting process, we define the \dfn{algebraic lift} $\hat{\alpha}: \Delta_\leg^k(a,\mba) \to \alg$ of a $k$-enriched disk $(u, \mbw)$ with $\mbw = \{m_1i_1, \ldots, m_li_l\}$ to be 
\[\hat{\alpha}(u, \mbw)=\hat{a}_{i_1} \hat{a}_{i_2} \cdots \hat{a}_{i_l},\]
where $\hat{a}_j = a_j$ if $a_j \in \reeb_\leg$ and $\hat{a}_j = x \cdots x$ for $m_j$ number of $x$ generators if $a_j = t^{\pm 1}$.  It will also be convenient to define one more piece of notation.  The augmentation of a word in $(k+1)$-copy form by taking the product of the augmentations of the pure generators:
\begin{equation} \label{eq:hat-aug}
    \hat{\aug}(\mba) = \aug_1(\mba_1) \aug_{1+m_1}(\mba_2) \cdots \aug_{k+1}(\mba_{k+1}).
\end{equation}

\begin{lem} \label{lem:aug+disks}
    In the definition of $m_+^k$ for a simply perturbed $\Aug_+(\leg)$, thin disks yield the terms dictated by the strict unitality of $y^+$, as well as the terms
    \begin{equation*}
        m_+^1(y^+) = \left(\aug_1(a) + \aug_2(a)\right) a^+.
    \end{equation*}
    Enriched disks yield additional terms
    \begin{equation*}
    m_+^k(a_k^+,\ldots,a_1^+) = \sum_{\substack{(u,\mbw) \in \Delta_\leg^k(\check{a}, \check{\mba})\\ \hat{\alpha}(u, \mbw)  = a_1 \cdots a_k}}  \hat{\aug}(\mba) a^+.
\end{equation*}
\end{lem}

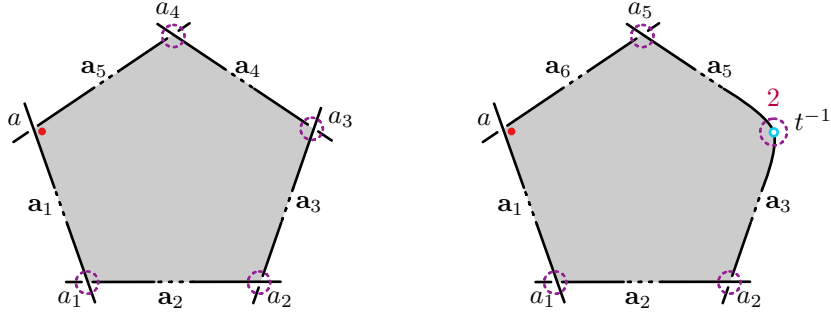
\begin{figure}
\labellist
\small\hair 2pt
 \pinlabel {$a$} [ ] at 7 77
 \pinlabel {$a$} [ ] at 184 77
 \pinlabel {$a_1$} [ ] at 28 9
 \pinlabel {$a_1$} [ ] at 206 9
 \pinlabel {$a_2$} [ ] at 107 9
 \pinlabel {$a_2$} [ ] at 284 9
 \pinlabel {$a_3$} [ ] at 130 77
 \pinlabel {$t^{-1}$} [ ] at 308 77
 \pinlabel {$a_4$} [ ] at 65 118
 \pinlabel {$a_5$} [ ] at 243 118
 \pinlabel {$\mba_1$} [ ] at 17 44
 \pinlabel {$\mba_1$} [ ] at 194 44
 \pinlabel {$\mba_2$} [ ] at 66 9
 \pinlabel {$\mba_2$} [ ] at 242 9
 \pinlabel {$\mba_3$} [ ] at 118 45
 \pinlabel {$\mba_3$} [ ] at 295 45
 \pinlabel {$\mba_4$} [ ] at 95 96
 \pinlabel {$\mba_5$} [ ] at 273 96
 \pinlabel {$\mba_5$} [ ] at 37 96
 \pinlabel {$\mba_6$} [ ] at 212 96
 \pinlabel {\textcolor{purple}{$2$}} [b] at 293 80
\endlabellist
\centering
\includegraphics{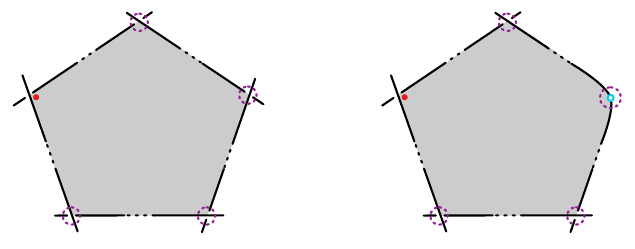}
\caption{Enriched disks that contribute $a^+$ to $m_+^4(a_4^+, \ldots, a_1^+)$ (left) and to $m_+^5(a_5^+, x^+, x^+, a_2^+, a_1^+)$ (right).  Enriched disks on the left, but not the right, side contribute to $m_-^4(a_4^-, \ldots, a_1^-)$.}
\label{fig:m+enrich}
\end{figure}

The enriched disks involved in the definition of $m_+$ appear in Figure~\ref{fig:m+enrich}.

\begin{proof}
    The thin disk contributions follow directly from Lemma~\ref{lem:thin-disk}, while the thick disk contributions come from using Lemma~\ref{lem:n-copy-dil} to identify the disks in the definition of $m^k_+$ with enriched disks and their algebraic lifts.
\end{proof}

We may also characterize the structure maps $m^k_-$ in terms of enriched disks.  If $(a,\mba)$ is in negative $(k+1)$-copy form, then all of the generators $a_{i_j}$ arise from genuine Reeb chords in $\leg$.  Thus, any disk involved in the definition of $m^k_-$ projects to an enriched disk with enrichments --- which must have multiplicity $1$ --- only at Reeb chords.

\begin{lem} \label{lem:aug-disks}
In the definition of $m^k_-$, enriched disks with enrichments only at Reeb chords yield terms
\begin{equation*}
    m_-^k(a_k^-,\ldots,a_1^-) = \sum_{\substack{(u,\mbw) \in \Delta_\leg^k(\check{a}, \check{\mba}) \\ \hat{\alpha}(u, \mbw) = a_1 \cdots a_k}} \hat{\aug}(\mba) a^-.
\end{equation*} 
\end{lem}

\begin{ex} \label{ex:running-aug-cat}
    Continuing the example of figure eight knot $\leg$ presented in Figure~\ref{fig:fig-eight}, we note that its augmentation categories have two objects, namely $\aug_1$ and $\aug_2$ described in Example~\ref{ex:running-aug-dga}. The morphism spaces are
    \begin{align*}
        \Hom_+(\aug_i,\aug_j) &= \ff_2 \langle a_1^+,\ldots,a_7^+,x^+,y^+ \rangle, \\
        \Hom_-(\aug_i,\aug_j) &= \ff_2 \langle a_1^-,\ldots,a_7^- \rangle,
    \end{align*}
    where $|a_k^+| = |a_k^-| = |a_k^\vee|$ as in Example~\ref{ex:running-a-infty-alg} and $|x^+| = 1$, $|y^+| = 0$. 

    The structure maps $m_-^k$, which only account for thick disks in $\tcleg$, match the codifferential $\delta_{\aug_i}$ of the DGCA of $\leg$ in Example~\ref{ex:running-dual-dga} with suitable change of notation. As explained in Example~\ref{ex:running-aug-dga}, the maps $m_-^k$ are independent of the choice of augmentations in this example.

    The structure maps $m_+^k$, on the other hand, do depend on the choice of augmentations due to additional thin disk contributions. We shall describe the structure maps on $\Hom_+(\aug_1,\aug_1)$, which happen to be the same as those on $\Hom_+(\aug_2,\aug_2)$. For compactness of presentation, we do not list the $2$-fold products that show that $y^+$ acts as the identity.
    \begin{align*} m_+^1(a_2^+) &= a_4^+ & m_+^2(a_1^+,a_2^+) &= a_7^+ &  m_+^3(a_3^+, a_1^+, a_2^+) &= a_4^+\\
     m_+^1 (a_3^+) &= a_4^+ & m_+^2(a_3^+,a_1^+) &= a_6^+ & m_+^3(a_5^+, a_3^+, a_1^+) &= a_6^+ \\
    m_+^1 (a_5^+) &= a_6^+ + a_7^+ & &&   m_+^3(a_1^+, a_2^+, a_5^+) &= a_7^+ \\
    m_+^1 (x^+) &= a_6^+ &  & 
    \end{align*}
    Unlike in the maps listed above, the $m_+^1$ map on $\Hom_+(\aug_1,\aug_2)$ is non-vanishing on $y^+$ with $m_+^1(y^+) = a_4^+$.
\end{ex}

\subsubsection{The Circle Category}
\label{sssec:circ-cat}

We also define a third augmentation category that will prove essential in the upcoming duality structure. First, we need to define $\Aug_+^{\text{Reeb}}(\leg)$ to be the wide subcategory of $\Aug_+(\leg)$ whose morphisms are given by $A^\vee_{12,\leg}$.  It follows from (the dual of)  Lemma~\ref{lem:stokes} and the fact that $\reeb^{12}_\leg$ contains all of the Reeb chords above a certain length in $\reeb^{12}$ that $\Aug_+^{\text{Reeb}}(\leg)$ is indeed a subcategory. By Lemmas~\ref{lem:aug+disks} and \ref{lem:aug-disks}, we may identify $\Aug_+^{\text{Reeb}}(\leg)$ with $\Aug_-(\leg)$.

\begin{defn} 
\label{defn:circle-category}
The \dfn{circle category} is the quotient category $\mcc(\leg) = \Aug_+(\leg)/\Aug_+^{\text{Reeb}}(\leg)$.
\end{defn}

In particular, $\mcc$ is an $A_\infty$ category with morphisms $\Hom_\mcc(\aug_1, \aug_2) = A^\vee_{12,f}$ and structure maps $m_\mcc^k$ counting thin disks in $\nleg{k+1}$. When $\nleg{k+1}$ is simply perturbed, Lemma~\ref{lem:thin-disk} shows that the only nontrivial structure maps for the circle category are those that show that $y^+$ is a strict unit.

On the linear level, the circle category reduces to the Morse complex for the perturbing function $f$ of a circle, thus justifying its name. In fact, the category recovers the cohomology ring. 

\subsection{Bimodules From Two-component Links}
\label{ssec:aug-bimod}

The algebraic structures involved in the main theorem are $A_\infty$ bimodules over $\Aug_+(\leg)$ and  morphisms between them. Thus, our next task is to define and understand the bimodule analogues of the augmentation categories in Section~\ref{ssec:aug-cat}. To this end, we begin with a general construction of $\Aug_+(\leg)$-bimodules from an ordered two-component Legendrian link.

\begin{prop} \label{prop:bimod-from-link-grading}
    A Legendrian link $\leg = \leg_1 \sqcup \leg_2$ induces an $\Aug_+(\leg_2)$-$\Aug_+(\leg_1)$ bimodule $\mcm$ with the space $\mcm(\aug^2,\aug^1)$ generated by $\reeb^{12}$.
\end{prop}

Intuitively, the induced bimodule $\mcm$ repackages the information in the category $\Aug_+(\leg)$ in a way that respects the $2$-component structure of the link.  

In the proof below and in subsequent sections, we adopt compact notation for arguments to bimodule structure maps.  First, let $\cev{\aug}^2 = (\aug_{r+1}^2, \ldots, \aug_1^2)$ and $\vec{\aug}\,^1 = (\aug_1^1, \ldots, \aug_{s+1}^1)$.  Similarly, let $\vec{a} = (a_1^+, \ldots, a_s^+)$ and $\cev{b} = (b_r^+, \ldots, b_1^+)$.

\begin{proof}
To define the desired $\Aug_+(\leg_2)$-$\Aug_+(\leg_1)$ bimodule $\mcm$, let $\aug^i$ be an augmentation of $\leg_i$, $i=1,2$. Set $\mcm(\aug^2,\aug^1)$ to be the vector space generated by $\reeb^{12}$, with generators denoted $w_i$.

We next define the structure maps 
\begin{equation} \label{eq:link-structure}
n^{r|s}: \Hom_{\Aug_+(\leg_2)}(\cev{\aug}^2) \otimes \mcm(\aug_1^2, \aug_1^1) \otimes \Hom_{\Aug_+(\leg_1)}(\vec{\aug}\,^1) \to \mcm(\aug_{r+1}^2, \aug_{s+1}^1)
\end{equation}
in terms of the structure maps of the category $\Aug_+(\leg)$ or, more accurately, using its diagonal bimodule.  In particular, we need to identify a tuple $(\cev{b}, \ul{w}, \vec{a})$ in the domain of $n^{r|s}$ with a tuple in the domain of $m_+^{r+s+1}$.

The first step is to extend the augmentations $\aug_k^1$ on $\leg_1$ to pure augmentations $\hat{\aug}_k^1$ on $\leg$.  To effect this, for each $k\in \{1,\ldots, s+1\}$, we define the extension $\hat{\aug}_k^1$ to equal $\aug_1^2$ on generators from $\leg_2$; we make a similar definition for $\hat{\aug}_l^2$ for $l \in \{1, \ldots, r+1\}$.  

Next, consider a tuple $\vec{a} \in \Hom_{\Aug_+(\leg_1)}(\vec{\aug}\,^1)$. Each $a_k^+$ corresponds to a mixed Reeb chord of the $2$-copy of the Legendrian link $\leg_1$, and hence a mixed Reeb chord of the $2$-copy of the Legendrian link $\leg$. It follows that we may think of $a_k^+$ as lying in $\Hom_{\Aug_+(\leg)}(\hat{\aug}_{k+1}^1, \hat{\aug}_k^1)$; notice the transposition of indices.  Similarly, for a tuple $\cev{b} \in \Hom_{\Aug_+(\leg_2)}(\cev{\aug}^2)$, we may think of $b_k^+$ as lying in $\Hom_{\Aug_+(\leg)}(\hat{\aug}_{k+1}^2, \hat{\aug}_{k}^2)$. There is no transposition of indices in this case.

Finally, a generator $w \in \mcm(\aug^2, \aug^1)$, which corresponds to a Reeb chord of $\leg$ from $\leg_1$ to $\leg_2$, may be identified with the nearby mixed Reeb chord in the $2$-copy $\tcleg$ that begins on the $\leg_2$ component of the second copy and ends on the $\leg_1$ component of the first copy. In particular, $w$ may be thought of as lying in $\Hom_{\Aug_+(\leg)}(\hat{\aug}^2_1, \hat{\aug}^1_1)$.

With the above identification of tuple $(\cev{b}, \ul{w}, \vec{a})$ in the domain of $n^{r|s}$ as a tuple in the domain of $m_+^{r+s+1}$ in hand, we define
\[n^{r|s}(\cev{b}, \ul{w}, \vec{a}) = m_+^{r+s+1}(\cev{b}, w, \vec{a}).\]
Since the structure maps of $\Aug_+(\leg)_\Delta$ satisfy the defining $A_\infty$ bimodule equations, so too will $n^{r|s}$. 
\end{proof}

It will be useful to precisely specify the chords and disks involved in the definition of $n^{r|s}$.  Recall that the structure map $m_+^{r+s+1}$ is defined by counting disks in $\nleg{r+s+2}$ whose first $s$ negative mixed corners lie in the top $(s+1)$-copy of $\leg_1^2$ and last $r$ negative mixed corners lie in the bottom $(r+1)$-copy of $\leg_2^2$; as before, there may be pure augmented corners as well. The definition of $n^{r|s}$ counts disks in a sublink of $\nleg{r+s+2}$ identified with $\nleg{s+1,r+1}$. Thus, the $(s+1,r+1)$-copy is the correct geometric model for the bimodule $\mcm$ as claimed. 

Next, to identify the chords in the $(s+1,r+1)$-copy used in the definition of $n^{r|s}$, we refine the maps $h_{ij}$ from Section~\ref{sssec:pm-aug-cat} to maps $h_{ij|kl}$ that identify a chord $a^{ij}$ of $\nleg{2,2}$ with the chord $a^{ij|kl}$ of $\nleg{s+1,r+1}$. 

Building on the remarks after Equation~\eqref{eq:m+}, 
we now identify the disks in $\nleg{s+1,r+1}$ used to define $n^{r|s}$. 

\begin{lem} \label{lem:sr-structure-map}
The disks used in the definition of $n^{r|s}(\cev{b}, \ul{w}, \vec{a})$ are those in \[\Delta_{\nleg{s+1,r+1}} (\ul{h_{12|s+1,r+1}(z)}, h_{11|\cdot \cdot}(\mba) \ul{h_{12|11}(w)} h_{22|\cdot \cdot}(\mbb)),\] where the associated word is in $(s+1,r+1)$-copy form with $a_{i_j} = a_j, b_{i_j} = b_j$ and $\aug_i^1(\mba_i) = \aug_i^2(\mbb_i) =  1$.
\end{lem}

As in the construction of $\Aug^{\text{reeb}}_+(\leg)$ as wide subcategory of $\Aug_+(\leg)$, we may use heights of Reeb chords to construct submodules of the bimodule $\mcm$ of Proposition~\ref{prop:bimod-from-link-grading}.  Suppose $\reeb_\lambda^{12} \subset \reeb^{12}$ consists of all Reeb chords with height greater than $\lambda$.  For $\aug^i$ an augmentation of $\leg_i$, $i=1,2$, let $\mcm_\lambda(\aug^2, \aug^1) \subset \mcm(\aug^2, \aug^1)$ be the subspace generated by $\reeb_\lambda^{12}$.  Combining Lemma~\ref{lem:stokes} with Lemma~\ref{lem:sr-structure-map} shows that $\mcm_\lambda$ is an $\Aug_+(\leg_2)$-$\Aug_+(\leg_1)$ submodule of $\mcm$, which we call the \dfn{length filtration submodule at height $\lambda$}.

\subsection{Augmentation Bimodules}
\label{ssec:aug-bimod-defn}

While one might be tempted to simply take diagonal bimodules of the three $A_\infty$ categories defined in Section~\ref{ssec:aug-cat}, we will need all three bimodules to have coefficients in the same category.  This necessitates a more subtle approach using Proposition~\ref{prop:bimod-from-link-grading}.

\begin{defn} \label{defn:+-aug-bimod}
    The \dfn{positive augmentation bimodule} $\mcm_+$ of a Legendrian knot $\leg$ is defined to be the result of applying Proposition~\ref{prop:bimod-from-link-grading} to the $2$-copy $\tcleg$.

    The \dfn{negative augmentation bimodule} $\mcm_-$ of a Legendrian knot $\leg$ is defined to be the result of applying Proposition~\ref{prop:bimod-from-link-grading} to the opposite $2$-copy $\otcleg$.
\end{defn}

An inspection of the proof of Proposition~\ref{prop:bimod-from-link-grading} shows that the augmentation bimodules defined above depend on the perturbation functions for both the (opposite) $2$-copy and its $(s+r+2)$-copy. For our purposes, we only consider augmentation bimodules defined by a simply perturbed (opposite) $2$-copy as well as its $(s+r+2)$-copy perturbed by a Morse function with the maxima and minima placed as shown in Figure~\ref{fig:2-copy-crit}. In this case, the combinatorics of the Lagrangian diagrams of $\tcleg^{s+r+2}$ and $\otcleg^{s+r+2}$ matches that of the diagrams of a simply perturbed $\nleg{2s+2r+4}$ and $\nleg{2s+2r+4}$, up to relabeling. As a result, the base category $\Aug_+(\leg)$ is simply perturbed. An augmentation bimodule defined in this manner is also termed \dfn{simply perturbed}.

In this setting, we label the generators of the spaces $\mcm_+(\aug^2,\aug^1)$ by $a^+$, $x^+$, and $y^+$, and the generators of $\mcm_-(\aug^2,\aug^1)$ by $a^-$, as the generators of  $\Hom_\pm$ match those of $\mcm_\pm$ in our chosen configuration. Following the discussion around Lemma~\ref{lem:sr-structure-map}, the sublink in $\tcleg^{s+r+2}$ used to define the structure maps $n_{\mcm_+}^{r|s}$ can be identified with a simply perturbed $(s+1,r+1)$-copy of $2$-copy. 

\begin{rem}
It would be reasonable to expect that, using a consistent sequence of DGA maps as in \cite[Proposition 4.21]{nrssz:aug-sheaf}, any two perturbation schemes are conjecturally equivalent under an appropriate notion of equivalence of $A_\infty$ bimodules over an equivalence of $A_\infty$ categories.
\end{rem}

In parallel to Lemma~\ref{lem:aug+disks}, we may use Lemmas~\ref{lem:s-r-2-copy-thin-disk} and \ref{lem:s-r-2-copy-dil} to characterize $n_{\mcm_+}^{r|s}$ in terms of thin disks and doubly enriched disks in the original Lagrangian diagram of $\leg$. In order to do so, we define the \dfn{algebraic lift} of an $(s,r)$-enriched disk $(u,\mbw_1,\mbw_2) \in \Delta_\leg^{s,r}(\ul{z},\mba \ul{w} \mbb)$ with $\mbw_1 = \{m_0 0, m_1i_1,\ldots,m_ki_k\}$ and $\mbw_2 = \{n_00, n_1j_1,\ldots,n_lj_l\}$ to be
\begin{equation*}
    \hat{\alpha}(u,\mbw_1,\mbw_2) = \hat{a}_{i_k}\cdots \hat{a}_{i_1}\hat{\ul{w}}\hat{b}_{j_1}\cdots \hat{b}_{j_l},
\end{equation*}
where $\hat{a}$ and $\hat{b}$ are defined as before, while $\hat{\ul{w}} = \ul{w}$ if $w \in \reeb_\leg$ and $\hat{w} = x\cdots x\ul{x}x \cdots x$ for $m_0$ number of $x$ generators preceding $\ul{x}$ and $n_0$ number of $x$ generators following $\ul{x}$.  

Using the notation defined after Proposition~\ref{prop:bimod-from-link-grading}, we characterize the disks used in the definition of $n_{\mcm_+}$ as follows.

\begin{lem} \label{lem:mcm+disks}
    In the definition of \[n_{\mcm_+}^{r|s}: \Hom_+(\cev{\aug}^2) \otimes \mcm_+(\aug_1^2,\aug_1^1) \otimes \Hom_+(\vec{\aug}\,^1) \to \mcm_+(\aug_{r+1}^2, \aug_{s+1}^1)\] (see discussion near Equation~\eqref{eq:link-structure} for the notation) for a simply perturbed $\mcm_+$, thin disks yield the terms dictated by the strict unitality of $y^+$, as well as the terms
    \begin{align*}
        n_{\mcm_+}^{0|0}(\ul{y}^+) &= \left(\aug_1(a) + \aug_2(a)\right) a^+, &
        n_{\mcm_+}^{1|0}(a^+,\ul{y}^+) &= n_+^{0|1}(\ul{y}^+,a^+) = a^+, \\
        && n_{\mcm_+}^{1|0}(x^+,\ul{y}^+) &= n_+^{0|1}(\ul{y}^+,x^+) = x^+.
    \end{align*}
    Doubly enriched disks yield additional terms
    \begin{equation*}
    n_{\mcm_+}^{r|s}(\cev{b},\ul{w}^+,\vec{a}) \\
    = \sum_{\substack{(u,\mbw_1,\mbw_2) \in \Delta_\leg^{s,r}(\check{\ul{z}},\check{\mba} \check{\ul{w}} \check{\mbb})\\ \hat{\alpha}(u, \mbw_1,\mbw_2)  = a_s \cdots a_1\ul{w}b_1\cdots b_r}}  \hat{\aug}^1(\mba) \hat{\aug}^2(\mbb) z^+
    \end{equation*}
    
\end{lem}

See Figure~\ref{fig:n+enrich} for illustrations of the doubly enriched disks that contribute to $n_{\mcm_+}$.

\begin{figure}
\labellist
\small\hair 2pt
 \pinlabel {$z$} [ ] at 9 54
 \pinlabel {$z$} [ ] at 196 54
 \pinlabel {$a_s$} [ ] at 51 3
 \pinlabel {$a_s$} [ ] at 238 3
 \pinlabel {$a_{s-1}$} [ ] at 84 3
 \pinlabel {$a_{s-1}$} [ ] at 272 3
 \pinlabel {$a_1$} [ ] at 126 3
 \pinlabel {$a_{1+m_0}$} [ ] at 313 3
 \pinlabel {$w$} [l] at 168 54
 \pinlabel {$t^{\pm 1}$} [l] at 348 54
 \pinlabel {$b_1$} [ ] at 126 106
 \pinlabel {$b_{1+n_0}$} [ ] at 313 106
 \pinlabel {$b_2$} [ ] at 92 106
 \pinlabel {$b_2$} [ ] at 281 106
 \pinlabel {$b_r$} [ ] at 51 106
 \pinlabel {$b_r$} [ ] at 238 106
 \pinlabel {$\mba_{s+1}$} [r] at 33 32
 \pinlabel {$\mba_{s+1}$} [r] at 219 33
 \pinlabel {$\mba_s$} [ ] at 67 18
 \pinlabel {$\mba_s$} [ ] at 254 18
 \pinlabel {$\mba_1$} [l] at 146 30
 \pinlabel {$\mba_{1+m_0}$} [l] at 328 25
 \pinlabel {$\mbb_1$} [l] at 145 78
 \pinlabel {$\mbb_{1+n_0}$} [l] at 328 82
 \pinlabel {$\mbb_2$} [ ] at 108 89
 \pinlabel {$\mbb_2$} [ ] at 296 89
 \pinlabel {$\mbb_{r+1}$} [r] at 33 76
 \pinlabel {$\mbb_{r+1}$} [r] at 219 76
 \pinlabel {\textcolor{purple}{$n_0$}} [l] at 343 65 
 \pinlabel {\textcolor{purple}{$m_0$}} [l] at 343 42 
\endlabellist
\centering
\includegraphics{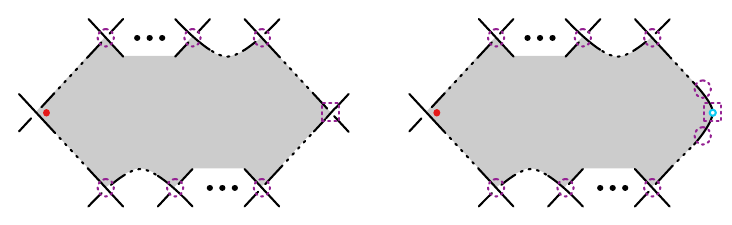}
\caption{Doubly enriched disks that contribute $z^+$ to $n_{\mcm_+}^{r|s}(\cev{b},\ul{w}^+,\vec{a})$ (left) and to $n_{\mcm_+}^{r|s}(\cev{b},\ul{t}^{\pm 1},\vec{a})$, with $n_0$ and $m_0$ $x^+$ generators before and after $\ul{w}^+$ (right); recall that $m_0=0=n_0$ if $w=t$.  Doubly enriched disks on the left side contribute $z^-$ to $n_{\mcm_-}^{r|s}(\cev{b},\ul{w}^-,\vec{a})$ or $w^\vee$ to $n_{\mcm_-^\vee}^{s|r}(\vec{a},\ul{z}^\vee,\cev{b})$, while doubly enriched disks on the right side will contribute $y^\vee$ to $(\rho^\vee)^{s|r}(\vec{a}, \ul{z}^\vee, \cev{b})$ with $m_0$ $x^+$ generators at the beginning and $n_0$ $x^+$ generators at the end; again, $m_0=0=n_0$ if $w=t$.  Finally, doubly enriched disks on the right side will contribute $y^+$ to $(\pi_\mcn \eta)^{s|r}$ with the same arguments, so long as $m_0=0$ when $w=t^{-1}$ or $m_0 \leq 1$ and $n_0=0$ when $w=t$. }
\label{fig:n+enrich}
\end{figure}

\begin{proof}
    The thin disk contributions follow directly from Lemma~\ref{lem:s-r-2-copy-thin-disk}, while the thick disk contributions come from using Lemma~\ref{lem:s-r-2-copy-dil} to identify the disks in the definition of $n_{\mcm_+}^{r|s}$ with doubly enriched disks and their algebraic lifts.
\end{proof}

Similarly, the structure maps $n_{\mcm_-}^{r|s}$ of $\mcm_-$ are defined using a simply perturbed $(s+1,r+1)$-copy of the opposite $2$-copy. As above, we may characterize $n_{\mcm_-}$ in terms of thin disks and doubly enriched disks using Lemma~\ref{lem:s-r-2-copy-thin-disk} and Lemma~\ref{lem:s-r-2-copy-dil}. In parallel to its category analogue $\Aug_-(\leg)$, note that a disk of the form $(\ul{z},\mba\ul{w}\mbb)$ involved in the definition of $n_{\mcm_-}^{r|s}$ projects to an enriched disk with mixed enrichment only at a Reeb chord.  

\begin{lem} \label{lem:mcm-disks}
    In the definition of 
    \begin{equation*}
        n_{\mcm_-}^{r|s}: \Hom_+(\cev{\aug}^2) \otimes \mcm_-(\aug_1^2,\aug_1^1) \otimes \Hom_+(\vec{\aug}\,^1) \to \mcm_-(\aug_{r+1}^2, \aug_{s+1}^1),
    \end{equation*}
    for a simply perturbed $\mcm_-$, thin disks yield the terms dictated by the strict unitality of $y^+$.
    
    Doubly enriched disks with mixed enrichment at a Reeb chord yield additional terms
    \begin{equation*}
    n_{\mcm_-}^{r|s}(\cev{b},\ul{w}^-,\vec{a}) \\
    = \sum_{\substack{(u,\mbw_1,\mbw_2) \in \Delta_\leg^{s,r}(\check{\ul{z}},\check{\mba} \check{\ul{w}} \check{\mbb})\\ \hat{\alpha}(u, \mbw_1,\mbw_2)  = a_s \cdots a_1\ul{w}b_1\cdots b_r}}  \hat{\aug}^1(\mba)\hat{\aug}^2(\mbb) z^-.
\end{equation*}
\end{lem}

See Figure~\ref{fig:n+enrich} once again.

It will also be useful to describe the structure maps of the linear dual $\mcm_-^\vee$ of $\mcm_-$ using thin and enriched disks. Following Section~\ref{sssec:a-infty-bimod-construct}, we label the generators of the morphisms of $\mcm_-^\vee$ by $a^\vee$ with $|a^\vee| = -|a^-|$. We also define the \dfn{dual algebraic lift} of an $(s,r)$-enriched disk $(u,\mbw_1,\mbw_2) \in \Delta_\leg^{s,r}(\ul{z},\mba \ul{w} \mbb)$ to be
\begin{equation*}
    \hat{\alpha}^\vee(u,\mbw_1,\mbw_2) = x \cdots x\hat{a}_{i_1}\cdots \hat{a}_{i_k}\hat{\ul{z}}\hat{b}_{j_l}\cdots \hat{b}_{j_1}x \cdots x.
\end{equation*}
where $\hat{a}$ and $\hat{b}$ are defined as before and $\hat{\ul{z}} = \ul{z}$ (recall that $z \in \reeb_\leg$). Further, there are $m_0$ number of $x$ generators at the front and $n_0$ number of $x$ generators at the end.

\begin{lem} \label{lem:mcm-vee-disks}
In the definition of \[n_{\mcm_-^\vee}^{s|r}:\Hom_+(\vec{\aug}\,^1) \otimes \mcm_-^\vee(\aug_{s+1}^1,\aug_{r+1}^2) \otimes \Hom_+(\cev{\aug}^2) \to \mcm_-^\vee(\aug_1^1,\aug_1^2)\] for a simply perturbed $\mcm_-^\vee$, thin disks yield the terms dictated by the strict unitality of $y^+$. Doubly enriched disks with mixed enrichment at a Reeb chord yield additional terms
    \begin{equation*}
    n_{\mcm_-^\vee}^{s|r}(\vec{a},\ul{z}^\vee,\cev{b}) \\
    = \sum_{\substack{(u,\mbw_1,\mbw_2) \in \Delta_\leg^{s,r}(\check{\ul{z}}, \check{\mba} \check{\ul{w}} \check{\mbb})\\ \hat{\alpha}^\vee(u, \mbw_1,\mbw_2)  =  a_1 \cdots a_s\ul{z}b_r\cdots b_1}}  \hat{\aug}^1(\mba) \hat{\aug}^2(\mbb) w^\vee.
    \end{equation*}
\end{lem}

Note the cyclic difference of the inputs in $n^{r|s}_{\mcm_-}$ and $n^{s|r}_{\mcm_-^\vee}$, which arises from taking the adjoint as in Equation~\ref{eq:linear-dual} in Section~\ref{sssec:a-infty-bimod-construct}.  See Figure~\ref{fig:n+enrich} yet again.

In parallel to the construction of $\Aug_+^{\text{Reeb}}(\leg)$ as a subcategory of $\Aug_+(\leg)$, let $\mcm_+^{\text{Reeb}}$ be the length filtration submodule of $\mcm_+$ generated by $\reeb^{12}_\leg$.

\begin{prop} \label{prop:aug-submodule}
    The submodule $\mcm_+^{\text{Reeb}}$ is isomorphic to $\mcm_-$.
\end{prop}

\begin{proof}
    Define a na\"ive morphism $F: \mcm_- \to \mcm_+^{\text{Reeb}}$ on the generators of each morphism space $\mcm_-(\aug^2, \aug^1)$ by $F^{0|0}(\ul{a}^-) = a^+$.  To show that $F$ is, indeed, an $\Aug_+(\leg)$-bimodule isomorphism, we note from Lemma~\ref{lem:mcm+disks} and \ref{lem:mcm-disks} that the structure maps for both bimodules agree under $F$ and its naïve inverse.  In particular, we have that the only terms in $n^{r|s}_{\mcm_+}$ arising from thin disks --- besides those realizing $y^+$ as a strict unit --- take inputs that are not in $\mcm_+^{\text{Reeb}}$. Further, since $\ul{x}$ cannot be an input for the structure map of $\mcm_+^{\text{Reeb}}$, the doubly enriched disks used in Lemmas~\ref{lem:mcm+disks} and \ref{lem:mcm-disks} match.
\end{proof}

Finally, we define the bimodule associated to the circle category.

\begin{defn} \label{defn:circle-bimod}
    The \dfn{circle bimodule} $\mcn$ is the quotient  $\mcm_+/\mcm_-$.
\end{defn}

To better understand these bimodules, we note that $\mcm_+$ and $\mcn$ arise from diagonal bimodules.  We state this relationship more precisely, using $\pi:\Aug_+(\leg) \to \mcc(\leg)$ to denote the canonical projection functor to the quotient category. Again, we assume that all constructions involved are simply perturbed.

\begin{prop}\label{prop:aug-bimod+}
    Assuming all structures are simply perturbed, the positive augmentation bimodule $\mcm_+$ is isomorphic to the diagonal bimodule $\Aug_+(\leg)_\Delta$, while the circle bimodule is isomorphic to the pullback module $\pi^* \mcc(\leg)_\Delta$.
\end{prop}

\begin{proof}
    Define a na\"ive morphism $F: \mcm_+ \to \Aug_+(\leg)_\Delta$ on the generators of each morphism space $\mcm_+(\aug^2, \aug^1)$ by $F^{0|0}(\ul{a}^+) = a^+$. That $F$ is, indeed, an $\Aug_+(\leg)$-bimodule isomorphism follows from the observation made after the proof of Proposition~\ref{prop:bimod-from-link-grading} that by identifying $\leg^{s+r+2}$ with $\leg^{s+1,r+1}$, the structure maps of $\Aug_+(\leg)_\Delta$ are the same as those in Lemma~\ref{lem:mcm+disks}.

    The proof that $\mcn \simeq \pi^* \mcc(\leg)_\Delta$ follows the same observation with appropriate quotients of the domains and codomains of the structure maps taken into account.
\end{proof}

The proposition above implies that the morphism spaces  $\mcn(\aug^2, \aug^1)$ are generated by $x^+$ and $y^+$. Besides those dictated by strict unitality of $y^+$, the only nontrivial maps are
\begin{equation} \label{eq:N-structure-maps}
n_{\mcn}^{1|0}(x^+, \ul{y}^+) = n_{\mcn}^{0|1}(\ul{y}^+, x^+) = x^+.
\end{equation}

We end this section with the following useful structural result that follows immediately from Lemma \ref{lem:submodule-to-cone}.

\begin{prop}\label{prop:aug+cone}
    $\mcm_+$ is the mapping cone of a bimodule morphism $\rho:\mcn[-1] \to \mcm_-$.
\end{prop}

The adjoint $\rho^\vee: \mcm_-^\vee \to \mcn^\vee[1]$ of the morphism in the mapping cone above will be important in the proof of the main theorem in Section~\ref{sec:main-pf}.  As above, we may describe the map in terms of thin disks and doubly enriched disks in $\leg$. Note the cyclic permutation arising from taking the adjoint as in Lemma~\ref{lem:mcm-vee-disks}. 

\begin{lem} \label{lem:rho-vee}
    In the definition of $\rho^\vee: \mcm_-^\vee \to \mcn^\vee[1]$ between simply perturbed bimodules, thin disks yield the terms
    \begin{equation*} \label{eq:rho-dual-thin}
    \begin{array}{ll}
        (\rho^{\vee})^{0|0}(\ul{a}^\vee) = (\aug_1^1(a) + \aug_1^2(a))y^\vee, &(\rho^{\vee})^{1|0}(a^+,\ul{a}^\vee) = (\rho^{\vee})^{0|1}(\ul{a}^\vee, a^+) = y^\vee.
    \end{array}
    \end{equation*}
    
    Doubly enriched disks with mixed enrichment at a base point yield additional terms
    \begin{equation*} \label{eq:rho-dual-thick}
    (\rho^\vee)^{s|r}(\vec{a},\ul{z}^\vee,\cev{b}) \\
    = \sum_{\substack{(u,\mbw_1,\mbw_2) \in \Delta_\leg^{s,r}(\check{\ul{z}},\check{\mba} \ul{t}^{\pm 1} \check{\mbb})\\ \hat{\alpha}^\vee(u, \mbw_1,\mbw_2)  =  a_1\cdots a_s \ul{z} b_r \cdots b_1}}  \hat{\aug}^1(\mba)\hat{\aug}^2(\mbb) x^\vee.
    \end{equation*}
\end{lem}

See the doubly enriched disks on the right side of Figure~\ref{fig:n+enrich}.

\begin{proof}
The construction of the map $\rho$ in Lemma~\ref{lem:submodule-to-cone} shows that, for $w = x^+$ or $y^+$, $\rho^{r|s}(b_r^+, \ldots, b_1^+, \ul{w}, a_1^+, \ldots, a_s^+)$ outputs the same terms as $n^{r|s}_{\mcm_+}$ except for contributions of other $x^+$ and $y^+$ terms.  Thus, the enumeration of the terms arising from thin disks follows from Lemma~\ref{lem:mcm+disks} after taking the adjoint as in Equation~\eqref{eq:linear-dual} in Section~\ref{sssec:a-infty-bimod-construct}.

Similarly, the terms from the doubly enriched disks follow from recognizing that the only possible input into a thick disk for $\rho$ is a maximum. Hence, once again, the enumeration of the terms arising from thick disks follows from Lemma~\ref{lem:mcm+disks} after taking the adjoint as in Equation~\eqref{eq:linear-dual}.
\end{proof}

\begin{ex} \label{ex:running-rho-dual}
We use Example~\ref{ex:running-aug-cat} to illustrate the map $\rho^\vee: \mcm_-^\vee \to \mcn^\vee[1]$ when both augmentations are $\aug_1$ (or $\aug_2$). The non-vanishing components are 
\begin{align*}
    (\rho^\vee)^{0|0}(\ul{a_6}^\vee) &= x^\vee,\\ (\rho^\vee)^{1|0}(a_i^+, \ul{a_i^\vee}) &= y^\vee, \\ (\rho^\vee)^{0|1}(\ul{a_i^\vee},a_i^+) &= y^\vee.
\end{align*}
\end{ex}

\subsection{The Separated Bimodule}
\label{ssec:sep-module} Our second application of Proposition~\ref{prop:bimod-from-link-grading} associates an $\Aug_+(\leg)$-bimodule to the separated $2$-copy of $\leg$. As we shall see in this section, this bimodule encodes the augmentation bimodules $\mcm_+$ and $\mcm_-^\vee$ as parts of an acyclic mapping cone.  The resulting quasi-isomorphism will be a key input into the algebraic machinery of Section~\ref{sec:wrcy} in the proof of the main theorem.

\begin{defn}\label{defn:sep-bimod}
    The \textbf{separated $2$-copy bimodule} $\widehat{\mcm}$ is defined to be the result of  applying Proposition \ref{prop:bimod-from-link-grading} to the separated $2$-copy $\stcleg$.
\end{defn}

As above, we only consider separated $2$-copy bimodules defined by simply perturbed copies of $\leg$. In particular, the base category $\Aug_+(\leg)$ is simply perturbed and the morphisms of $\widehat{\mcm}$ are generated by Reeb chords of a simply perturbed separated $2$-copy $\stcleg$. The structure maps $n_{\widehat{\mcm}}^{r|s}$ count disks in a simply perturbed $(s,r)$-copy of $\stcleg$. Again, $\widehat{\mcm}$ defined in this manner is termed \dfn{simply perturbed}.

In order to find the augmentation bimodules inside $\widehat{\mcm}$, we use Proposition~\ref{prop:p-q-chords} and the discussion afterwards to obtain, for all $\aug^1,\aug^2 \in \obj(\Aug_+(\leg))$,
\begin{equation} \label{eq:sep-bimod}
    \widehat{\mcm}(\aug^2, \aug^1) \simeq  \Hom_+(\aug^2,\aug^1) \oplus \Hom_-^\vee(\aug^2,\aug^1)[-1].
\end{equation}
Recall that, in the decomposition above, Reeb chords that generate $\Hom_+$ are denoted by $q^+$, $x^+$, or $y^+$, while chords that generate $\Hom_-^\vee$ are denoted by $p^-$.

With respect to this decomposition, we find a submodule $\mcq$ of $\widehat{\mcm}$ using the length filtration bimodule construction at length $Z$.

\begin{lem} \label{lem:split-2-copy}
    There exists a submodule $\mcq$ of $\widehat{\mcm}$ so that,  for all $\aug^1,\aug^2 \in \obj(\Aug_+(\leg))$, $\mcq(\aug^2, \aug^1) \simeq \Hom_+(\aug^2,\aug^1)$.   Further, there is a quotient module $\mcp = \widehat{\mcm}/\mcq$ whose underlying vector spaces are generated by equivalence classes of generators of $\Hom_-^\vee(\aug^2, \aug^1)[-1]$.
\end{lem}

We proceed to explicitly characterize the structure maps of $\mcq$ and $\mcp$ in terms of thin disks and doubly enriched disks in $\leg$. In more detail, using Lemma~\ref{lem:sr-structure-map} with Lemma~\ref{lem:s-r-sep-thin-disk-1} for thin disk contributions and Lemma~\ref{lem:separated+dil}(1) for enriched disk contributions yields the following lemma.

\begin{lem} \label{lem:mcq-disks}
    In the definition of \[n_{\mcq}^{r|s}: \Hom_+(\cev{\aug}^2) \otimes \mcq(\aug_1^2,\aug_1^1) \otimes \Hom_+(\vec{\aug}\,^1) \to \mcq(\aug_{r+1}^2, \aug_{s+1}^1)\] for a simply perturbed $\mcq$, thin disks yield the terms dictated by the strict unitality of $y^+$, as well as the terms
    \begin{align*}
        n_{\mcq}^{0|0}(\ul{y}^+) &= \left(\aug_1(a) + \aug_2(a)\right) q^+ &
        n_{\mcq}^{1|0}(q^+,\ul{y}^+) &= n_+^{0|1}(\ul{y}^+,q^+) = q^+ \\
        && n_{\mcq}^{1|0}(x^+,\ul{y}^+) &= n_+^{0|1}(\ul{y}^+,x^+) = x^+
    \end{align*}
    Doubly enriched disks yield additional terms
    \begin{equation*}
    n_{\mcq}^{r|s}(\cev{b},\ul{w}^+,\vec{a}) \\
    = \sum_{\substack{(u,\mbw_1,\mbw_2) \in \Delta_\leg^{s,r}(\check{\ul{z}}, \check{\mba} \check{\ul{w}} \check{\mbb})\\ \hat{\alpha}(u, \mbw_1,\mbw_2)  = a_s \cdots a_1\ul{w}b_1\cdots b_r}}  \hat{\aug}^1(\mba)\hat{\aug}^2(\mbb) z^+.
\end{equation*}
\end{lem}

Next, the characterization of the structure maps $n_{\mcp}$ of $\mcp$ follows from Lemma~\ref{lem:s-r-sep-thin-disk-2} for thin disk contributions and Lemma~\ref{lem:separated+dil}(2) for enriched disk contributions. To explain this characterization, we first emphasize the cyclic permutation in the correspondence between enriched and lifted disks in Lemma~\ref{lem:separated+dil}(2).  Second, the mixed enrichment for the enriched disks in Lemma~\ref{lem:separated+dil}(2) must lie at a Reeb chord in $\leg$. To match the enriched disks in the characterization of $n_\mcp$ with those in the characterization of $n_{\mcm_-^\vee}$ in Lemma~\ref{lem:mcm-vee-disks}, we apply Lemma~\ref{lem:separated+dil}(2) to enriched disks in $\Delta_\leg^{s,r}(\ul{z},\mba \ul{w}\mbb)$ in place of the enriched disks in $\Delta^{r,s}_\leg(\ul{w}, \mbb \ul{w} \mba)$ that are referred to in the statement of the lemma.

\begin{lem}\label{lem:mcp-disks}
    In the definition of \[n_{\mcp}^{s|r}:\Hom_+(\vec{\aug}\,^1) \otimes \mcp(\aug_{s+1}^1,\aug_{r+1}^2) \otimes \Hom_+(\cev{\aug}^2) \to \mcp(\aug_1^1,\aug_1^2)\] for a simply perturbed $\mcp$, thin disks yield the terms dictated by the strict unitality of $y^+$. 
    
    Doubly enriched disks with mixed enrichment at a Reeb chord yield additional terms
    \begin{equation*}
    n_{\mcp}^{s|r}(\vec{a},\ul{z}^-,\cev{b}) \\
    = \sum_{\substack{(u,\mbw_1,\mbw_2) \in \Delta_\leg^{s,r}(\check{\ul{z}}, \check{\mba} \check{\ul{w}} \check{\mbb})\\ \hat{\alpha}^\vee(u, \mbw_1,\mbw_2)  =  a_1 \cdots a_s \ul{z} b_r\cdots b_1}}  \hat{\aug}^1(\mba)\hat{\aug}^2(\mbb) w^-.
    \end{equation*}
\end{lem}

A comparison of the characterizations of the structure maps for $\mcm_\pm$, $\mcq$, and $\mcp$ leads to the following identification.

\begin{prop} \label{prop:separated-id}
    There are $\Aug_+(\leg)$-bimodule isomorphims $\mcm_+ \simeq \mcq$ and $\mcm_-^\vee[-1] \simeq \mcp$. 
\end{prop}

\begin{proof}
    Define a na\"ive morphism $F: \mcm_+ \to \mcq$ by $F^{0|0}(\ul{a}_k^+) = q_k^+$, $F^{0|0}(x^+) = x^+$, $F^{0|0}(y^+) = y^+$. We see from Lemmas~\ref{lem:mcm+disks} and \ref{lem:mcq-disks} that the structure maps of $\mcm_+$ and $\mcq$ agree under this identification, yielding $\mcm_+ \simeq \mcq$.

    Similarly, the na\"ive morphism $G: \mcm_-^\vee \to \mcp$ defined on the generators by $G^{0|0}(a^\vee) = p^-$, combined with Lemmas~\ref{lem:mcm-vee-disks} and \ref{lem:mcp-disks}, shows that $\mcm_-^\vee[-1] \simeq \mcp$.
\end{proof}

Finally, an application of Lemma~\ref{lem:submodule-to-cone} to $(\widehat{\mcm},\mcq)$ yields a bimodule morphism $\eta:\mcp[-1] \to \mcq$ with $\widehat{\mcm} \simeq \cone(\eta)$. The map $\eta$ extends the linear duality map in \cite{duality} to the $A_\infty$ bimodule setting.

\begin{prop}\label{prop:eta-quasi-iso}
    The map $\eta:\mcp[-1] \to \mcq$ is a strictly unital quasi-isomorphism. Thus, by a slight abuse of notation, there exists a bimodule quasi-isomorphism $\eta: \mcm_-^\vee[-2] \to \mcm_+$.
\end{prop}

\begin{proof}
    There is a Legendrian isotopy of $\stcleg$ that shifts $\leg_1$ horizontally, e.g.\ in the $x$ direction, so that the resulting link has no mixed chords. Thus, by the invariance of (bi-)linearized Legendrian contact homology, we see that $\widehat{\mcm}$ is acyclic as a chain-level mapping cone and hence that $\eta$ is a quasi-isomorphism. The strict unitality follows from Remark \ref{rmk:strict-unital-dil}.
\end{proof}

\begin{ex} \label{ex:running-eta-map} Referring back to Example~\ref{ex:running-aug-cat}, we describe the thick disk contribution to the duality morphism $\eta:\mcm_-^\vee[-2] \to \mcm_+$. By Lemma~\ref{lem:separated+dil} and Remark~\ref{rem:qp-thick-disks}, such disks correspond to disks with either two positive punctures or with one positive puncture and a base point on the boundary in Figure~\ref{fig:fig-eight}. Different inputs of augmentations produce different maps in $\eta$. For instance, the thick disks yield the following terms to the collection of maps on $\Hom_+(\aug_1,\ldots,\aug_1) \otimes \mcm_-^\vee(\aug_1,\aug_1) \otimes \Hom_+(\aug_1,\ldots,\aug_1)$: 
\begin{align*}
\eta^{0|0}(\ul{a_1}^\vee) &= a_2^+ + a_3^+ & \eta^{1|0}(a_4^+,\ul{a_2}^\vee) &= a_6^+ & \eta^{2|0}(a_5^+,a_4^+,\ul{a_2}^\vee) &= a_6^+ \\
\eta^{0|0}(\ul{a_2}^\vee) &= a_1^+ + a_6^+ & \eta^{1|0}(a_4^+,\ul{a_7}^\vee) &= a_3^+ & \eta^{2|0}(a_5^+,a_4^+,\ul{a_7}^\vee) &= a_3^+ \\
\eta^{0|0}(\ul{a_3}^\vee) &= a_1^+ + a_7^+ & \eta^{0|1}(\ul{a_6}^\vee,a_4^+) &= a_2^+ & \eta^{0|2}(\ul{a_6}^\vee,a_4^+,a_5^+) &= a_2^+ \\
\eta^{0|0}(\ul{a_4}^\vee) &= a_5^+ & \eta^{0|1}(\ul{a_3}^\vee,a_4^+) &= a_7^+ & \eta^{0|2}(\ul{a_3}^\vee,a_4^+,a_5^+) &= a_7^+ \\
\eta^{0|0}(\ul{a_5}^\vee) &= a_4^+ & \eta^{0|1}(\ul{a_6}^\vee,x^+) &= y^+ && \\
 \eta^{0|0}(\ul{a_6}^\vee) &= a_2^+ + y^+ && && \\
 \eta^{0|0}(\ul{a_7}^\vee) &= a_3^+ && && 
\end{align*}
A direct computation shows that the cohomology ring $H^*\mcm_-^\vee(\aug_1,\aug_1)$ is generated by $\{[a_1^\vee], [a_2^\vee], [a_6^\vee+a_7^\vee]\}$ and $H^*\mcm_+(\aug_1,\aug_1)$ is generated by $\{[a_2^+ + a_3^+],[a_1^+], [y^+]\}$. Thus, $\eta$ is, indeed, a quasi-isomorphism on $\mcm_-^\vee(\aug_1,\aug_1)$.
\end{ex}

In the proof of the main theorem, we will need to identify disks in the definition of $\pi_\mcn \eta: \mcm_-^\vee[-2] \to \mcn$ in a manner similar to that in Lemma~\ref{lem:rho-vee}.  

\begin{lem} \label{lem:pi-eta-disks}
In the definition of $\pi_\mcn \eta: \mcm_-^\vee[-2] \to \mcn$ between simply perturbed bimodules, thin disks yield the following terms to $\pi_\mcn\eta$:
\begin{equation} \label{eq:pi-eta-thin}
\begin{aligned}
    (\pi_\mcn \eta)^{0|0}(\ul{a}^\vee) &= (\aug_1^1(a) + \aug_1^2(a))x^+ & (\pi_\mcn \eta)^{1|1}(x^+,\ul{a}^\vee,a^+) &= x^+\\
        (\pi_\mcn \eta)^{1|0}(a^+,\ul{a}^\vee) &= x^+ & (\pi_\mcn \eta)^{2|0}(x^+,a^+,\ul{a}^\vee) &= x^+ \\ 
        (\pi_\mcn \eta)^{1|0}(x^+,\ul{a}^\vee) &= (\aug_1^1(a) + \aug_1^2(a))x^+ & \\
         (\pi_\mcn \eta)^{0|1}(\ul{a}^\vee,a^+) &= x^+ &
\end{aligned}
\end{equation}

Doubly enriched disks with mixed enrichment at a base point yield additional terms
\begin{equation} \label{eq:pi-eta-thick}
    (\pi_\mcn\eta)^{s|r}(\vec{a},\ul{z}^\vee,\cev{b}) \\
    = \sum_{\substack{(u,\mbw_1,\mbw_2) \in \Delta_\leg^{s,r}(\check{\ul{z}}, \check{\mba} \ul{t}^{\pm 1} \check{\mbb})\\ \hat{\alpha}^\vee(u, \mbw_1,\mbw_2)  =  a_1 \cdots a_s \ul{w}b_r\cdots b_1 }}  \hat{\aug}^1(\mba)\hat{\aug}^2(\mbb) y^+,
\end{equation}
where $m_0 \leq 1$ and $n_0 = 0$ if the mixed enrichment is at $t$, and $m_0 = 0$ if the mixed enrichment is at $t^{-1}$. 
\end{lem}

Once again, see the doubly enriched disks on the right side of Figure~\ref{fig:n+enrich}, though care should be taken with $m_0$ and $n_0$ as in the statement of the lemma.

\begin{proof}
The formulae for the thin disks follows from Lemma~\ref{lem:s-r-sep-thin-disk-3}, while the formula for the doubly enriched disks follows from Lemma~\ref{lem:separated+dil}(3). Once again, we apply Lemma~\ref{lem:separated+dil}(3) to enriched disks in $\Delta_\leg^{s,r}(\ul{z},\check{\mba}\ul{t}^{\pm 1}\check{\mbb})$ in place of the enriched disks in $\Delta^{r,s}_\leg(\ul{w}, \check{\mbb} \ul{t}^{\pm 1} \check{\mba})$ that are referred to in the statement of the lemma.
\end{proof}

\begin{ex}\label{ex:running-homotopy}
    From Example~\ref{ex:running-eta-map}, we see that when all augmentations are $\aug_1$, the only non-vanishing terms in $\pi_\mcn \eta$ arising from a thick disk are
    \begin{equation*}
        (\pi_\mcn\eta)^{0|0}(a_6^\vee) = y^+,\quad (\pi_\mcn\eta)^{0|1}(a_6^\vee,x^+) = y^+.
    \end{equation*}
\end{ex}

\section{Proof of the Main Theorem}
\label{sec:main-pf}

With the algebraic tools of Section~\ref{sec:wrcy} and the geometric structures of Sections~\ref{sec:copies} and \ref{sec:aug-cat} in hand, we are ready to prove Theorem~\ref{thm:main}. The first step is to prove that when the augmentation category of a Legendrian knot is simply perturbed, there is a very weak relative Calabi-Yau structure of dimension $-2$ on the morphism $\rho^\vee: \mcm_-^\vee[-1] \to \mcn^\vee$, which was defined in Proposition~\ref{prop:aug+cone}. This is accomplished in Sections~\ref{ssec:poincare-dual} and \ref{ssec:lch-duality-morphism}. After noting that $(\Aug_+(\leg) \stackrel{\pi}{\to} \mcc(\leg), \mcn[-1] \stackrel{\rho}{\to} \mcm_-)$ is a conical pair, we then apply Proposition \ref{prop:alg-gadget} to obtain a weak right relative Calabi-Yau structure of dimension $2$ on $\pi: \Aug_+(\leg) \to \mcc(\leg)$. This is accomplished in Section~\ref{ssec:weak-CY-final-proof}, where we also lift the requirement that the augmentation category is simply perturbed.

\subsection{The Poincaré Duality Morphism}
\label{ssec:poincare-dual}

We start by describing a weak right Calabi-Yau structure of dimension $1$ on $\mcc$. That is, we specify a bimodule quasi-isomorphism $\kappa: \mcc(\leg)_\Delta \to \mcc(\leg)^\vee[-1]$. The very weak Calabi-Yau structure that we construct in Section~\ref{ssec:lch-duality-morphism} then relies on the pullback of the homotopy inverse $\theta$ of $\kappa$ along $\pi_\mcc:\Aug_+(\leg) \to \mcc(\leg)$. The morphism $\kappa$ itself, however, is part of the relative right Calabi-Yau structure as we ``invert" the very weak Calabi-Yau structure using Proposition~\ref{prop:alg-gadget}. 

In the simply perturbed setting, given that the generators of the underlying vector spaces of $\mcc(\leg)$ are $x^+,y^+$, we denote the dual basis for the underlying vector spaces of $\mcc(\leg)^\vee$ by $x^\vee,y^\vee$. We then have the following proposition.

\begin{prop} 
\label{prop:pd-bimodule}
For a simply perturbed $\mcc(\leg)$, the bimodule pre-morphism $\kappa:\mcc(\leg)_\Delta \to \mcc(\leg)^\vee[-1]$ defined by
\begin{equation} \label{eq:kappa}
    \begin{array}{rlrl}
    \kappa^{0|0}(\ul{y}^+) &= x^\vee,& \kappa^{0|s}(\ul{y}^+, x^+,\ldots,x^+) &= x^\vee, \\
    \kappa^{0|0}(\ul{x}^+) &= y^\vee,& \kappa^{0|s}(\ul{x}^+,x^+,\ldots,x^+) &= y^\vee, 
\end{array}
\end{equation}
with all other components vanishing, is an $A_\infty$ bimodule quasi-isomorphism. Moreover, the bimodule morphism $\theta:\mcc(\leg)^\vee[-1] \to \mcc(\leg)_\Delta$ defined by
\begin{equation} \label{eq:theta}
\begin{array}{rlrl}
    \theta^{0|0}(\ul{y}^\vee) &= x^+,& \theta^{0|1}(\ul{y}^\vee, x^+) &= x^+, \\
    \theta^{0|0}(\ul{x}^\vee) &= y^+,& \theta^{0|1}(\ul{x}^\vee,x^+) &= y^+, 
\end{array}
\end{equation}
with all other components vanishing, is an $A_\infty$ homotopy inverse of $\kappa$ with $\theta\kappa = \bbid_{\mcc(\leg)_\Delta}$ and $\kappa\theta = \bbid_{\mcc(\leg)^\vee}$.
\end{prop}

The proof follows from a direct calculation using strict unitality, the formulae of the structure maps $m_{\mcc(\leg)}$ discussed after Definition~\eqref{defn:circle-category}, and their adjoints.

The motivation for this construction is a Morse-theoretic proof of Poincar\'e Duality for $S^1$.  Given a Morse function $f:S^1 \to \rr$, swap the maximum and minimum by trading $f$ for $-f$, and then compose with the continuation map for a homotopy of $-f$ back to $f$.  In the $A_\infty$ setting, the swap of $-f$ for $f$ extends na\"ively to a bimodule quasi-isomorphism from $\mcn^\vee[-1]$ to $\mcn$. The map $\theta$ above is then the composition of this quasi-isomorphism with an extension of the continuation map $\Phi: \mcn \to \mcn$.  See Remark~\ref{rmk:ctn-map} for further discussion.

\subsection{Very Weak Calabi-Yau Structures} 
\label{ssec:lch-duality-morphism} 

At this point, we have obtained morphisms that fit into the following diagram, where the vertical morphisms are quasi-isomorphisms by Propositions~\ref{prop:eta-quasi-iso} and \ref{prop:pd-bimodule}.

\begin{equation} \label{eq:very-weak-CY-2} \begin{tikzcd}
	\cdots \arrow[r] & \mcm_-^\vee[-2] \arrow[r,"\rho^\vee"] \arrow[d,"\eta"] & \mcn^\vee[-1] \arrow[r,"\pi_\mcn^\vee"] \arrow[d,"\pi_\mcc^*\theta"] & \mcm_+^\vee[-1] \arrow[d,"\eta'"] \arrow[r] & \cdots \\	
	\cdots \arrow[r] &\mcm_+ \arrow[r,"\pi_\mcn"] & \mcn \arrow[r,"\rho"] & \mcm_-[1] \arrow[r] & \cdots
\end{tikzcd}
\end{equation}

Once we show that the left square of the diagram commutes up to homotopy, we will be licensed to use Lemma \ref{lem:very-weak-shortcut} to complete the construction of the very weak Calabi-Yau structure on $\rho^\vee:\mcm_-^\vee[-1] \to \mcn^\vee$. Since the formulae for the non-vanishing components of $\pi_\mcc^* \theta:\mcn^\vee[-1] \to \mcn$ are identical to those of $\theta:\mcc(\leg)^\vee[-1] \to \mcc(\leg)$, for convenience we denote both maps by $\theta$ throughout this section.

\begin{lem}
\label{lem:v-weak-cy-commute}
    $\theta \rho^\vee \sim \pi_{\mcn} \eta$.
\end{lem}

We begin the proof of the lemma by characterizing the morphism $\theta \rho^\vee$ in terms of thin disks and doubly enriched disks in $\leg$; such a characterization for $\pi_\mcn \eta$ already appeared in Lemma~\ref{lem:pi-eta-disks}. We will then write down a homotopy $H$ between the two morphisms, and finally verify that $H$ is, indeed, a homotopy.

First, we use Lemma~\ref{lem:rho-vee} and Proposition~\ref{prop:pd-bimodule} to write down  formulae for $\theta \rho^\vee:\mcm_-^\vee[-2] \to \mcn$. Fix augmentations $\aug^1_1, \ldots, \aug^1_{r+1}$ and $\aug^2_1, \ldots, \aug^2_{s+1}$.  A composition of the formula in Lemma~\ref{lem:rho-vee} with Equation~\eqref{eq:theta} shows that thin disks yield the following terms for $\theta\rho^\vee$:
\begin{equation} \label{eq:theta-rho-thin}
    \begin{aligned}
        (\theta\rho^\vee)^{0|0}(\ul{a}^\vee) &= (\aug_1^1(a) + \aug_1^2(a))x^+ & (\theta\rho^\vee)^{1|1}(a^+,\ul{a}^\vee,x^+) &= x^+ \\
        (\theta\rho^\vee)^{1|0}(a^+,\ul{a}^\vee) &= x^+ & (\theta\rho^\vee)^{0|2}(\ul{a}^\vee,a^+,x^+) &= x^+ \\ 
        (\theta\rho^\vee)^{0|1}(\ul{a}^\vee,a^+) &= x^+  \\
        (\theta\rho^\vee)^{0|1}(\ul{a}^\vee,x^+) &= 
        (\aug_1^1(a) + \aug_1^2(a))x^+ 
    \end{aligned}
\end{equation}
The composition of the formulae in Lemma~\ref{lem:rho-vee} and Equation~\eqref{eq:theta} shows that doubly enriched disks with mixed enrichment at a base point yield additional terms
\begin{equation} \label{eq:theta-rho-thick-1}
    \theta^{0|0}(\rho^\vee)^{s|r}(\vec{a},\ul{z}^\vee,\cev{b}) \\
    = \sum_{\substack{(u,\mbw_1,\mbw_2) \in \Delta_\leg^{s,r}(\check{\ul{z}}, \check{\mbb} \ul{t}^{\pm 1} \check{\mba})\\ \hat{\alpha}^\vee(u, \mbw_1,\mbw_2)  =  a_1\cdots a_s \ul{z} b_r \cdots b_1}}  \hat{\aug}^1(\mba)\hat{\aug}^2(\mbb) y^+.
\end{equation}
and
\begin{equation} \label{eq:theta-rho-thick-2}
    \theta^{0|1}((\rho^\vee)^{s|r-1}(\vec{a},\ul{z}^\vee,\cev{b}),x^+)
    = \sum_{\substack{(u,\mbw_1,\mbw_2) \in \Delta_\leg^{s,r-1}(\check{\ul{z}}, \check{\mbb} \ul{t}^{\pm 1} \check{\mba})\\ \hat{\alpha}^\vee(u, \mbw_1,\mbw_2)  =  a_1\cdots a_s \ul{z} b_r \cdots b_2}}  \hat{\aug}^1(\mba)\hat{\aug}^2(\mbb) y^+,
\end{equation}
where $\cev{b} = (b_r^+, \ldots, b_2^+)$ in this case. All other components of $\theta\rho^\vee$ vanish.

We claim that the pre-morphism $H:\mcm_-^\vee[-2] \to \mcn$ defined by its non-vanishing components
\begin{equation} \label{eq:homotopy}
\begin{aligned}
    H^{0|0}(\ul{a}^\vee) &= (\aug_1^1(a) + \aug_1^2(a))y^+, \\
    H^{1|0}(a^+, \ul{a}^\vee) &= y^+, \\  H^{0|1}(\ul{a}^\vee, a^+) &= y^+, 
\end{aligned}
\end{equation}
is an $A_\infty$ homotopy between  $\pi_\mcn \eta$ and $\theta\rho^\vee$. Since we have explicit formulae for all of the maps, the proof will be a direct verification using the definition of homotopy (cf.\ Definition~\ref{defn:a-infty-homotopy}). 

\begin{rem}\label{rmk:ctn-map}
The continuation morphism $\Psi: \mcn \to \mcn$ discussed after Proposition~\ref{prop:pd-bimodule} and the homotopy $H:\mcm_-^\vee[-2] \to \mcn$ used in the proof of Lemma~\ref{lem:v-weak-cy-commute} can both be interpreted geometrically. In our setup, the $(s,r)$-copy of the separated $2$-copy is obtained via a triple of Morse functions $(f,f,f)$, where the second entry is used to perturb the separated $2$-copy, while the first and third entries are used to perturb the $s$- and $r$-copy. We may also use $(-f,f,-f)$ for the perturbation; see Figure~\ref{fig:perturb-compare} for a comparison. There is a Legendrian isotopy from $\stcleg^{s,r}_{f,f,f}$ to $\stcleg^{s,r}_{-f,f,-f}$ induced by a sequence of Reidemeister III that moves the $y$ chords for the $s$- and $r$-copies around the knot. The corresponding components of the induced DGA map yield $\Phi$ and $H$.
\end{rem}

\begin{figure}
    \centering
    \includegraphics{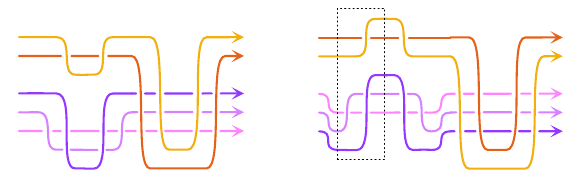}
    \caption{A comparison of the perturbation functions $(f,f,f)$ at left and $(-f,f,-f)$ at right for the separated $2$-copy.  By isotoping the boxed $y$ chords for the $s$- and $r$-copies in the $(-f,f,-f)$ perturbation leftwards around the knot, we recover the $(f,f,f)$ perturbation.}
    \label{fig:perturb-compare}
\end{figure}

Consider an input of the form $(\cev{b},\ul{a}^\vee, \vec{a})$. The verification that $H$ is a homotopy breaks down to what type of generators the input begins and ends with. In particular, we consider the following three cases:

\begin{description}
    \item[Case 1] Both $a_s$ and $b_r$ are in $\reeb_\leg$ (or are empty),
    \item[Case 2] One of $a_s$ and $b_r$ is in $\reeb_f$ and the other is in $\reeb_\leg$ (or is empty),
    \item[Case 3] Both $a_s$ and $b_r$ are in $\reeb_f$. 
\end{description}

\subsubsection{Case 1}

We begin by considering Case 1, detailing two special cases in Claims~\ref{clm:00-input} and \ref{clm:01-input}. The verification of the general case follows the same ideas as in the second claim. The proofs of these and subsequent claims are all structured in the same way: terms in the homotopy equation arising from thin disks cancel by direct computation using the explicit characterizations of $\theta \rho^\vee$, $\pi \eta$, and $H$, while one class of terms arising from thick disks form telescoping sums whose ``ends'' cancel with other terms that also arise from thick disks.  The differences lie in which thick disks are involved and in how the ``ends'' of the telescoping sums are canceled.

\begin{claim} \label{clm:00-input}
    Definition~\ref{defn:a-infty-homotopy} holds for a single input $a^\vee$.
\end{claim}

\begin{proof}
In this case, the homotopy equation from Definition~\ref{defn:a-infty-homotopy} reads as follows:

\begin{equation} \label{eq:clm00}
    (\pi_\mcn \eta)^{0|0}(\ul{a}^\vee) + (\theta\rho^\vee)^{0|0}(\ul{a}^\vee) = H^{0|0}n_{\mcm^\vee_-}^{0|0}(\ul{a}^\vee) + n_\mcn^{0|0}H^{0|0}(\ul{a}^\vee).
\end{equation}

The formulae \eqref{eq:pi-eta-thin} and \eqref{eq:theta-rho-thin} show that terms arising from thin disks that contribute to $(\pi_\mcn \eta)^{0|0}(\ul{a}^\vee)$ and $(\theta\rho^\vee)^{0|0}(\ul{a}^\vee)$ cancel.  On the other side, we know that $n_{\mcn}^{0|0} = 0$ by the discussion after Proposition~\ref{prop:aug-bimod+}, and that no thin disks contribute to $n_{\mcm^\vee_-}^{0|0}$ by Lemma~\ref{lem:mcm-vee-disks}.

Turning to thick disks, we first note that contributions to $(\pi_\mcn \eta)^{0|0}(\ul{a}^\vee)$ and $(\theta\rho^\vee)^{0|0}(\ul{a}^\vee)$ cancel since the formulae in Equations~\eqref{eq:pi-eta-thick} and \eqref{eq:theta-rho-thick-1} match precisely when $r=s=0$.  In particular, thick disk contributions to these terms arise from the same doubly enriched disks.

For the final term in Equation~\eqref{eq:clm00}, we need only verify that 
\begin{equation} \label{eq:claim-a}
    H^{0|0}n_{\mcm^\vee_-}^{0|0}(\ul{a}^\vee) = 0.
\end{equation} 

Lemma~\ref{lem:mcm-vee-disks} shows that $n_{\mcm_-^\vee}^{0|0}(\ul{a}^\vee)$ counts doubly enriched disks with no enrichments other than the mixed enrichment at a Reeb chord. Rewriting the contributions to $n^{0|0}_{\mcm_-}(\ul{a}^\vee)$ in terms of the collection of all disks in $\leg$ with a positive corner at $a$ (which we denote $\Delta(a, \mba)$ for notational compactness) yields the following formula, illustrated in Figure~\ref{fig:multiple-enrichments}:

\begin{figure}
\labellist
\small\hair 2pt
 \pinlabel {$a^\vee$} [r] at 7 42
 \pinlabel {$a_1^\vee$} [t] at 43 4
 \pinlabel {$\aug^2_1(a_2)$} [l] at 78 42
 \pinlabel {$\aug^2_1(a_3)$} [b] at 43 78
 \pinlabel {$a^\vee$} [r] at 126 42
 \pinlabel {$\aug^1_1(a_1)$} [t] at 161 4
 \pinlabel {$a_2^\vee$} [l] at 199 42
 \pinlabel {$\aug^2_1(a_3)$} [b] at 161 78
 \pinlabel {$a^\vee$} [r] at 245 42
 \pinlabel {$\aug^1_1(a_1)$} [t] at 279 4
 \pinlabel {$\aug^1_1(a_2)$} [l] at 316 42
 \pinlabel {$a_3^\vee$} [b] at 279 78
\endlabellist
    \centering
    \includegraphics{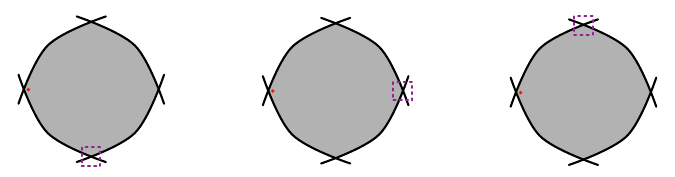}
    \caption{The disk in $\leg$ pictured has multiple possible mixed enrichments, each of which could contribute to $n^{0|0}_{\mcm_-}(a^\vee)$, depending on the value the augmentations $\aug_1^1$ and $\aug_1^2$ take on the non-enriched negative corners.  The corners are labeled with these augmentation values and the potential output that appear in Equation~\eqref{eq:n-from-one-disk}.}
    \label{fig:multiple-enrichments}
\end{figure}
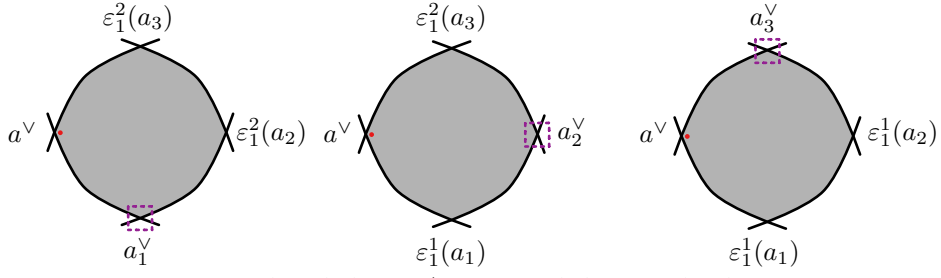

\begin{equation} \label{eq:n-from-one-disk}
    n_{\mcm^\vee_-}^{0|0}(\ul{a}^\vee) = \sum_{\Delta(a,\mba)}\sum_{k = 1}^n \aug_1^2(a_1)\cdots\aug_1^2(a_{k-1})\aug_1^1(a_{k+1})\cdots \aug_1^1(a_n) a_k^\vee.
\end{equation}
It follows that
\begin{multline*}
    H^{0|0}n_{\mcm^\vee_-}^{0|0}(\ul{a}^\vee) = \sum_{\Delta(a,\mba)}\sum_{k = 1}^n (\aug_1^2(a_1)\cdots\aug_1^2(a_{k-1})\aug_1^1(a_k)\aug_1^1(a_{k+1})\cdots \aug_1^1(a_n) \\
    + \aug_1^2(a_1)\cdots\aug_1^2(a_{k-1})\aug_1^2(a_k)\aug_1^1(a_{k+1})\cdots \aug_1^1(a_n))y^+.
\end{multline*}
This is a telescoping sum, and hence we obtain
\begin{align*}
    H^{0|0}n_{\mcm^\vee_-}^{0|0}(\ul{a}^\vee) &= \sum_{\Delta(a,\mba)}(\aug_1^2(a_1)\cdots \aug_1^2(a_n) + \aug_1^1(a_1)\cdots \aug_1^1(a_n))y^+ \\
    &= (\aug_1^2 \circ \partial(a) + \aug_1^1 \circ \partial(a))y^+ \\
    &= 0,
\end{align*}
where $\partial$ is the differential of the Chekanov-Eliashberg DGA of $\leg$, and hence has the property that $\aug_1^2 \circ \partial = \aug_1^1 \circ \partial = 0$.  In all, we have shown that Equation~\eqref{eq:claim-a} holds, and the claim follows.
\end{proof}

\begin{claim} \label{clm:01-input}
    Definition~\ref{defn:a-infty-homotopy} holds for inputs of the form $(\ul{a}^\vee,a_j^+)$, where $a_j \in \reeb_\leg$. 
\end{claim}

\begin{proof}
In this case, the homotopy equation from Definition~\ref{defn:a-infty-homotopy} reads as follows:
\begin{multline} \label{eq:clm01}
    (\pi_\mcn \eta)^{0|1}(\ul{a}^\vee,a_j^+) + (\theta\rho^\vee)^{0|1}(\ul{a}^\vee,a_j^+) = \\ H^{0|0}n_{\mcm_-^\vee}^{0|1}(\ul{a}^\vee,a_j^+) + H^{0|1}(n_{\mcm_-^\vee}^{0|0}(\ul{a}^\vee),a_j^+) 
    + H^{0|1}(\ul{a}^\vee,m_+^1(a_j^+))\\
    + n_\mcn^{0|1}(H^{0|0}(\ul{a}^\vee),a_j^+) + n_\mcn^{0|0}(H^{0|1}(\ul{a}^\vee,a_j^+)).
\end{multline}
Both $n_\mcn^{0|1}(H^{0|0}(\ul{a}^\vee),a_j^+)$ and $n_\mcn^{0|0}(H^{0|1}(\ul{a}^\vee,a_j^+))$ vanish according to the definition of $n_\mcn$. For the remaining terms, we follow a similar analysis of thin and thick disks as in the previous proof..

First, the formulae \eqref{eq:pi-eta-thin} and \eqref{eq:theta-rho-thin} show that terms arising from thin disks that contribute to $(\pi_\mcn \eta)^{0|1}(\ul{a}^\vee,a_j^+)$ and $(\theta\rho^\vee)^{0|1}(\ul{a}^\vee,a_j^+)$ cancel. Lemma~\ref{lem:mcm-vee-disks} shows that no thin disks contribute to other terms in the equation.

Turning to thick disks, we see that contributions to $(\pi_\mcn \eta)^{0|1}(\ul{a}^\vee,a_j^+)$ and $(\theta\rho^\vee)^{0|1}(\ul{a}^\vee,a_j^+)$ cancel since formulae in Equations~\eqref{eq:pi-eta-thick} and \eqref{eq:theta-rho-thick-1} match precisely when $r=0$ and $s = 1$. Thus, we need only verify that, for thick disk contributions, we have
\begin{equation}\label{eq:claim-b}
    H^{0|0}n_{\mcm_-^\vee}^{0|1}(\ul{a}^\vee,a_j^+) + H^{0|1}(n_{\mcm_-^\vee}^{0|0}(\ul{a}^\vee),a_j^+)
    + H^{0|1}(\ul{a}^\vee,m_+^1(a_j^+)) = 0.
\end{equation}

Lemma~\ref{lem:mcm-vee-disks} shows that $n_{\mcm_-^\vee}^{0|1}(\ul{a}^\vee,a_j^+)$ counts doubly enriched disks with a pure enrichment in $\mbw_1$ at $a_j$ and a mixed enrichment at a Reeb chord counterclockwise between $a$ and $a_j$. Rewriting these contributions in terms of disks in $\leg$ with a positive corner at $a$ and negative corners counterclockwise at $a_{i_1},\ldots,a_{i_m},a_j,a_{j_n},\ldots,a_{j_1}$ yields the following formula, illustrated in Figure~\ref{fig:multiple-enrichments-2}:

\begin{multline} \label{eq:n01-from-one-disk}
    n_{\mcm_-^\vee}^{0|1}(\ul{a}^\vee,a_j^+) = \sum_{\Delta(a,\mba)}\sum_{k = 1}^m \aug_1^2(a_{i_1})\cdots \aug_1^2(a_{i_{k-1}})\aug_2^1(a_{i_{k+1}}) \\
    \cdots\aug_2^1(a_{i_m})\aug_1^1(a_{j_1})\cdots \aug_1^1(a_{j_n})a_{i_k}^\vee.
\end{multline}

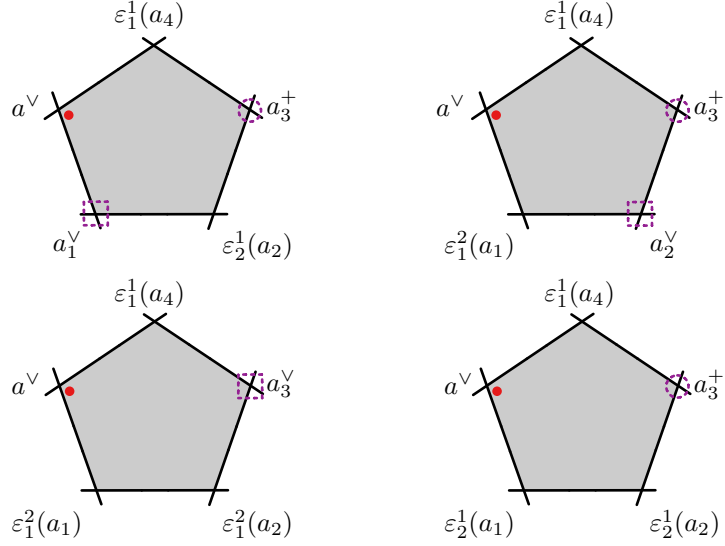
\begin{figure}
\labellist
\small\hair 2pt
 \pinlabel {$a^\vee$} [r] at 15 160
 \pinlabel {$a_1^\vee$} [tr] at 30 114
 \pinlabel {$\aug_2^1(a_2)$} [tl] at 80 114
 \pinlabel {$a_3^+$} [l] at 96 160
 \pinlabel {$\aug^1_1(a_4)$} [b] at 54 187
 \pinlabel {$a^\vee$} [r] at 175 160
 \pinlabel {$\aug^2_1(a_1)$} [tr] at 193 114
 \pinlabel {$a_2^\vee$} [tl] at 241 114
 \pinlabel {$a_3^+$} [l] at 257 160
 \pinlabel {$\aug^1_1(a_4)$} [b] at 216 187
 \pinlabel {$a^\vee$} [r] at 15 55
 \pinlabel {$\aug^2_1(a_1)$} [tr] at 30 9
 \pinlabel {$\aug^2_1(a_2)$} [tl] at 80 9
 \pinlabel {$a_3^\vee$} [l] at 96 55
 \pinlabel {$\aug^1_1(a_4)$} [b] at 54 82
 \pinlabel {$a^\vee$} [r] at 178 55
 \pinlabel {$\aug^1_2(a_1)$} [tr] at 193 9
 \pinlabel {$\aug^1_2(a_2)$} [tl] at 241 9
 \pinlabel {$a_3^+$} [l] at 257 55
 \pinlabel {$\aug^1_1(a_4)$} [b] at 216 82
\endlabellist
    \centering
    \includegraphics{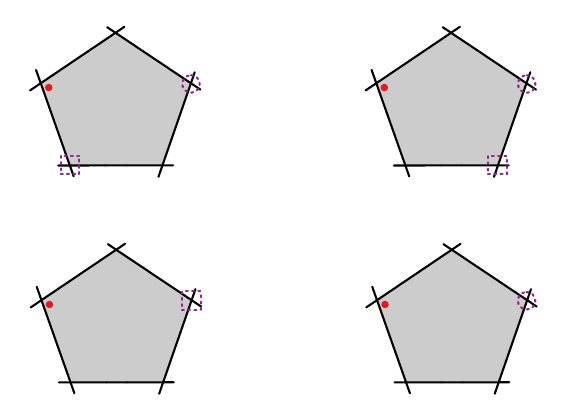}
    \caption{The disk in $\leg$ pictured has multiple possible enrichments.  The disks in the top row  could contribute to $n^{0|1}_{\mcm_-^\vee}(\ul{a}^\vee, a_j^+)$, depending on the value the augmentations take on the non-enriched negative corners; here, $j=3$.  The disk at bottom left could contribute to $H^{0|1}(n_{\mcm_-^\vee}^{0|0}(\ul{a}^\vee), a_j^+)$, while the disk at bottom right could contribute to $H^{0|1}(\ul{a}^\vee, m^1_+(a_j^+))$.}
    \label{fig:multiple-enrichments-2}
\end{figure}

As above, the telescoping nature of the sum yields the following:
\begin{multline} \label{eq:claim01-1}
    H^{0|0}n_{\mcm_-^\vee}^{0|1}(\ul{a}^\vee,a_j^+)  = \sum_{\Delta(a, \mba)}(\aug_1^2(a_{i_1})\cdots\aug_1^2(a_{i_m})\aug_1^1(a_{j_1})\cdots \aug_1^1(a_{j_n}) \\
    + \aug_2^1(a_{i_1})\cdots\aug_2^1(a_{i_m})\aug_1^1(a_{j_1})\cdots \aug_1^1(a_{j_n}))y^+.
\end{multline}

Next, we consider $H^{0|1}(n^{0|0}_{\mcm_-^\vee}(\ul{a}^\vee), a_j^+)$. This term can only be nonzero when $n^{0|0}_{\mcm_-^\vee}(\ul{a}^\vee) = a_j^\vee$, and Lemma~\ref{lem:mcm-vee-disks} shows that such contributions arise from the same underlying disks as above, but with no enrichment other than the mixed enrichment at $a_j$. The result is the following:
\begin{equation} \label{eq:claim01-2}
    H^{0|1}(n^{0|0}_{\mcm_-^\vee}(\ul{a}^\vee),a_j^+) = \sum_{\Delta(a,\mba)} \aug_1^2(a_{i_1})\cdots \aug_1^2(a_{i_m})\aug_1^1(a_{j_1})\cdots \aug_1^1(a_{j_n})y^+.
\end{equation}

Finally, $H^{0|1}(\ul{a}^\vee, m_+^1(a_j^+))$ is non-vanishing only if $m_+^1(a_j^+) = a^+$, and Lemma~\ref{lem:aug+disks} shows that such contributions arise from the same disks as above with a single enrichment at $a_j$.  This yields a similar formula as above, but with a different set of augmentations involved due to the difference of the domains:
\begin{equation} \label{eq:claim01-3}
    H^{0|1}(\ul{a}^\vee, m_+^1(a_j^+)) = \sum_{\Delta(a,\mba)} \aug_2^1(a_{i_1})\cdots \aug_2^1(a_{i_m})\aug_1^1(a_{j_1})\cdots \aug_1^1(a_{j_n})y^+.
\end{equation}

Putting the formulae \eqref{eq:claim01-1}, \eqref{eq:claim01-2}, and \eqref{eq:claim01-3} together, we see that the desired Equation~\eqref{eq:claim-b} holds.
\end{proof}

As mentioned above, a similar argument using disks with more pure enrichments works for longer inputs that begin and end with generators in $\reeb_\leg$.

\subsubsection{Case 2}

Next, we verify the homotopy equation when one of $a_s$ and $b_r$ is in $\reeb_f$ or the other is in $\reeb_\leg$ (or is empty). As above, we detail three special cases in Claims~\ref{clm:01-y-input}, ~\ref{clm:01-x-input}, and \ref{clm:11-x-input}.  The remaining cases either follow from symmetric or slightly generalized arguments with no new ideas necessary. 

We first examine a case with a $y^+$ as part of the input.

\begin{claim} \label{clm:01-y-input}
    Definition~\ref{defn:a-infty-homotopy} holds for inputs of the form $(\ul{a}^\vee,y^+)$.
\end{claim}

\begin{proof}
Having $y^+$ as one of the input variables implies that only thin disks are involved in the verification of the homotopy equation, which is of the same form as  Equation~\eqref{eq:clm01} with $y^+$ in place of $a_j^+$. The formulae in Equations~\eqref{eq:pi-eta-thin} and \eqref{eq:theta-rho-thin} show that the terms in $(\pi_\mcn \eta)^{0|1}$ and  $(\theta\rho^\vee)^{0|1}$ vanish. Of the remaining terms, $n_\mcn^{0|0}(H^{0|1}(\ul{a}^\vee,y^+))$ and $H^{0|1}(n_{\mcm_-^\vee}^{0|0}(\ul{a}^\vee),y^+)$ vanish by the definition of $H$, while $H^{0|1}(\ul{a}^\vee,m_+^1(y^+))$ vanishes since $m_+^1(y^+)=0$ for simply perturbed augmentation categories. Finally, both $H^{0|0}n_{\mcm_-^\vee}^{0|1}(\ul{a}^\vee,y^+)$ and 
$n_\mcn^{0|1}(H^{0|0}(\ul{a}^\vee),y^+)$ are equal to $(\aug_1^1(a) + \aug_1^2(a))y^+$, and hence cancel.
\end{proof}

A similar computation holds for any inputs in Case 2 that begin or end with $y^+$.  Inputs that begin or end with $x^+$ are more subtle, as they involve thick disks as well as thin ones.

\begin{claim} \label{clm:01-x-input}
    Definition~\ref{defn:a-infty-homotopy} holds for inputs of the form $(\ul{a}^\vee,x^+)$.
\end{claim}

\begin{proof}
In parallel to Equation~\eqref{eq:clm01}, we want to verify that 
\begin{equation} \label{eq:claim-01-x}
\begin{split}
    (\pi_\mcn\eta)^{0|1}&(\ul{a}^\vee, x^+) + \theta^{0|0}(\rho^\vee)^{0|1}(\ul{a}^\vee,x^+) + \theta^{0|1}((\rho^\vee)^{0|0}(\ul{a}^\vee),x^+) = \\ &H^{0|0}n_{\mcm_-^\vee}^{0|1}(\ul{a}^\vee,x^+)   + H^{0|1}(n_{\mcm_-^\vee}^{0|0}(\ul{a}^\vee),x^+) + H^{0|1}(\ul{a}^\vee,m_+^1(x^+)) \\
    &+ n_\mcn^{0|0}H^{0|1}(\ul{a}^\vee,x^+) + n_\mcn^{0|1}(H^{0|0}(\ul{a}^\vee),x^+).
\end{split}
\end{equation}
By definition of $H$ and of $n_\mcn$, observe that the terms $H^{0|1}(n_{\mcm_-^\vee}^{0|0}(\ul{a}^\vee),x^+)$ and $n_\mcn^{0|0}(H^{0|1}(\ul{a}^\vee,x^+))$ both vanish.  As in previous claims, terms coming from the thin disks cancel by direct computation using Equations~\eqref{eq:pi-eta-thin}, \eqref{eq:theta-rho-thin} and \eqref{eq:homotopy}. Further, the thick disks that contribute terms to $H^{0|0}n_{\mcm_-^\vee}^{0|1}(\ul{a}^\vee,x^+)$, $H^{0|1}(\ul{a}^\vee,m_+^1(x^+))$, and $\theta^{0|1}(\rho^\vee)^{0|0}(\ul{a}^\vee),x^+)$ cancel using a telescoping sum argument similar to that in the proof of Claim~\ref{clm:01-input}.

The thick disks beyond those that appear in the telescoping sum contribute to $(\pi_\mcn\eta)^{0|1}(\ul{a}^\vee, x^+)$ and $\theta^{0|0}\rho^{0|1}(\ul{a}^\vee,x^+)$.  Equations~\eqref{eq:pi-eta-thick} and \eqref{eq:theta-rho-thick-1} show that contributions come from doubly enriched disks with a positive corner at $a$, a pure enrichment in $\mbw_2$ at $t^{\pm 1}$, and the mixed enrichment at $t^{\pm 1}$. There are two possibilities for the underlying disks.  First, the $x^+$ could correspond to a pure enrichment at index $0$ (i.e.\ $n_0=1$, which is possible for both $\theta \rho^\vee$ and $\pi_\mcn\eta$).  Second, the $x^+$ could correspond to a pure enrichment at a nonzero index, which arises when the disk passes through the base point at least twice.  In either case, the terms in $\theta \rho^\vee$ and $\pi_\mcn\eta$ cancel.
\end{proof}

In fact, a similar computation holds for any inputs in Case 2 that end with $x^+$ because enrichments at index $0$ for the $t^{-1}$ markings in Lemma~\ref{lem:pi-eta-disks} and in \eqref{eq:theta-rho-thick-1} and \eqref{eq:theta-rho-thick-2} are handled similarly when the $x^+$ coefficients come after $\ul{a}^\vee$. For inputs that begin with $x^+$, the argument can be more subtle because of the asymmetry in how basepoints are handled in Lemma~\ref{lem:pi-eta-disks}.  We detail the case with inputs of the form $(x^+, \ul{a}^\vee, a_j^+)$; other cases beginning with $x^+$ follow the same pattern with no new ideas.

\begin{claim} \label{clm:11-x-input}
    Definition~\ref{defn:a-infty-homotopy} holds for inputs of the form $(x^+, \ul{a}^\vee, a_j^+)$, where $a_j \in \reeb_\leg$.
\end{claim}

\begin{proof}
In this case, we want to verify that 
\begin{multline} \label{eq:claim-11-x}
    (\pi_\mcn\eta)^{1|1}(x^+,\ul{a}^\vee, a_j^+) + \theta^{0|0}(\rho^\vee)^{1|1}(x^+,\ul{a}^\vee,a_j^+) \\
    = H^{0|0}n_{\mcm_-^\vee}^{1|1}(x^+,\ul{a}^\vee,a_j^+) + H^{0|1}(n_{\mcm_-^\vee}^{1|0}(x^+,\ul{a}^\vee),a_j^+) \\
    + n_\mcn^{0|1}(x^+,H^{0|1}(\ul{a}^\vee,a_j^+)),
\end{multline}
where we only display terms that do not vanish a priori. The analysis is similar to the previous proof with one new ingredient. Again, terms coming from the thin disks cancel by direct computation. 

Turning to thick disks, a similar analysis to that in the proof of Claim~\ref{clm:01-input} will show that the thick disk contributions to $H^{0|0}n_{\mcm_-^\vee}^{1|1}(x^+,\ul{a}^\vee,a_j^+)$ and $H^{0|1}(n_{\mcm_-^\vee}^{1|0}(x^+,\ul{a}^\vee),a_j^+)$ will cancel with \emph{some} of the thick disk contributions to either $\theta^{0|0}(\rho^\vee)^{1|1}(x^+,\ul{a}^\vee,a_j^+)$ or $\pi_\mcn^{0|0}\eta^{1|1}(x^+,\ul{a}^\vee, a_j^+)$. In more detail, Lemma~\ref{lem:mcm-vee-disks} shows that thick disk contributions to $n_{\mcm_-^\vee}^{1|1}(x^+,\ul{a}^\vee,a_j^+)$ arise from doubly enriched disks with a positive corner at the Reeb chord $a$, a pure enrichment in $\mbw_1$ at $t^{\pm 1}$, a pure enrichment in $\mbw_2$ at $a_j$, and mixed enrichment at a Reeb chord that lies counterclockwise between $t^{\pm 1}$ and $a_j$. As in the proof of Claim~\ref{clm:01-input}, for each of the underlying disks, $H^{0|0}n_{\mcm_-^\vee}^{1|1}(x^+,\ul{a}^\vee,a_j^+)$ is a telescoping sum. One end of this sum cancels with a term in $H^{0|1}(n_{\mcm_-^\vee}^{1|0}(x^+,\ul{a}^\vee),a_j^+)$, which comes from doubly enriching the disk involved with a pure enrichment in $\mbw_1$ at $t^{\pm 1}$ and mixed enrichment at $a_j$ that lies counterclockwise between $t^{\pm 1}$ and $a$.

The other end of the telescoping sum is more subtle to cancel.  Equations~\eqref{eq:pi-eta-thick} and \eqref{eq:theta-rho-thick-1} show that the other end of the telescoping sum cancels with a term in either $\pi_\mcn^{0|0}\eta^{1|1}(x^+,\ul{a}^\vee, a_j^+)$ or $\theta^{0|0}(\rho^\vee)^{1|1}(x^+,\ul{a}^\vee,a_j^+)$, depending on whether the disk involved has a pure enrichment at $t$ or $t^{-1}$ for its contribution to $n_{\mcm_-^\vee}^{1|1}(x^+,\ul{a}^\vee,a_j^+)$. Specifically, if the enrichment is at $t$, the sum cancels with a term in $\pi_\mcn^{0|0}\eta^{1|1}(x^+,\ul{a}^\vee, a_j^+)$ that comes from doubly enriching the disk involved with a pure enrichment $0 \in \mbw_1$ at $t$, a pure enrichment in $\mbw_2$ at $a_j$, and the mixed enrichment at $t$ again; the disk does not contribute to $\theta^{0|0}(\rho^\vee)^{1|1}(x^+,\ul{a}^\vee,a_j^+)$ since the index $0$ cannot be purely enriched if the mixed enrichment is a $t$. Otherwise, the sum cancels with a term in $\theta^{0|0}(\rho^\vee)^{1|1}(x^+,\ul{a}^\vee,a_j^+)$ that comes from doubly enriching the disk involved with a pure enrichment $0 \in \mbw_1$ at $t^{-1}$, a pure enrichment in $\mbw_2$ at $a_j$, and the mixed enrichment at $t^{-1}$ again; the disk does not contribute to $\pi_\mcn^{0|0}\eta^{1|1}(x^+,\ul{a}^\vee, a_j^+)$ according to Lemma~\ref{lem:pi-eta-disks}.

Finally, Equation~\eqref{eq:pi-eta-thick} and Equation~\eqref{eq:theta-rho-thick-1} show that the remaining thick disk contributions to $\pi_\mcn^{0|0}\eta^{1|1}(x^+,\ul{a}^\vee, a_j^+)$ and $\theta^{0|0}(\rho^\vee)^{1|1}(x^+,\ul{a}^\vee,a_j^+)$ both arise from doubly enriched disks with a positive corner at $a$, a pure enrichment in $\mbw_1$ at $t^{\pm 1}$, a pure enrichment in $\mbw_2$ at $a_j$, and  mixed enrichment at $t^{\pm 1}$ that lies \emph{strictly} between the pure enrichments $t^{\pm 1}$ and $a_j$ counterclockwise. Thus, they cancel out.
\end{proof}

\subsubsection{Case 3}

It remains to consider inputs in Case 3. If either $a_r$ or $b_s$ is $y$, only thin disks may be involved and it is a direct check that all terms in the homotopy equation vanish. Thus, we are led to the case $a_r = b_s = x$. As above, we detail the case when the input is $(x^+, \ul{a}^\vee, x^+)$. The remaining cases follow similar arguments with no new ideas necessary.

\begin{claim} \label{clm:x-x-input}
    Definition~\ref{defn:a-infty-homotopy} holds for inputs of the form $(x^+, \ul{a}^\vee, x^+)$.
\end{claim}
\begin{proof}
In parallel to Equation~\eqref{eq:claim-11-x}, we want to verify that 
\begin{multline}
    (\pi_\mcn\eta)^{1|1}(x^+,\ul{a}^\vee, x^+) + \theta^{0|0}(\rho^\vee)^{1|1}(x^+,\ul{a}^\vee,x^+) + \theta^{0|1}((\rho^\vee)^{1|0}(x^+,\ul{a}^\vee),x^+)  \\
    = H^{0|0}n_{\mcm_-^\vee}^{1|1}(x^+,\ul{a}^\vee,x^+), 
\end{multline}
where, as above, we only present terms that do not vanish a priori. Observe that there is no thin disk contribution to the equation. Turning to thick disks, first, the thick disks that contribute terms to $H^{0|0}n_{\mcm_-^\vee}^{1|1}(x^+,\ul{a}^\vee,x^+)$ cancel with \emph{some} of the thick disk contributions to $\theta^{0|1}((\rho^\vee)^{1|0}(x^+,\ul{a}^\vee),x^+)$ and either to $\theta^{0|0}(\rho^\vee)^{1|1}(x^+,\ul{a}^\vee,x^+)$ or $(\pi_\mcn\eta)^{1|1}(x^+,\ul{a}^\vee, x^+)$ using a telescoping sum argument as in Claim~\ref{clm:11-x-input}. 

The only new feature in this case is that Equation~\eqref{eq:theta-rho-thick-2} shows that there is thick disk contribution to the term $\theta^{0|1}((\rho^\vee)^{1|0}(x^+,\ul{a}^\vee),x^+)$ that comes from doubly enriched disks with a pure enrichment $0 \in \mbw_1$ at $t^{-1}$ and the mixed enrichment at $t^{-1}$ again. The same underlying disks contribute to $\theta^{0|0}(\rho^\vee)^{1|1}(x^+,\ul{a}^\vee,x^+)$ by pure enrichments $0 \in \mbw_1, \mbw_2$ at $t^{-1}$ and the mixed enrichment at $t^{-1}$ again. Thus, they cancel out.

Finally, Equation~\eqref{eq:pi-eta-thick} and Equation~\eqref{eq:theta-rho-thick-1} show that the remaining thick disk contributions to $\pi_\mcn^{0|0}\eta^{1|1}(x^+,\ul{a}^\vee, x^+)$ and $\theta^{0|0}(\rho^\vee)^{1|1}(x^+,\ul{a}^\vee,x^+)$ cancel as before.
\end{proof}

\subsection{Weak Relative Calabi-Yau Structures}
\label{ssec:weak-CY-final-proof}

In summary, using Lemma \ref{lem:very-weak-shortcut}, we obtain by Proposition \ref{prop:separated-id}, Proposition \ref{prop:pd-bimodule}, and Lemma \ref{lem:v-weak-cy-commute} the promised very weak relative Calabi-Yau structure on $\rho:\mcn[-1] \to \mcm_-$. In sum, we obtain:

\begin{prop}\label{prop:v-weak-cy-rho}
    Suppose $\Aug_+(\leg)$ is simply perturbed. There is a very weak relative Calabi-Yau structure of dimension $-2$ on the morphism $\rho:\mcn[-1] \to \mcm_-$ defined in Proposition \ref{prop:aug+cone}.
\end{prop}

\begin{proof}[Proof of Theorem~\ref{thm:main}] We begin by using Proposition $\ref{prop:aug+cone}$, which shows that  $(\Aug_+(\leg) \stackrel{\pi_\mcc}{\to} \mcc(\leg), \mcn[-1] \stackrel{\rho}{\to} \mcm_-)$ is a conical pair. The weak relative Calabi-Yau structure asserted in Theorem~\ref{thm:main} for a simply perturbed Legendrian then follows from Propositions \ref{prop:alg-gadget} and \ref{prop:v-weak-cy-rho}.

It remains to show that the existence of a weak relative Calabi-Yau structure does not depend on the perturbation chosen in the definition of $\Aug_+$.  Suppose $\Aug_+(\leg)$ and $\mcc(\leg)$ are constructed by a general Morse function $f:\leg \to \rr$. Let $\Aug_+'(\leg)$ and $\mcc'(\leg)$ be simply perturbed. The invariance of the augmentation category \cite[Proposition 4.22]{nrssz:aug-sheaf} yields an explicit $A_\infty$ equivalence $F:\Aug_+(\leg) \to \Aug_+'(\leg)$. An examination of the formula of $F$ in the proof of \cite[Proposition 4.22]{nrssz:aug-sheaf} shows that $F$ restricts to a functor $\bar{F}:\mcc(\leg) \to \mcc'(\leg)$. Moreover, the formula of $F_1$ shows that $\bar{F}_1$ is surjective with the same acyclic kernel as $F_1$. Hence, $\bar{F}_1$ is a quasi-isomorphism.  In summary, we have the following commutative diagram of $A_\infty$ functors, where the vertical functors are $A_\infty$ equivalences and the horizontal functors are projections.
\begin{equation}
\label{eq:functor-invariance}
 \begin{tikzcd}
	\Aug_+(\leg) \arrow[d,swap, "F"] \arrow[r,"\pi_\mcc"] & \mcc(\leg) \arrow[d,"\bar{F}"] \\	\Aug_+'(\leg)
	 \arrow[r, "\pi_{\mcc'}"] & \mcc'(\leg)
\end{tikzcd}
\end{equation}

The idea for the rest of the proof is to translate the information in Diagram~\eqref{eq:functor-invariance} into the language of Hochschild homology, then use the functoriality of Hochschild homology discussed in Section~\ref{sssec:HH-functor} to transfer the class that defines the weak relative Calabi-Yau structure to the non-simply-perturbed setting, and finally translate back to the $A_\infty$ language to prove non-degeneracy.

By the functoriality of Hochschild cohomology under functors of $A_\infty$ categories (see the formula~\eqref{eq:hoch-cat-functor}), Diagram~\eqref{eq:functor-invariance} induces the following diagram of Hochschild cochain complexes, where the vertical maps are quasi-isomorphisms.
\begin{equation}
\label{eq:hoch-invariance}
 \begin{tikzcd}
	\CH^*(\Aug_+(\leg))   & \CH^*(\mcc(\leg)) \arrow[l,swap,"\pi^*_\mcc"] \\	
    \CH^*(\Aug_+'(\leg) \arrow[u, "F^*"]  &  \arrow[l, swap, "\pi_{\mcc'}^*"] \arrow[u,swap,"\bar{F}^*"] \CH^*(\mcc'(\leg))
\end{tikzcd}
\end{equation}
Let $(\phi',\sigma') \in \HH^*(\pi_{\mcc'})$ be the weak relative CY structure constructed on the simply perturbed categories. Consider its pushforward $(\phi,\sigma) := (F^*\phi', \bar{F}^*\sigma')$ in $\CH^*(\pi_{\mcc})$. We claim that $(\phi,\sigma)$ constitutes a weak relative CY structure of $\pi_\mcc$.

First, we verify that $(\phi,\sigma)$ is closed. Indeed, $\sigma$ is closed in $\CH^*(\mcc(\leg))$ since $\sigma'$ is closed in $\CH^*(\mcc'(\leg))$. Further, by the commutativity of diagram~\eqref{eq:hoch-invariance}, we have \[b^*\phi + \pi_\mcc^*\sigma = b^*F^*\phi' + \pi_\mcc^* \bar{F}^*\sigma' = F^*(b^*\phi'+\pi_{\mcc'}^*\sigma') = 0.\] 

It remains to verify the non-degeneracy condition on $(\phi,\sigma)$. Let $\zeta$ and $\zeta'$ denote the leftmost vertical morphisms in Diagram~\eqref{eq:weak-right-cy} for $(\phi,\sigma)$ and $(\phi',\sigma')$, respectively.\footnote{The notation $\zeta'$ here is meant to fit with other local notation for morphisms, and should not be confused with the map $\zeta'$ on the right side of Diagram~\eqref{eq:weak-right-cy}.} Our goal is to show that $\pi_{\mcc}^*\Psi(\sigma)$ and $\zeta$ are quasi-isomorphisms, where $\Psi:\CH^*(\mcc(\leg)) \to \hom_{\mcc(\leg)}(\mcc(\leg),\mcc(\leg)^\vee)$ is defined as in Proposition~\ref{prop:hoch-a-infty}. The proof will then follow from the Five lemma. 

We first prove that $\pi_{\mcc}^*\Psi(\sigma)$ is a quasi-isomorphism. Consider the induced bimodule morphism $f:\mcm_+ \to F^*\mcm_+'$ of $F$ and the induced bimodule morphism $\bar{f}:\mcn \to F^*\mcn'$ of $\bar{F}$ pullback along $\pi_\mcc$. Since $F$ and $\bar{F}$ are $A_\infty$ equivalences, $f$ and $\bar{f}$ are quasi-isomorphisms.  Together with the quasi-isomorphism $\pi_{\mcc'}^*\Psi(\sigma')$, we see from Equation~\eqref{eq:hoch-cat-htpy} that
\begin{equation}
\label{eq:mid-htpy}
    \pi_{\mcc}^*\Psi(\sigma) = \pi_{\mcc}^*\Psi(\bar{F}^*\sigma') \simeq \bar{f}^\vee(\pi_{\mcc}^*\bar{F}^*\Psi(\sigma'))\bar{f} = \bar{f}^\vee(F^*\pi_{\mcc'}^*\Psi(\sigma'))\bar{f}.
\end{equation}
Thus, $\pi_{\mcc}^*\Psi(\sigma)$ is a quasi-isomorphism.

To see that $\zeta$ is a quasi-isomorphism, note from the commutativity of diagram~\eqref{eq:functor-invariance} that
\begin{equation*}
    \bar{f}\pi_\mcn = (F^*\pi_{\mcn'}) f.
\end{equation*}
Since $f$ and $\bar{f}$ are quasi-isomorphisms, it follows from the long exact sequences of the cones that the induced morphism
\begin{equation*}
    \hat{f}:\cone(\pi_\mcn) \to F^*\cone(\pi_{\mcn'}),
\end{equation*}
is a quasi-isomorphism, and so is its linear dual $\hat{f}^\vee:F^*\cocone(\pi_{\mcn'}) \to \cocone(\pi_\mcn)$. Thus, it suffices to show that
\begin{equation}
\label{eq:zeta-htpy}
    \zeta \simeq \hat{f} F^*(\zeta') f,
\end{equation}
as the right hand side is known to be a composition of quasi-isomorphisms.

To prove Equation~\eqref{eq:zeta-htpy}, we use the explicit chain models $\zeta = (\pi_{\mcc}^*\Psi(\sigma)\pi_{\mcc},\Psi(\phi))$ and $\zeta' = (\pi_{\mcc'}^*\Psi(\sigma')\pi_{\mcc'},\Psi(\phi'))$. Equation~\eqref{eq:zeta-htpy} then comes down to the homotopy in Equation~\eqref{eq:mid-htpy}, which we have already shown, and a homotopy
\begin{equation*}
    \Psi(\phi) \simeq f^\vee(F^* \Psi(\phi')) f + \pi_\mcn^\vee K \pi_\mcn,
\end{equation*}
where $K$ is the homotopy witnessing Equation~\eqref{eq:mid-htpy}.

To establish the homotopy, we apply Equation~\eqref{eq:hoch-cat-htpy} to $\phi'$ and obtain
\begin{equation}
\label{eq:null-htpy}
    \Psi(\phi) + f^\vee(F^* \Psi(\phi')) f = \delta H_F(\phi') + H_F(\pi'_{\mcc'} \sigma'),
\end{equation}
where we used $\phi = F^*\phi'$ and $b^*\phi' = \pi'_{\mcc'} \sigma'$. We claim that $H_F(\pi'_{\mcc'} \sigma') = \pi_\mcn^\vee K \pi_\mcn$, thus completing the proof. To verify, we apply Equation~\eqref{eq:hoch-cat-htpy-composite} to $\pi_{\mcc'}F$ and obtain
\begin{equation}
    H_{\pi_{\mcc'}F}(\phi) = H_F(\pi_{\mcc'}^*\phi)+f^\vee(F^*H_{\pi_{\mcc'}}(\phi))f.
\end{equation}
Since $\pi_{\mcc'}$ is naïve as a functor, it follows from the construction of the naturality homotopy that $H_{\pi_{\mcc'}} = 0$ and $H_{\pi_{\mcc'}F}(\phi) = H_F(\pi_{\mcc'}^*\phi)$. Similarly, applying Equation~\eqref{eq:hoch-cat-htpy-composite} to $\bar{F}\pi_{\mcc}$ yields $H_{\bar{F}\pi_{\mcc}}(\phi) = \pi_{\mcn}^\vee(\pi_{\mcc}^*H_{\bar{F}}(\phi))\pi_{\mcn} = \pi_{\mcn}^\vee K\pi_\mcn$. Since $\pi_{\mcc'}F = \bar{F}\pi_{\mcc}$, we have $H_F(\pi'_{\mcc'} \sigma') = \pi_\mcn^\vee K \pi_\mcn$ as desired.
\end{proof}

\subsection{Relationship to Linear Duality}
\label{ssec:wrcy-linear}

We end the paper by discussing the relationship between the constructions of main theorem and the original linear duality of Theorem~\ref{thm:lin-duality}.  First, we return to Remark~\ref{rem:reduce1} now that all of the categories and bimodules have been defined and the commutative diagrams underlying the main theorem have been explained.  To see that Theorem~\ref{thm:main} truly is a generalization of the duality long exact sequence of \cite{high-d-duality}, fix an augmentation $\aug$ and consider Diagram~\eqref{eq:wrcy} at the linear level. Use Proposition~\ref{prop:aug-submodule}, Definition~\ref{defn:circle-bimod}, and Proposition~\ref{prop:aug-bimod+} to translate the bimodule vector spaces to the morphism spaces of $\Aug_\pm(\leg)$ and $\mcc$.  In the derived category, the top line of the diagram becomes
\[
\begin{tikzcd}[column sep=scriptsize]
    \cdots \ar[r] & H^{k}\Hom_+(\aug,\aug) \ar[r] & H^{k}(\leg) \ar[r] & H^{k+1}\Hom_-(\aug,\aug) \ar[r] & \cdots
\end{tikzcd}
\]
Use the leftmost vertical quasi-isomorphism in Diagram~\eqref{eq:wrcy} --- that is, the quasi-isomorphism from Proposition~\ref{prop:eta-quasi-iso} --- to replace $H^{k}\Hom_+(\aug,\aug)$ with $H^{k-2}\Hom_-^\vee(\aug,\aug)$. Finally, as noted in \cite[\S5.2]{nrssz:aug-sheaf}, we may translate $H^k\Hom_-(\aug,\aug)$ to $LCH^{k-1}(\aug)$ and $H^{k-2}\Hom_-^\vee(\aug, \aug)$ to $LCH_{-k+1}(\aug)$. Hence, we obtain the original duality long exact sequence described in Theorem~\ref{thm:lin-duality}.

To explicitly interpret  the weak relative CY structure we obtained on the functor $\pi_\mcc:\Aug_+(\leg) \to \mcc(\leg)$, recall from Definition~\ref{defn:rel-cy} that it is given by a relative Hochschild cohomology class $([\phi],[\sigma]) \in \HH^*(\pi_\mcc)[1]$, with $[\phi] \in \HH^*(\Aug_+(\leg))[2]$ and $[\sigma] \in \HH^*(\mcc(\leg))[1]$. Following the proof of Proposition~\ref{prop:alg-gadget}, we see that $[\sigma]$ is the preimage of the unit $[y^+] \in \HH^*(\mcc(\leg),\mcc(\leg))$ under the Hochschild cohomology isomorphism induced by $\theta:\mcc^\vee(\leg)[-1] \to \mcc(\leg)$. Since $\theta$ acts as the Poincaré duality map of $S^1$, we can interpret $[\sigma]$ as the \dfn{fundamental class} of $S^1$ in $\HH^*(\mcc(\leg))$. Specifically, we have $[\sigma^0] = [x^\vee] \in H_1(\mcc(\leg)^\vee) \simeq H_1(S^1)$, and $[\sigma]$ generalizes the standard fundamental class of $S^1$.

Similarly, we can write $[\phi] = [(\lambda, s)]$, where $[\lambda] \in \HH^*(\Aug_+(\leg),\mcm_-^\vee)$ is the preimage of the unit $[y^+] \in \HH^*(\Aug_+(\leg),\mcm_+(\leg))$ under the isomorphism induced by $\eta:\mcm_-^\vee[-2] \to \mcm_+$, and $s \in \CH^*(\Aug_+(\leg),\mcn^\vee)$ is a bounding cochain between $(\rho^\vee)^*\lambda$ and $\pi_\mcc^\sharp \sigma$. Recall that $\eta$ extends the linear duality map in \cite{duality}. In that context, the \dfn{fundamental class} of the linearized Legendrian contact homology $LCH_1(\aug) \simeq H^{-2}\Hom_-^\vee(\aug, \aug)$ is defined as the preimage of $[y^+]$ under the isomorphism $\eta^{0|0}_*:H^{-2}\Hom_-^\vee(\aug, \aug) \to H_0\Hom_+(\aug, \aug)$. Therefore, $[\lambda]$ is the extension of this fundamental class to Hochschild cohomology.

\begin{ex}
    Consider the standard max-tb unknot $\leg$ with one Reeb chord $a$ and one base point $t$. The DGA is determined by $\df a = 1+t^{-1}$ and it admits a unique augmentation $\aug$ over $\ff_2$, where $\aug(a)=0$, and $\aug(t)=1$. The structure maps for the simply perturbed $\Aug(\leg)$ are
    \begin{equation*}
        m_+^k(x^+,\ldots,x^+) = a^+,
    \end{equation*}
    for $k \ge 1$, in addition to the operations dictated by the strict unit $y^+$. A representative of the fundamental class $[\lambda, s] \in \HH(\Aug_+(\leg))$ is given by
     \begin{align*}
     \lambda^0&=a^\vee,\quad \lambda^1(x^+) = a^\vee,\\
     s^2(a^+,a^+)&=y^\vee,\quad s^3(a^+,x^+,a^+)=y^\vee.
    \end{align*}
    One can check that $b^*\lambda = 0$ and $b^*s = (\rho^\vee)^*\lambda+\pi_\mcc^\sharp \sigma$. Moreover, $[\eta^*\lambda]=[y^+]$. Indeed, $\eta^{0|0}(a^\vee)=y^+$, so the cocycle $\eta^*\lambda-y^+$ vanishes in arity zero. Since $m_1(x^+)=a^+$, for $A = \Hom_+(\aug,\aug)$, the reduced complex $\overline{A}=A/\langle y^+\rangle$ is contractible. Thus, $\Hom(\overline A[1]^{\otimes k},A)$ is contractible for every $k \ge 1$, with differential induced by $m_1$. Solving successively in increasing arity therefore gives a Hochschild cochain $h$ such that $b^*h=\eta^*\lambda-y^+$.
\end{ex}

In \cite[\S5.2]{duality}, a homotopy inverse of the linear duality map $\eta^{0|0}$ is given by the ``cap product" with the fundamental class $[\lambda^0]$. In our language, we can recast this map as $\xi_0:\Hom_+(\aug, \aug) \to \Hom_-^\vee(\aug, \aug)[-2]$, defined by
\begin{equation*}
    \xi_0(\ul{w}^+) := n_{\mcm_-^\vee}^{0|1}(\ul{\lambda}^0,w^+).
\end{equation*}
In fact, $\xi_0$ is the linear component of an $A_\infty$ morphism $\xi:\mcm_+ \to \mcm_-^\vee[-2]$ given by $\xi := \Psi(\lambda)$, where $\Psi:\CH^*(\Aug_+(\leg),\mcn^\vee) \to \hom_{\Aug_+(\leg)}(\mcm_+,\mcn^\vee)$ is defined in Proposition~\ref{prop:hoch-a-infty}. Explicitly, $\xi$ acts as the cap product with the fundamental class $[\lambda]$, given by the formula
\begin{equation*}
    \xi(\cev{b},\ul{w}^+,\vec{a}) = \sum_{i,j} n_{\mcm_-^\vee}^{r-i-j|s}(b_r^+,\ldots,b_{i+1}^+,\ul{\lambda}(b_i,\ldots,b_j^+),b_{j-1}^+,\ldots,b_1^+,w^+,\vec{a})
\end{equation*}
It follows from Proposition~\ref{prop:hoch-a-infty} that $\xi$ is the $A_\infty$ homotopy inverse of $\eta$. This formula (conjecturally) recovers the explicit homotopy inverse of $\eta$ constructed in \cite{chen:lsft-dual} when $\lambda = \lambda^0$. Similarly, the homotopy inverse $\kappa$ of $\theta$ can be defined by $\kappa := \Psi(\sigma)$.


\bibliographystyle{amsplain} 
\bibliography{main}

\end{document}